\documentclass[draft,reqno]{amsart}

\usepackage{amsmath,amssymb,amsthm,enumerate} 
\usepackage{amsfonts} 
\usepackage{color}
\usepackage{dsfont} 

\theoremstyle{plain}
\newtheorem{thm}{Theorem}[section]

\newtheorem{lem}[thm]{Lemma}
\newtheorem{prp}[thm]{Proposition}
\newtheorem{rmk}[thm]{Remark}

\DeclareMathOperator{\dv}{div}

\DeclareMathOperator{\spp}{supp}
\DeclareMathOperator{\Dv}{Div}

\newcommand{\grv}{\gamma_a}
\newcommand{\intd}{\,d}
\newcommand{\jmp}[1]{\ensuremath{[\![#1]\!]}}

\newcommand{\pd}{\partial}
\newcommand{\wht}[1]{\widehat{#1}}
\newcommand{\wtd}[1]{\widetilde{#1}}

\newcommand{\Ba}{\mathbf{a}}
\newcommand{\Bb}{\mathbf{b}}

\newcommand{\Be}{\mathbf{e}}
\newcommand{\Bf}{\mathbf{f}}
\newcommand{\Bg}{\mathbf{g}}
\newcommand{\Bh}{\mathbf{h}}

\newcommand{\Bu}{\mathbf{u}}
\newcommand{\Bv}{\mathbf{v}}
\newcommand{\Bw}{\mathbf{w}}

\newcommand{\BC}{\mathbf{C}}
\newcommand{\BD}{\mathbf{D}}

\newcommand{\BF}{\mathbf{F}}

\newcommand{\BI}{\mathbf{I}}

\newcommand{\BL}{\mathbf{L}}
\newcommand{\BM}{\mathbf{M}}
\newcommand{\BN}{\mathbf{N}}

\newcommand{\BR}{\mathbf{R}}

\newcommand{\BU}{\mathbf{U}}

\newcommand{\CB}{\mathcal{B}}

\newcommand{\CD}{\mathcal{D}}
\newcommand{\CE}{\mathcal{E}}
\newcommand{\CF}{\mathcal{F}}
\newcommand{\CG}{\mathcal{G}}
\newcommand{\CH}{\mathcal{H}}
\newcommand{\CI}{\mathcal{I}}
\newcommand{\CJ}{\mathcal{J}}

\newcommand{\CL}{\mathcal{L}}
\newcommand{\CM}{\mathcal{M}}

\newcommand{\CS}{\mathcal{S}}

\newcommand{\CU}{\mathcal{U}}

\newcommand{\Fa}{\mathfrak{a}}
\newcommand{\Fb}{\mathfrak{b}}

\newcommand{\Fp}{\mathfrak{p}}
\newcommand{\Fq}{\mathfrak{q}}
\newcommand{\Fr}{\mathfrak{r}}

\newcommand{\SST}{\mathsf{T}}

\numberwithin{equation}{section} 
\allowdisplaybreaks[4]

\begin{document}
\title[Linearized two-phase Navier-Stokes equations]
{Time-decay estimates for linearized two-phase Navier-Stokes equations with surface tension and gravity}

\author[Hirokazu Saito]{Hirokazu Saito}
\address{Department of Mathematics, The University of Electro-Communications,
5-1 Chofugaoka 1-chome, Chofu, Tokyo 182-8585, Japan}
\email{hsaito@uec.ac.jp}


\thanks{This research was partly supported by JSPS KAKENHI Grant Number JP17K14224.}

\begin{abstract}
The aim of this paper is to show time-decay estimates of solutions to
linearized two-phase Navier-Stokes equations with surface tension and gravity.
The original two-phase Navier-Stokes equations describe the two-phase incompressible viscous flow with a sharp interface
that is close to the hyperplane $x_N=0$ in the $N$-dimensional Euclidean space, $N \geq 2$.
It is well-known that the Rayleigh-Taylor instability occurs
when the upper fluid is heavier than the lower one,
while this paper assumes that the lower fluid is heavier than the upper one
and proves time-decay estimates of $L_p-L_q$ type for the linearized equations.
Our approach is based on solution formulas, given by Shibata and Shimizu (2011),
for a resolvent problem associated with the linearized equations.
\end{abstract}

\maketitle

\section{Introduction}\label{sec:intro}

Let us consider the motion of two immiscible, viscous, incompressible capillary fluids,
${\it fluid}_+$ and ${\it fluid}_-$, in the $N$-dimensional Euclidean space $\BR^N$ for $N\geq 2$.
Here ${\rm fluid}_+$ and ${\rm fluid}_-$ occupy $\Omega_+(t)$ and $\Omega_-(t)$, respectively, which are given by
\begin{equation*}
\Omega_\pm(t)=\{(x',x_N) : x'=(x_1,\dots,x_{N-1})\in\BR^{N-1}, \pm(x_N-H(x',t))>0\}
\end{equation*}
for time $t>0$ and
the so-called height function\footnote{The height function is unknown. This needs to be determined as part of the problem.}
$H=H(x',t)$.
The fluids are thus separated by the interface 
\begin{equation*}
\Gamma(t)=\{(x',x_N) : x'=(x_1,\dots,x_{N-1})\in\BR^{N-1}, x_N=H(x',t)\}.
\end{equation*}
We denote the density of ${\rm fluid}_\pm$ by $\rho_\pm$, 
while the viscosity coefficient of ${\rm fluid}_\pm$ by $\mu_\pm$.
Suppose that $\rho_\pm$ and $\mu_\pm$ are positive constants throughout this paper.
That motion of  two fluids is governed by the two-phase Navier-Stokes equations
where surface tension is included on the interface. In addition, we allow for gravity to act on the fluids.
The two-phase Navier-Stokes equations was studied by Pr$\ddot{{\rm u}}$ss and Simonett \cite{PS10},
and they proved that the Rayleigh-Taylor instability occurs in an $L_p$-setting
when the upper fluid is heavier than the lower one, i.e., $\rho_+>\rho_-$.
In the present paper, we assume that the lower fluid is heavier than the upper one, i.e., $\rho_->\rho_+$,
and show time-decay estimates of $L_p-L_q$ type for some linearized system
as the first step in proving global existence results for the two-phase Navier-Stokes equations when $\rho_->\rho_+$.

Let us define $\dot\BR^N=\BR_+^N\cup\BR_-^N$ for
\begin{equation*}
\BR_\pm^N = \{(x',x_N)  : x'=(x_1,\dots,x_{N-1})\in\BR^{N-1}, \pm x_N>0\}.
\end{equation*}
This paper is concerned with the following linearized system of the two-phase Navier-Stokes equations:
\begin{equation}\label{eq:main}
\left\{\begin{aligned}
\pd_t H-U_N|_{x_N=0} &= 0 \quad \text{on $\BR^{N-1}\times (0,\infty)$,} \\
\rho\pd_t\BU -\Dv(\mu\BD(\BU)-P\BI) &=0 \quad \text{in $\dot\BR^N\times(0,\infty)$,} \\
\dv\BU&=0 \quad \text{in $\dot\BR^N\times(0,\infty)$,} \\
-\jmp{(\mu\BD(\BU)-P\BI)\Be_N}+\left(\omega-\sigma\Delta'\right)H\Be_N&=0 \quad \text{on $\BR^{N-1}\times(0,\infty)$,} \\
\jmp{\BU}&=0 \quad \text{on $\BR^{N-1}\times(0,\infty)$,} \\
H(x',0) =d(x') \ (x'\in\BR^{N-1}),  \quad
&\BU(x,0) = \Bf(x) \ (x\in \dot\BR^N),
\end{aligned}\right.
\end{equation}
where $\sigma$ is a positive constant called the surface tension coefficient and
one has set for the indicator function $\mathds{1}_A$ of $A\subset\BR^N$
\begin{equation*}
\rho=\rho_+\mathds{1}_{\BR_+^N}+\rho_-\mathds{1}_{\BR_-^N}, \quad 
\mu=\mu_+\mathds{1}_{\BR_+^N}+\mu_-\mathds{1}_{\BR_-^N}.
\end{equation*}

Here $\BU=\BU(x,t)= (U_1(x,t),\dots,U_N(x,t))^\SST\footnote{$\BM^\SST$ denotes the transpose of $\BM$.}$ 
and $P=P(x,t)$ respectively denote the velocity field of the fluid and the pressure field of the fluid
at position $x\in\dot\BR^N$ and time $t>0$,
while $d=d(x')$ and $\Bf=\Bf(x)=(f_1(x),\dots,f_N(x))^\SST$ are given initial data.
Note that $\Be_N=(0,\dots,0,1)^\SST$ and $\BI$ is the $N\times N$ identity matrix.
Let $\pd_j=\pd/\pd x_j$ for $j=1,\dots,N$. Then
\begin{equation*}
\dv\BU=\sum_{j=1}^N\pd_j U_j,   \quad 
\Delta' H=\sum_{j=1}^{N-1}\pd_j^2 H,
\end{equation*}
while $\BD(\BU)$ is an $N\times N$ matrix whose $(i,j)$ element is given by $\pd_i U_j+\pd_j U_i$.
In addition, for matrix-valued functions $\BM=(M_{ij}(x))$, 
\begin{equation*}
\Dv\BM=\left(\sum_{j=1}^N\pd_j M_{1j},\dots,\sum_{j=1}^N\pd_j M_{N j}\right)^\SST.
\end{equation*}

Let $f=f(x)$ be a function defined on $\dot\BR^N$.
Then $\jmp{f}=\jmp{f}(x')$ denotes the jump of the quantity $f$ across the interface 
$\BR_0^N=\{(x',x_N) : x'\in\BR^{N-1},x_N=0\}$, that is,
\begin{equation*}
\jmp{f}=\jmp{f}(x')=f(x',0\,+)-f(x',0\,-),
\end{equation*}
where $f(x',0\,\pm)=\lim_{x_N \to 0,\, \pm x_N>0}f(x',x_N)$.
Note that $\jmp{\BU}=0$ on $\BR^{N-1}$ implies $U_N|_{x_N=0}=U_N(x',0\,+)=U_N(x',0\,-)$.
For the acceleration of gravity $\grv>0$, the constant $\omega$ is given by
\begin{equation*}
\omega=-\jmp{\rho}\grv = (\rho_--\rho_+)\grv,
\end{equation*}
which is positive when $\rho_->\rho_+$.

The local well-posedness for the two-phase Navier-Stokes equations
with $\Gamma(t)$ as above was proved in Pr$\ddot{\rm u}$ss and Simonett \cite{PS10b,PS11}.
Note that the local well-posedness holds for any positive constants $\rho_\pm$,
that is, the condition $\rho_->\rho_+$ is not required.
Those results were extended to a class of non-Newtonian fluids in \cite{HS17}.
In addition, \cite{SSZ20} considered the two-phase inhomogeneous incompressible viscous flow
without surface tension when gravity is not taken into account,  
and proved the local well-posedness in general domains including the above-mentioned $\Omega_\pm(t)$. 
If $\Omega_\pm(t)$ are assumed to be layer-like domains, 
then it is known that the global well-posedness holds when $\rho_->\rho_+$.
In fact, it was shown in \cite{WTK14} in a horizontally periodic setting, and also we refer to \cite{XZ10}.

Let us recall Shibata and Shimizu \cite{SS11b}. 
They considered the following two resolvent problems:
\begin{equation}\label{res-prob:1}
\left\{\begin{aligned}
\lambda\eta-u_N|_{x_N=0} &= d && \text{on $\BR^{N-1}$,} \\
\rho\lambda \Bu -\Dv(\mu\BD(\Bu)-\Fp\BI) &=\rho\Bf && \text{in $\dot\BR^N$,} \\
\dv\Bu&=0 && \text{in $\dot\BR^N$, } \\
-\jmp{(\mu\BD(\Bu)-\Fp\BI)\Be_N}+(\omega-\sigma\Delta')\eta\Be_N&=0 && \text{on $\BR^{N-1}$, } \\
\jmp{\Bu}&=0 && \text{on $\BR^{N-1}$,}
\end{aligned}\right.
\end{equation}
and also 
\begin{equation}\label{res-prob:2}
\left\{\begin{aligned}
\rho\lambda \Bv -\Dv(\mu\BD(\Bv)-\Fq\BI) &=\rho\Bf && \text{in $\dot\BR^N$,} \\
\dv\Bv&=0 && \text{in $\dot\BR^N$, } \\
\jmp{(\mu\BD(\Bv)-\Fq\BI)\Be_N}&=0 && \text{on $\BR^{N-1}$, } \\
\jmp{\Bv}&=0 && \text{on $\BR^{N-1}$.}
\end{aligned}\right.
\end{equation}
We define the sector
\begin{equation*}
\Sigma_\varepsilon=\{\lambda\in\BC\setminus\{0\} : |\arg\lambda|\leq \pi-\varepsilon\} \quad (0<\varepsilon<\pi/2),
\end{equation*}
where $\BC$ is the set of all complex numbers.
Let $q\in(1,\infty)$ and $(d,\Bf)\in X_q:=W_q^{2-1/q}(\BR^{N-1})\times L_q(\dot\BR^N)^N$.
In \cite{SS11b}, they obtained the following results:
there exists a constant $\lambda_0(\varepsilon)\geq 1$
such that, for any $\lambda \in\Sigma_\varepsilon$ with $|\lambda|\geq \lambda_0(\varepsilon)$,
\eqref{res-prob:1} admits a unique solution
$(\eta,\Bu,\Fp)\in W_q^{3-1/q}(\BR^{N-1})\times H_q^2(\dot\BR^N)^N\times \wht H_q^1(\dot\BR^N)$, which satisfies
\begin{equation}\label{res-est:1}
\|(\lambda\eta,\nabla'\eta)\|_{W_q^{2-1/q}(\BR^{N-1})}
+\|(\lambda\Bu,\lambda^{1/2}\nabla \Bu,\nabla^2\Bu,\nabla\Fp)\|_{L_q(\dot\BR^N)}
\leq C_{N,q,\varepsilon,\lambda_0(\varepsilon)}\|(d,\Bf)\|_{X_q},
\end{equation}
where
\begin{alignat*}{2}
\nabla'\eta&=(\pd_1\eta,\dots,\pd_{N-1}\eta)^\SST, \quad
 &\nabla\Fp&=(\pd_1\Fp,\dots,\pd_{N-1}\Fp,\pd_N\Fp)^\SST, \\
\nabla\Bu&=\{\pd_j u_k: j,k=1,\dots,N\}, \quad
&\nabla^2\Bu&=\{\pd_j\pd_k u_l : j,k,l=1,\dots,N\};
\end{alignat*}
for any $\lambda\in\Sigma_\varepsilon$, \eqref{res-prob:2} admits a unique solution 
$(\Bv,\Fq)\in H_q^2(\dot\BR^N)^N\times \wht H_q^1(\dot\BR^N)$, which satisfies
\begin{equation}\label{res-est:2}
\|(\lambda\Bv,\lambda^{1/2}\nabla\Bv,\nabla^2\Bv,\nabla\Fq)\|_{L_q(\dot\BR^N)}
\leq C_{N,q,\varepsilon}\|\Bf\|_{L_q(\dot\BR^N)}.
\end{equation}
These results hold for any $\rho_\pm>0$
and play a key role in proving time-decay estimates of solutions for \eqref{eq:main}
in the present paper.

This paper is organized as follows:
The next section introduces the notation used throughout this paper
and states the main results of this paper, that is,
time-decay estimates of solutions for \eqref{eq:main}.
Section 3 gives the representation formulas of solutions for \eqref{eq:main}
by means of the partial Fourier transform with respect to $x'=(x_1,\dots,x_{N-1})$ and its inverse transform.
Section 4 analyzes the boundary symbol appearing in the representation formulas given in Section 3.
Section 5 shows our main result stated in Section 2 for a low frequency part.
Section 6 shows our main result stated in Section 2 for a high frequency part.

\section{Notation and main results}
\subsection{Notation}
First, we introduce function spaces.
Let $X$ be a Banach space.
Then $X^m$, $m\geq 2$, stands for the $m$-product space of $X$,
while the norm of $X^m$ is usually denoted by $\|\cdot\|_X$ instead of $\|\cdot\|_{X^m}$ for the sake of simplicity.
For another Banach space $Y$, we set $\|u\|_{X\cap Y}=\|u\|_X+\|u\|_Y$.
Let $\BN$ be the set of all natural numbers and $\BN_0=\BN\cup\{0\}$.
Let $p\geq 1$ or $p=\infty$. 
For an open set $G\subset \BR^M$, $M\geq 1$,
the Lebesgue spaces on $G$ are denoted by $L_p(G)$ with norm $\|\cdot\|_{L_p(G)}$,
while the Sobolev spaces on $G$ are denoted by $H_p^n(G)$, $n\in\BN$, with norm $\|\cdot\|_{H_p^n(G)}$.
Let $H_p^0(G)=L_p(G)$.
In addition, $C_0^\infty(G)$ is the set of all functions in $C^\infty(G)$ whose supports are compact and contained in $G$,
and $C^\infty(I,X)$ is the set of all $C^\infty$ functions on an interval $I\subset\BR$ with value $X$.
For any multi-index $\alpha=(\alpha_1,\dots,\alpha_M)\in\BN_0^M$, 
\begin{equation*}
\pd^\alpha u=\pd_x^\alpha  u =\frac{\pd^{|\alpha|}u(x)}{\pd x_1^{\alpha_1}\dots \pd x_M^{\alpha_M}}
\quad \text{with $|\alpha|=\alpha_1+\dots+\alpha_M$.}
\end{equation*}
Let $q\in(1,\infty)$ and $\wht H_q^1(G)=\{u\in L_{q,\rm loc}(\overline{G}) : \pd^\alpha u \in L_q(G) \text{ for $|\alpha|=1$}\}$.
Let $s\in[0,\infty)\setminus\BN$ and $[s]$ be the largest integer less than s.
The Sobolev-Slobodeckij spaces on $\BR^{N-1}$ are defined by
\begin{align*}
W_q^s(\BR^{N-1})&=\{u\in L_q(\BR^{N-1}) : \|u\|_{W_q^s(\BR^{N-1})}<\infty\}, \\
\|u\|_{W_q^s(\BR^{N-1})}
&=\|u\|_{H_q^{[s]}(\BR^{N-1})} \\
&+\sum_{|\alpha|=[s]}\left(\int_{\BR^{N-1}}\int_{\BR^{N-1}}
\frac{|\pd^\alpha u(x)-\pd^\alpha u(y)|^q}{|x-y|^{N+(s-[s])q}}\intd xdy\right)^{\frac{1}{q}}.
\end{align*}
Let us define a solenoidal space by 
\begin{equation*}
J_q(\dot\BR^N)=\{\Bu\in L_q(\dot\BR^N)^N : (\Bu,\nabla\varphi)_{\dot\BR^N}=0 \text{ for any $\varphi\in\wht H_{q'}^1(\BR^N)$}\},
\end{equation*}
where $q'=q/(q-1)$ and 
\begin{equation*}
(\Ba,\Bb)_{\dot\BR^N}=\int_{\dot\BR^N}\Ba(x)\cdot \Bb(x)\intd x = \sum_{j=1}^N\int_{\dot\BR^N}a_j(x)b_j(x)\intd x.
\end{equation*}
Furthermore, we set
\begin{alignat*}{2}
X_q&=W_q^{2-1/q}(\BR^{N-1})\times L_q(\dot\BR^N)^N, 
\quad 
&JX_q&=W_q^{2-1/q}(\BR^{N-1})\times J_q(\dot\BR^N), \\
Y_q&=L_q(\BR^{N-1})\times L_q(\dot\BR^N)^N, 
\quad & JY_q&=L_q(\BR^{N-1})\times J_q(\dot\BR^N),
\end{alignat*}
and also $Y_1=L_1(\BR^{N-1})\times L_1(\dot\BR^N)^N$.

Next, we define the partial Fourier transform with respect to $x'=(x_1,\dots,x_{N-1})$ by
\begin{equation*}
\wht u(\xi',x_N)=\int_{\BR^{N-1}}e^{-ix'\cdot\xi'}u(x',x_N)\intd x',
\end{equation*}
where $i=\sqrt{-1}$. Its inverse transform is also defined by 
\begin{equation*}
\CF_{\xi'}^{-1}[v(\xi',x_N)](x')=\frac{1}{(2\pi)^{N-1}}\int_{\BR^{N-1}}e^{ix'\cdot\xi'}v(\xi',x_N)\intd \xi'.
\end{equation*}

Finally, we introduce some constants and symbols.
Let 
\begin{equation}\label{dfn:theta}
\theta_j={\rm Tan}^{-1}\left(\frac{j}{16}\right) \quad \text{for $j=1,2$,}
\quad \lambda_1=\lambda_0(\theta_1),
\end{equation}
where $\lambda_0(\varepsilon)$ is given in \eqref{res-est:1}.
In addition, we set
\begin{equation}\label{dfn:al-beta}
\alpha = \frac{\omega}{\rho_++\rho_-}=\frac{(\rho_--\rho_+)\gamma_a}{\rho_++\rho_-}, \quad
\beta=\frac{\sqrt{\rho_+\mu_+}\sqrt{\rho_-\mu_-}}{(\rho_++\rho_-)(\sqrt{\rho_+\mu_+}+\sqrt{\rho_-\mu_-})}.
\end{equation}
The integral path $\Gamma_0$ is defined by 
\begin{align}
\Gamma_0&=\Gamma_0^+\cup\Gamma_0^-, \quad
\Gamma_0^+ 
= \{\lambda \in\BC : \lambda=\frac{2\lambda_1}{\sin\theta_1} + s e^{i(\pi-\theta_1)}, s:0\to\infty\}, \notag \\
\Gamma_0^- 
&=\{\lambda \in\BC : \lambda=\frac{2\lambda_1}{\sin\theta_1} + s e^{-i(\pi-\theta_1)}, s:\infty\to 0\}.
\label{dfn:gamma_0}
\end{align}

\subsection{Main results}
We first introduce the existence of solution operators for \eqref{eq:main}.
This immediately follows from the resolvent estimate \eqref{res-est:1} and the standard theory of operator semigroups.

\begin{thm}\label{thm:main1}
Let $q\in(1,\infty)$ and $\rho_\pm$ be any positive constants.
Let
\begin{equation*}
T(t)(d,\Bf)=\frac{1}{2\pi i}\int_{\Gamma_0} e^{\lambda t}(\eta(x',\lambda),\Bu(x,\lambda))\,d\lambda,
\end{equation*}
with the solution $(\eta,\Bu)=(\eta(x',\lambda),\Bu(x,\lambda))$  to \eqref{res-prob:1}.
Then the following assertions hold.
\begin{enumerate}[$(1)$]
\item
For any $(d,\Bf)\in X_q$,
$T(t)(d,\Bf) \in C^\infty((0,\infty),W_q^{3-1/q}(\BR^{N-1})\times H_q^2(\dot\BR^N)^N)$. 
\item
$\{T(t)\}_{t\geq 0}$ is an analytic $C_0$-semigroup on $JX_q$.
\item
For any $(d,\Bf)\in JX_q$, $(H,\BU)=T(t)(d,\Bf)$ is  a unique solution to \eqref{eq:main} with some pressure $P\in\wht H_q^1(\dot\BR^N)$.
\end{enumerate}
\end{thm}

Let us define projections $P_1$, $P_2$ by 
\begin{equation*}
P_1:X_q\ni (a, \Bb)\mapsto  a\in W_q^{2-1/q}(\BR^{N-1}), \quad
P_2:X_q\ni(a,\Bb)\mapsto \Bb\in L_q(\dot\BR^N)^N.
\end{equation*}
We set for $(d,\Bf)\in X_q$
\begin{align}
\CH(t)(d,\Bf)&=P_1 T(t)(d,\Bf)=\frac{1}{2\pi i}\int_{\Gamma_0} e^{\lambda t}\eta(x',\lambda)\,d\lambda, \notag \\ 
\CU(t)(d,\Bf)&=P_2 T(t)(d,\Bf)=\frac{1}{2\pi i}\int_{\Gamma_0} e^{\lambda t}\Bu(x,\lambda)\,d\lambda. \label{ope-U}
\end{align}

One now decomposes the solution $(\Bu,\Fp)$ of \eqref{res-prob:1} into 
a solution of parabolic system and a solution of hyperbolic-parabolic coupled system as follows:
\begin{align}
&\left\{\begin{aligned}\label{eq:para}
\rho\lambda\Bu^P-\Dv(\mu\BD(\Bu^P)-\Fp^P\BI)&=\rho\Bf && \text{in $\dot\BR^N$,} \\
\dv\Bu^P&=0 && \text{in $\dot \BR^N$,} \\
\jmp{(\mu\BD(\Bu^P)-\Fp^P\BI)\Be_N}&=0 && \text{on $\BR^{N-1}$,} \\
\jmp{\Bu^P}&=0  && \text{on $\BR^{N-1}$,} \\
\end{aligned}\right. \\
&\left\{\begin{aligned}\label{eq:hype}
\lambda\eta-u_N^H|_{x_N=0} &= d+u_N^P|_{x_N=0} && \text{on $\BR^{N-1}$,} \\
\rho\lambda \Bu^H -\Dv(\mu\BD(\Bu^H)-\Fp^H\BI) &=0 && \text{in $\dot\BR^N$,} \\
\dv\Bu^H&=0 && \text{in $\dot\BR^N$, } \\
-\jmp{(\mu\BD(\Bu^H)-\Fp^H\BI)\Be_N}+(\omega-\sigma\Delta')\eta\Be_N&=0 && \text{on $\BR^{N-1}$, } \\
\jmp{\Bu^H}&=0 && \text{on $\BR^{N-1}$.}
\end{aligned}\right.
\end{align}
It then holds that $\Bu=\Bu^P+\Bu^H$ and $\Fp=\Fp^P+\Fp^H$.
In addition, we set
\begin{equation*}
\CU^P(t)\Bf=\frac{1}{2\pi i}\int_{\Gamma_0} e^{\lambda t}\Bu^P(x,\lambda)\,d\lambda, \quad
\CU^H(t)(d,\Bf)=\frac{1}{2\pi i}\int_{\Gamma_0} e^{\lambda t}\Bu^H(x,\lambda)\,d\lambda,
\end{equation*}
which implies
\begin{equation}\label{decomp:1-P}
\CU(t)(d,\Bf)=\CU^P(t)\Bf+\CU^H(t)(d,\Bf) \quad \text{for $(d,\Bf)\in X_q$.}
\end{equation}

For the parabolic part, we have the following theorem 
by the resolvent estimate \eqref{res-est:2} and the standard theory of operator semigroups.

\begin{thm}\label{thm:main2}
Let $r\in(1,\infty)$ and $\rho_\pm$ be any positive constants.
Then the following assertions hold.
\begin{enumerate}[$(1)$]
\item
For any $\Bf\in L_r(\dot\BR^N)^N$, $\CU^P(t)\Bf\in C^\infty((0,\infty),H_r^2(\dot\BR^N)^N)$.
\item
$\{\CU^P(t)\}_{t\geq 0}$ is an analytic $C_0$-semigroup on $J_r(\dot\BR^N)$.
\item
Let $j\in\BN_0$ and $k=0,1,2$.
For any $t>0$ and $\Bf\in J_r(\dot\BR^N)$,
\begin{equation*}
\|\pd_t^j\nabla^k\CU^P(t)\Bf\|_{L_s(\dot\BR^N)}\leq 
C t^{-j-\frac{k}{2}-\frac{N}{2}\left(\frac{1}{r}-\frac{1}{s}\right)}\|\Bf\|_{L_r(\dot\BR^N)},
\end{equation*}
where $r\leq s \leq \infty$ when $k=0,1$ and $r \leq s <\infty$ when $k=2$.
Here $C$ is a positive constant independent of $t$ and $\Bf$.
\end{enumerate}
\end{thm}

To complete time-decay estimates for $\CH(t)$ and $\CU(t)$, we further decompose $(\eta,\Bu^H,\Fp^H)$
satisfying \eqref{eq:hype}
as follows: for $z^1=d$ and $z^2=u_N^P|_{x_N=0}$,
let $(\eta^k,\Bu^k)$ be the solution to 
\begin{equation}\label{eq:hype-2}
\left\{\begin{aligned}
\lambda\eta^k-u_N^{k}|_{x_N=0} &= z^k && \text{on $\BR^{N-1}$,} \\
\rho\lambda \Bu^k -\Dv(\mu\BD(\Bu^k)-\Fp^k\BI) &=0 && \text{in $\dot\BR^N$,} \\
\dv\Bu^k&=0 && \text{in $\dot\BR^N$, } \\
-\jmp{(\mu\BD(\Bu^k)-\Fp^k\BI)\Be_N}+(\omega-\sigma\Delta')\eta^k\Be_N&=0 && \text{on $\BR^{N-1}$, } \\
\jmp{\Bu^k}&=0 && \text{on $\BR^{N-1}$.}
\end{aligned}\right.
\end{equation}
It then holds that $\eta=\eta^1+\eta^2$, $\Bu^H=\Bu^1+\Bu^2$, and $\Fp^H=\Fp^1+\Fp^2$.
Let $\varphi=\varphi(\xi')$ be a function in $C^\infty(\BR^{N-1})$ and satisfy
$0\leq \varphi\leq 1$ with
\begin{equation*}
\varphi(\xi')=\left\{\begin{aligned}
&1 && (|\xi'|\leq 1), \\
&0 && (|\xi'|\geq 2).
\end{aligned}\right.
\end{equation*}
In addition, we set $\varphi_{A_0}(\xi')=\varphi(\xi'/A_0)$ and $\varphi_{A_\infty}(\xi')=1-\varphi(\xi'/A_\infty)$
for positive constants $A_0\in(0,1)$ and $A_\infty\geq 2$.
Let $\varphi_{[A_0,A_\infty]}(\xi')=1-\varphi_{A_0}(\xi')-\varphi_{A_\infty}(\xi')$.
Together with these cut-off functions, 
we define for $Z\in\{A_0,A_\infty,[A_0,A_\infty]\}$ and for an integral path $\Gamma$
\begin{align}
\wht \CH_{Z}^1(t;\Gamma)d&=
\frac{\varphi_Z(\xi')}{2\pi i}\int_{\Gamma} e^{\lambda t}\wht \eta^{\,1}(\xi',\lambda)\,d\lambda, \notag \\
\wht \CH_{Z}^2(t;\Gamma)\Bf&=
\frac{\varphi_Z(\xi')}{2\pi i}\int_{\Gamma} e^{\lambda t}\wht \eta^{\,2}(\xi',\lambda)\,d\lambda, \notag \\
\wht \CU_{Z}^{1}(t;\Gamma)d&=
\frac{\varphi_Z(\xi')}{2\pi i}\int_{\Gamma} e^{\lambda t}\wht \Bu^{1}(\xi',x_N,\lambda)\,d\lambda, \notag \\
\wht \CU_{Z}^{2}(t;\Gamma)\Bf&=\frac{\varphi_Z(\xi')}{2\pi i}\int_{\Gamma} e^{\lambda t}\wht \Bu^{2}(\xi',x_N,\lambda)\,d\lambda.
\label{hat-H-U}
\end{align}
Furthermore, we set for $\CS\in\{\CH,\CU\}$ and $\Gamma_0$ given by \eqref{dfn:gamma_0}
\begin{equation}\label{eq:sol-main}
\CS_Z^1(t)d=\CF_{\xi'}^{-1}[\wht\CS_Z^{\,1}(t;\Gamma_0)d](x'), \quad 
\CS_Z^2(t)\Bf=\CF_{\xi'}^{-1}[\wht\CS_Z^{\,2}(t;\Gamma_0)\Bf](x').
\end{equation}
Noting \eqref{decomp:1-P},
we see that the formulas in \eqref{ope-U} satisfy
\begin{align}
\CH(t)(d,\Bf) &= \sum_{Z\in\{A_0,A_\infty,[A_0,A_\infty]\}} 
\big(\CH_Z^1(t)d +\CH_Z^2(t)\Bf\big), \notag \\ 
\CU(t)(d,\Bf) &= 
\sum_{Z\in\{A_0,A_\infty,[A_0,A_\infty]\}}\big(\CU_Z^{1}(t)d
+ \CU_Z^{2}(t)\Bf\big)+\CU^P(t)\Bf. \label{U:decomp}
\end{align}

The following two theorems are our main results of this paper.
The first one is time-decay estimates for the low frequency part.

\begin{thm}\label{thm:main3}
Let $1\leq p <  2 \leq  q \leq \infty$ and suppose that $\rho_->\rho_+>0$.
Then there exists a constant $A_0\in(0,1)$ such that the following assertions hold.
\begin{enumerate}[$(1)$]
\item
For $\xi'\in\BR^{N-1}$ with $|\xi'|\in(0,A_0)$, 
let $\zeta_\pm = \pm i \alpha^{1/2}|\xi'|^{1/2}-\sqrt{2}\alpha^{1/4}\beta(1\pm i)|\xi'|^{5/4}$ and
\begin{equation*}
\wht \Gamma_{\rm Res}^\pm=\{\lambda\in\BC : \lambda=\zeta_\pm+|\xi'|^{6/4}e^{is}, s:0\to 2\pi\}.
\end{equation*}
In addition, for $\CS\in\{\CH,\CU\}$ and $(d,\Bf)\in Y_p$, set  
\begin{equation*}
\CS_{A_0}^{1\pm}(t)d
=\CF_{\xi'}^{-1}[\wht \CS_{A_0}^{\,1}(t;\wht \Gamma_{\rm Res}^\pm)d](x'), \quad 
\CS_{A_0}^{2\pm}(t)\Bf
=\CF_{\xi'}^{-1}[\wht \CS_{A_0}^{\,2}(t;\wht \Gamma_{\rm Res}^\pm)\Bf](x').
\end{equation*}
Then for any $t\geq 1$
\begin{align*}
\|(\CH_{A_0}^{1\pm}(t)d, \CH_{A_0}^{2\pm}(t)\Bf)\|_{L_q(\BR^{N-1})}  
&\leq C t^{-\frac{4(N-1)}{5}\left(\frac{1}{p}-\frac{1}{q}\right)}\|(d,\Bf)\|_{Y_p}, \\
\|(\CU_{A_0}^{1\pm}(t)d, \CU_{A_0}^{2\pm}(t)\Bf)\|_{L_q(\dot\BR^{N})} 
&\leq C t^{-\frac{4(N-1)}{5}\left(\frac{1}{p}-\frac{1}{q}\right)-\frac{4}{5}\left(\frac{1}{2}-\frac{1}{q}\right)}\|(d,\Bf)\|_{Y_p},
\end{align*}
with some positive constant $C$ independent of $t$, $d$, and $\Bf$.
\item
For $\CS\in\{\CH,\CU\}$ and $(d,\Bf)\in Y_p$, set
\begin{align*}
\wtd\CS_{A_0}^{1}(t)d 
&=\CS_{A_0}^1(t)d-\CS_{A_0}^{1+}(t)d-\CS_{A_0}^{1-}(t)d, \\
\wtd\CS_{A_0}^{2}(t)\Bf 
&=\CS_{A_0}^2(t)\Bf-\CS_{A_0}^{2+}(t)\Bf-\CS_{A_0}^{2-}(t)\Bf.
\end{align*}
Let $\gamma_1$ be a constant satisfying
\begin{equation*}
0<\gamma_1<\min\left\{1,2(N-1)\left(\frac{1}{p}-\frac{1}{2}\right)\right\}.
\end{equation*}
Then for any $t\geq 1$
\begin{align*}
\|(\wtd \CH_{A_0}^{1}(t)d, \wtd \CH_{A_0}^{2}(t)\Bf)\|_{L_q(\BR^{N-1})} 
&\leq C t^{-\frac{N-1}{2}\left(\frac{1}{p}-\frac{1}{q}\right)-\frac{3}{4}\gamma_1}\|(d,\Bf)\|_{Y_p}, \\
\|(\wtd \CU_{A_0}^{1}(t)d,\wtd \CU_{A_0}^{2}(t)\Bf)\|_{L_q(\dot\BR^{N})}
&\leq C t^{-\frac{N}{2}\left(\frac{1}{p}-\frac{1}{q}\right)}\|(d,\Bf)\|_{Y_p},
\end{align*}
with some positive constant  $C$ independent of $t$, $d$, and $\Bf$.
\end{enumerate}
\end{thm}

\begin{rmk}
Time-decay estimates for higher-order derivatives of the low frequency part will be discussed in a forthcoming paper.
\end{rmk}

The second one is time-decay estimates for the high frequency part.

\begin{thm}\label{thm:main4}
Let $q\in(1,\infty)$ and $j\in\BN_0$. Suppose that $\rho_->\rho_+>0$.
Then there exist constants $A_\infty\geq 2$ and $\gamma_0>0$
such that for any $t\geq 1$ and $(d,\Bf)\in X_q$
\begin{align*}
\|(\pd_t^j\CH_Z^1(t)d,\pd_t^j\CH_Z^2(t)\Bf)\|_{W_q^{3-1/q}(\BR^{N-1})}
&\leq C e^{-\gamma_0 t}\|(d,\Bf)\|_{X_q}, \\
\|(\pd_t^j\CU_Z^1(t)d,\pd_t^j\CU_Z^2(t)\Bf)\|_{H_q^2(\dot\BR^{N})}
&\leq C e^{-\gamma_0 t}\|(d,\Bf)\|_{X_q}, 
\end{align*}
where $Z\in\{A_\infty,[A_0,A_\infty]\}$ and $C$ is a positive constant independent of $t$, $d$, and $\Bf$.
Here $A_0$ is the positive constant given in Theorem $\ref{thm:main2}$
\end{thm}

Recalling \eqref{U:decomp},
we have time-decay estimates of solutions for \eqref{eq:main}
from Theorems $\ref{thm:main2}$, $\ref{thm:main3}$, and $\ref{thm:main4}$ immediately.

\begin{thm}
Let $1 <   p<2\leq q < \infty$ and $(d,\Bf)\in X_q\cap JY_p$.
Then for any $t\geq 1$
\begin{align*}
\|\CH(t)(d,\Bf)\|_{L_q(\BR^{N-1})}
&\leq Ct^{-\min\left\{\frac{4(N-1)}{5}\left(\frac{1}{p}-\frac{1}{q}\right),\,\frac{(N-1)}{2}\left(\frac{1}{p}-\frac{1}{q}\right)+\frac{3}{4}\gamma_1\right\}}
\|(d,\Bf)\|_{X_q\cap Y_p}, \\
\|\CU(t)(d,\Bf)\|_{L_q(\dot\BR^{N})}
&\leq C t^{-\min\left\{\frac{4(N-1)}{5}\left(\frac{1}{p}-\frac{1}{q}\right)+\frac{4}{5}\left(\frac{1}{2}-\frac{1}{q}\right),
\, \frac{N}{2}\left(\frac{1}{p}-\frac{1}{q}\right)\right\}}
\|(d,\Bf)\|_{X_q\cap Y_p},
\end{align*}
with a positive constant $C$ independent of $t$, $d$, and $\Bf$,
where $\gamma_1$ is the positive constant given in Theorem $\ref{thm:main3}$.
\end{thm}

\section{Representation formulas for solutions}

This section introduces the representation formulas for solutions of \eqref{eq:main}.
In this section, we assume that $\rho_\pm$ are any positive constants except for Lemma \ref{lem:nonzero} (2) below.
Here we collect several symbols appearing in the representation formulas.
Let $z_0=\min\{\mu_+/\rho_+,\mu_-/\rho_-\}$.
For $\xi'=(\xi_1,\dots,\xi_{N-1})\in\BR^{N-1}$ and $\lambda\in\BC\setminus(-\infty,-z_0|\xi'|^2]$, we set
\begin{equation}\label{dfn:AB}
A=|\xi'|, \quad B_\pm=\sqrt{\frac{\rho_\pm}{\mu_\pm}\lambda+|\xi'|^2}, 
\end{equation}
where we have chosen a branch cut along the negative real axis and a branch of the square root so that 
$\Re\sqrt z>0$ for $z\in\BC\setminus(-\infty,0]$.
In addition, 
\begin{equation}\label{dfn:DEM}
D_\pm=\mu_\pm B_\pm +\mu_\mp A, \quad E=\mu_+B_++\mu_-B_-, \quad 
\CM_\pm(a)=\frac{e^{-Aa}-e^{-B_\pm a}}{A-B_\pm} \quad (a\geq 0),
\end{equation}
and also
\begin{align}
F(A,\lambda)&=
-(\mu_+-\mu_-)^2A^3+\{(3\mu_+-\mu_-)\mu_+B_++(3\mu_--\mu_+)\mu_-B_-\}A^2 \notag \\
&+\{(\mu_+B_++\mu_-B_-)^2+\mu_+\mu_-(B_++B_-)^2\}A \notag \\
&+(\mu_+B_++\mu_-B_-)(\mu_+B_+^2+\mu_-B_-^2). \label{dfn:Lop-1}
\end{align}
Except for $D_\pm$ and $E$,
the above symbols are introduced in \cite[(3.3), (3.8), (3.15)]{SS11b}.
Furthermore, the following properties are proved in \cite[Lemmas 4.7 and 4.8]{SS11b}.
\begin{lem}\label{lem:fund-sym}
Let $\varepsilon\in(0,\pi/2)$ and $\rho_\pm$ be any positive constants.
Then the following assertions hold.
\begin{enumerate}[$(1)$]
\item
For any $\xi'\in\BR^{N-1}\setminus\{0\}$ and $\lambda\in\Sigma_{\varepsilon}\cup\{0\}$,
\begin{align*}
&C_1(|\lambda|^{1/2}+A)\leq \Re B_\pm \leq |B_\pm|\leq C_2(|\lambda|^{1/2}+A), \\
&C_1(|\lambda|^{1/2}+A)^3\leq |F(A,\lambda)|\leq C_2(|\lambda|^{1/2}+A)^3,
\end{align*}
with positive constants $C_1$ and $C_2$ depending on $\varepsilon$,
but independent of $\xi'$ and $\lambda$. 
\item
Let $s\in\BR$ and $\alpha'\in\BN_0^{N-1}$.
Then for any $\xi'\in\BR^{N-1}\setminus\{0\}$ and $\lambda\in\Sigma_{\varepsilon}\cup\{0\}$
\begin{align*}
&|\pd_{\xi'}^{\alpha'} A^s|\leq CA^{s-|\alpha'|}, 
\quad |\pd_{\xi'}^{\alpha'} B_\pm^s|\leq C(|\lambda|^{1/2}+A)^{s-|\alpha'|}, \\
&|\pd_{\xi'}^{\alpha'} E^s|\leq C(|\lambda|^{1/2}+A)^{s-|\alpha'|}, \\
&|\pd_{\xi'}^{\alpha'}(A+B_\pm)^s|\leq C(|\lambda|^{1/2}+A)^s A^{-|\alpha'|}, \\
&|\pd_{\xi'}^{\alpha'}F(A,\lambda)^s|  \leq C(|\lambda|^{1/2}+A)^{3s} A^{-|\alpha'|},
\end{align*}
where $C$ is a positive constant depending on $\varepsilon$, $s$, and $\alpha'$,
but independent of $\xi'$ and $\lambda$.
\end{enumerate}
\end{lem}

From Lemma \ref{lem:fund-sym} and the Bell formula of derivatives of composite functions, we have

\begin{lem}\label{fundlem:D}
Let $\varepsilon\in(0,\pi/2)$ and $\rho_\pm$ be any positive constants.
\begin{enumerate}[$(1)$]
\item
Let $s\in\BR$ and $\alpha'\in\BN_0^{N-1}$.
Then for any $\xi'\in\BR^{N-1}\setminus\{0\}$ and $\lambda\in\Sigma_{\varepsilon}\cup\{0\}$
\begin{align*}
|\pd_{\xi'}^{\alpha'}D_\pm^s|\leq C(|\lambda|^{1/2}+A)^{s}A^{-|\alpha'|}, \quad
|\pd_{\xi'}^{\alpha'}(D_++D_-)^s|\leq C(|\lambda|^{1/2}+A)^{s}A^{-|\alpha'|}, 
\end{align*}
where $C$ is a positive constant depending on $\varepsilon$, $s$, and $\alpha'$,
but independent of $\xi'$ and $\lambda$.
\item
Let $\theta,\nu,\tau>0$ and $\alpha'\in\BN_0^{N-1}$.
Then there exists a positive constant $C_{\alpha',\theta}$, independent of $\nu$ and $\tau$,
such that for any $\xi'\in\BR^{N-1}\setminus\{0\}$
\begin{equation*}
|\pd_{\xi'}^{\alpha'}e^{-\nu\tau A^\theta}|\leq C_{\alpha',\theta}A^{-|\alpha'|}e^{-(\nu\tau A^\theta)/2}.
\end{equation*}

\end{enumerate}
\end{lem}

\subsection{A representation formula for the parabolic part}
In this subsection, we introduce a representation formula for solutions of \eqref{eq:para}. 
Note that \eqref{eq:para} admits a unique solution $(\Bu^P,\Fp^P)$
for $\lambda\in\Sigma_\varepsilon$ and $\Bf\in L_q(\dot\BR^N)^N$, $\varepsilon\in(0,\pi/2)$ and  $q\in(1,\infty)$,
as discussed in Section 1. 
Since $C_0^\infty(\dot\BR^N)$ is dense in $L_p(\dot\BR^N)$ for $1\leq p <\infty$,
it suffices to consider $\Bf\in C_0^\infty(\dot\BR^N)^N$ in what follows.
Our aim is to prove

\begin{prp}\label{thm:sol-para}
Let $\varepsilon\in(0,\pi/2)$ 
and $\rho_\pm$ be any positive constants.
Suppose that $(\Bu^P,\Fp^P)$ is a solution to \eqref{eq:para}
for some $\lambda\in\Sigma_\varepsilon$ and $\Bf=(f_1,\dots,f_N)^\SST\in C_0^\infty(\dot\BR^N)^N$. 
Then there holds
\begin{align*}
&\wht u_N^P(\xi',x_N,\lambda)|_{x_N=0} \\
&=
\sum_{\Fa,\Fb\in\{+,-\}}\sum_{j=1}^N
\int_0^\infty \frac{\Phi_j^{\Fa,\Fb}(\xi',\lambda)}{A(A+B_\Fb)F(A,\lambda)}\CM_\Fb(y_N)\wht f_j(\xi',\Fa y_N)\,dy_N \\
&+\sum_{\Fa,\Fb\in\{+,-\}}\sum_{j=1}^N\int_0^\infty 
\frac{\Psi_j^{\Fa,\Fb}(\xi',\lambda)}{AB_\Fb(A+B_\Fb)F(A,\lambda)}e^{-B_\Fb y_N}\wht f_j(\xi',\Fa y_N)\,dy_N.
\end{align*}
Here $\wht u_N^P(\xi',x_N,\lambda)$ stands for the partial Fourier transform
of the $N$th component of $\Bu^P$ and 
\begin{align*}
\Phi_j^{\Fa,\Fb}(\xi',\lambda)
&=\sum_{\substack{|\alpha'|+k+l+m=5 \\ |\alpha'|+k\geq 2}}
C^{\Fa,\Fb}_{j,\alpha',k,l,m}(\xi')^{\alpha'}A^k B_+^l B_-^m, \\
\Psi_j^{\Fa,\Fb}(\xi',\lambda)
&=\sum_{\substack{|\alpha'|+k+l+m=5 \\ |\alpha'|+k\geq 2}}
\wtd C^{\Fa,\Fb}_{j,\alpha',k,l,m}(\xi')^{\alpha'}A^k B_+^l B_-^m,
\end{align*}
with constants $C^{\Fa,\Fb}_{j,\alpha',k,l,m}$ and $\wtd C^{\Fa,\Fb}_{j,\alpha',k,l,m}$
independent of $\xi'$ and $\lambda$. 
\end{prp}

\begin{proof}
Let $j=1,\dots,N-1$ and $J=1,\dots,N$ in this proof.
Although we follow calculations of \cite[Section 3]{SS11b},
we will achieve a set of equations simpler than \cite[(3.6)]{SS11b} in what follows,
see \eqref{s-eq} below.

{\bf Step 1.}
We compute $\wht w_N(\xi',0\,-,\lambda)$ for the following resolvent problem:
\begin{equation}\label{rp:w}
\left\{\begin{aligned}
\rho\lambda \Bw-\Dv(\mu \BD(\Bw)-\Fr \BI) &=0 && \text{in $\dot\BR^N$,} \\
\dv\Bw&=0 && \text{in $\dot \BR^N$,} \\
\jmp{(\mu\BD(\Bw)-\Fr\BI)\Be_N}&=\Bg && \text{on $\BR^{N-1}$,} \\
\jmp{\Bw}&=\Bh && \text{on $\BR^{N-1}$,}
\end{aligned}\right.
\end{equation}
where $\Bg=(g_1,\dots,g_N)^\SST$
and $\Bh=(h_1,\dots,h_N)^\SST$ are suitable functions on $\BR^{N-1}$ specified in Step 2 below.
The restriction of $\Bw$ and $\Fr$ to $\BR_\pm^N$ are denoted by $\Bw_\pm$ and $\Fr_\pm$, respectively.
Let us write the $J$th component of $\Bw_\pm$ by $w_{J\pm}$
and observe $\Dv(\mu\BD(\Bw_\pm))=\mu_\pm\Delta\Bw_\pm$ by $\dv\Bw_\pm =0$ in $\BR_\pm^N$.
Then \eqref{rp:w} can be written as 
\begin{equation*}
\left\{\begin{aligned}
\rho_\pm\lambda w_{J\pm}-\mu_\pm\Delta w_{J\pm}+\pd_J\Fr_\pm &=0 && \text{in $\BR_\pm^N$,} \\
\sum_{j=1}^{N-1}\pd_j w_{j\pm}+\pd_N w_{N\pm}&=0 && \text{in $\BR_\pm^N$,} \\
\mu_+(\pd_N w_{j+}+\pd_j w_{N+})|_{x_N=0}-\mu_-(\pd_N w_{j-}+\pd_j w_{N-})|_{x_N=0}
&= g_j && \text{on $\BR^{N-1}$,} \\
(2\mu_+\pd_N w_{N+}-\Fr_+)|_{x_N=0}-(2\mu_-\pd_N w_{N-}-\Fr_-)|_{x_N=0}&=g_N && \text{on $\BR^{N-1}$,} \\
w_{J+}|_{x_N=0}-w_{J-}|_{x_N=0}&= h_J && \text{on $\BR^{N-1}$.}
\end{aligned}\right.
\end{equation*}
Let $\wht w_{J\pm}(x_N)=\wht w_{J\pm}(\xi',x_N,\lambda)$ and $\wht\Fr_\pm(x_N)=\wht \Fr_\pm(\xi',x_N,\lambda)$.
Applying the partial Fourier transform to the last system yields
\begin{align}
\rho_\pm\lambda\wht w_{j\pm}(x_N)
-\mu_\pm(\pd_N^2-|\xi'|^2)\wht w_{j\pm}(x_N)
+i\xi_j\wht \Fr_\pm(x_N)&=0, \text{ $\pm x_N>0$,} \label{eq:w-1} \\
\rho_\pm\lambda\wht w_{N\pm}(x_N)
-\mu_\pm(\pd_N^2-|\xi'|^2)\wht w_{N\pm}(x_N)
+\pd_N\wht \Fr_\pm(x_N)&=0, \text{ $\pm x_N>0$,} \label{eq:w-2} \\
\sum_{j=1}^{N-1}i\xi_j\wht w_{j\pm}(x_N)+\pd_N \wht w_{N\pm}(x_N) &=0, \text{ $\pm x_N>0$,}\label{eq:w-3} \\
\mu_+(\pd_N\wht w_{j+}(0)+i\xi_j\wht w_{N+}(0))
-\mu_-(\pd_N\wht w_{j-}(0)+i\xi_j\wht w_{N-}(0))&=\wht g_j, \label{eq:w-4} \\
(2\mu_+\pd_N\wht w_{N+}(0)-\wht \Fr_+(0))-(2\mu_-\pd_N\wht w_{N-}(0)-\wht \Fr_-(0))&=\wht g_N, \label{eq:w-5} \\
\wht w_{J+}(0)-\wht w_{J-}(0)&=\wht h_J, \label{eq:w-6}
\end{align}
where $\wht g_J=\wht g_J(\xi')$ and $\wht h_J=\wht h_J(\xi')$.
Note that \eqref{eq:w-1} and \eqref{eq:w-2} are respectively equivalent to
\begin{align}
-\mu_\pm(\pd_N^2-B_\pm^2)\wht w_{j\pm}(x_N)+i\xi_j\wht \Fr_\pm(x_N)&=0, \text{ $\pm x_N>0$,} \label{eq:w-7} \\
-\mu_\pm(\pd_N^2-B_\pm^2)\wht w_{N\pm}(x_N)+\pd_N\wht \Fr_\pm(x_N)&=0, \text{ $\pm x_N>0$.} \label{eq:w-8}
\end{align}

From now on, 
we look for $\wht w_{J\pm}(x_N)$ and $\wht\Fr_\pm(x_N)$ of the forms: for $\pm x_N>0$,
\begin{equation}\label{sol:form}
\wht w_{J\pm}(x_N)=\alpha_{J\pm}(e^{\mp A x_N}-e^{\mp B_\pm x_N})+\beta_{J\pm} e^{\mp B_\pm x_N}, \quad
\wht\Fr_\pm(x_N)=\gamma_\pm e^{\mp A x_N}.
\end{equation}
Inserting these formulas into \eqref{eq:w-7}, \eqref{eq:w-8}, and \eqref{eq:w-3}-\eqref{eq:w-6} furnishes
\begin{align}
-\mu_\pm\alpha_{j\pm}(A^2-B_\pm^2)+i\xi_j\gamma_\pm =0, \quad 
-\mu_\pm\alpha_{N\pm}(A^2-B_\pm^2)\mp A\gamma_\pm &=0, \label{eq:w-11} \\
i\xi'\cdot\alpha_\pm'\mp A\alpha_{N\pm}=0,  \quad 
-i\xi'\cdot\alpha_\pm'+i\xi'\cdot\beta_\pm'\pm B_\pm\alpha_{N\pm}\mp B_\pm \beta_{N\pm}&=0, \label{eq:w-13} \\
\mu_+\{\alpha_{j+}(-A+B_+)-\beta_{j+}B_++i\xi_j \beta_{N+}\}  \qquad\qquad\qquad\qquad\qquad& \notag \\
-\mu_-\{\alpha_{j-}(A-B_-)+\beta_{j-}B_-+i\xi_j\beta_{N-}\}&=\wht g_j, \label{eq:w-14} \\
[2\mu_+\{\alpha_{N+}(-A+B_+)-\beta_{N+}B_+\}-\gamma_+] \qquad\qquad\qquad\qquad\qquad& \notag \\ 
-[2\mu_-\{\alpha_{N-}(A-B_-)+\beta_{N-}B_-\}-\gamma_-]&=\wht g_N, \label{eq:w-15}  \\
\beta_{J+}-\beta_{J-}&=\wht h_J, \label{eq:w-16}
\end{align}
where $i\xi'\cdot\alpha_\pm'=\sum_{j=1}^{N-1}i\xi_j\alpha_{j\pm}$ and $i\xi'\cdot\beta'=\sum_{j=1}^{N-1}i\xi_j\beta_{j\pm}$.

Let us solve the equations \eqref{eq:w-11}-\eqref{eq:w-16}.
By \eqref{eq:w-13}, we have
\begin{equation}\label{eq:w-20}
\alpha_{N\pm}
=\pm\frac{(i\xi'\cdot\beta_\pm'\mp B_\pm\beta_{N\pm})}{A-B_\pm}, \quad
i\xi'\cdot\alpha_\pm'
=\frac{A(i\xi'\cdot\beta_\pm'\mp B_\pm\beta_{N\pm})}{A-B_\pm}.
\end{equation}
By the first equation of \eqref{eq:w-20} and the second equation of \eqref{eq:w-11},
\begin{equation}
\gamma_\pm=-\frac{\mu_\pm(A+B_\pm)}{A}(i\xi'\cdot\beta_\pm'\mp B_\pm\beta_{N\pm}). \label{eq:w-22}
\end{equation}
Multiplying \eqref{eq:w-14} by $i\xi_j$ and summing the resultant formulas yield
\begin{multline*}
\mu_+\{i\xi'\cdot\alpha_+'(-A+B_+)-i\xi'\cdot\beta_+' B_+-A^2\beta_{N+}\} \\
-\mu_-\{i\xi'\cdot\alpha_-'(A-B_-)+i\xi'\cdot\beta_-'B_--A^2\beta_{N-}\}
=i\xi'\cdot\wht g',
\end{multline*}
where $i\xi'\cdot\wht g'=\sum_{j=1}^{N-1}i\xi_j\wht g_j$.
Combining this equation with the second one of \eqref{eq:w-20}, furnishes
\begin{multline}\label{eq:w-23}
\mu_+\{-(A+B_+)i\xi'\cdot\beta_+'+A(B_+-A)\beta_{N+}\} \\
-\mu_-\{(A+B_-)i\xi'\cdot\beta_-'+A(B_--A)\beta_{N-}\}=i\xi'\cdot\wht g'. 
\end{multline}
In addition, by \eqref{eq:w-15} and \eqref{eq:w-22} together with the first equation of \eqref{eq:w-20}, 
\begin{multline}\label{eq:w-24}
\mu_+\{(-A+B_+)i\xi'\cdot\beta_+'-B_+(A+B_+)\beta_{N+}\}  \\
-\mu_-\{(-A+B_-)i\xi'\cdot\beta_-'+B_-(B_-+A)\beta_{N-}\}=A\wht g_N.
\end{multline}
Since it follows from \eqref{eq:w-16} that
$i\xi'\cdot\beta_+'=i\xi'\cdot\beta_-'+i\xi'\cdot\wht h'$
and $\beta_{N+}=\beta_{N-}+\wht h_N$, 
it holds by \eqref{eq:w-23} and \eqref{eq:w-24} that
\begin{align}
&\{\mu_+(A+B_+)+\mu_-(A+B_-)\}i\xi'\cdot\beta_-'  
-\{\mu_+A(B_+-A)-\mu_-A(B_--A)\}\beta_{N-} \notag  \\
&=-i\xi'\cdot \wht g'-\mu_+(A+B_+)i\xi'\cdot\wht h'+\mu_+A(B_+-A)\wht h_N =:G(\Bg,\Bh), \notag \\
&-\{\mu_+(-A+B_+)-\mu_-(-A+B_-)\}i\xi'\cdot\beta_-' \notag \\
&+\{\mu_+B_+(A+B_+)+\mu_-B_-(B_-+A)\}\beta_{N-} \notag \\
&=-A\wht g_N+\mu_+(-A+B_+)i\xi'\cdot\wht h'-\mu_+B_+(A+B_+)\wht h_N=:H(\Bg,\Bh). \label{dfn:G-H}
\end{align}
We have thus achieved
\begin{equation}\label{s-eq}
\BL\begin{pmatrix}i\xi'\cdot\beta_-' \\ \beta_{N-}\end{pmatrix}
=\begin{pmatrix}G(\Bg,\Bh) \\ H(\Bg,\Bh) \end{pmatrix},
\end{equation}
where
\begin{equation*}
\BL=
\begin{pmatrix}
\mu_+(A+B_+)+\mu_-(A+B_-) & -\mu_+A(B_+-A)+\mu_-A(B_--A) \\
-\mu_+(-A+B_+)+\mu_-(-A+B_-) & \mu_+B_+(A+B_+)+\mu_-B_-(B_-+A)
\end{pmatrix}.
\end{equation*}
Then the inverse matrix $\BL^{-1}$ of $\BL$ is given by
\begin{equation*}
\BL^{-1}=\frac{1}{F(A,\lambda)}
\begin{pmatrix}
L_{11} & L_{12} \\
L_{21} & L_{22}
\end{pmatrix},
\end{equation*}
where $F(A,\lambda)$ is defined in \eqref{dfn:Lop-1} and
\begin{align}
L_{11}
&=\mu_+B_+(A+B_+)+\mu_-B_-(B_-+A), \notag \\
L_{12}
&=\mu_+A(B_+-A)-\mu_-A(B_--A), \notag \\
L_{21}
&=\mu_+(-A+B_+)-\mu_-(-A+B_-), \notag\\
L_{22}
&=\mu_+(A+B_+)+\mu_-(A+B_-).  \label{dfn:cofa}
\end{align}
Solving \eqref{s-eq}, we have
\begin{equation}\label{sol:beta}
i\xi'\cdot\beta_-'= \frac{L_{11}G(\Bg,\Bh)+L_{12}H(\Bg,\Bh)}{F(A,\lambda)}, 
\quad \beta_{N-}=\frac{L_{21}G(\Bg,\Bh)+L_{22}H(\Bg,\Bh)}{F(A,\lambda)}.
\end{equation}
Since $\wht w_N(\xi',0\,-,\lambda)=\wht w_{N-}(\xi',0,\lambda)=\beta_{N-}$,
there holds
\begin{equation}\label{w-final}
\wht w_N(\xi',0\,-,\lambda)=\frac{L_{21}G(\Bg,\Bh)+L_{22}H(\Bg,\Bh)}{F(A,\lambda)}.
\end{equation}

{\bf Step 2.}
We compute the formula of $\wht u_N^P(\xi',0,\lambda)$.
Let 
$$(\psi_\pm,\phi_\pm)=(\psi_{1\pm},\dots,\psi_{N\pm},\phi_\pm)$$
be solutions to the following whole space problems without interface:
\begin{equation}\label{eq:1-A}
\left\{\begin{aligned}
\rho_+\lambda\psi_+-\Dv(\mu_+\BD(\psi_+)-\phi_+\BI)=\BF, \quad \dv\psi_+=0 \quad \text{in $\BR^N$,}  \\
\rho_-\lambda\psi_--\Dv(\mu_-\BD(\psi_-)-\phi_-\BI)=\BF, \quad \dv\psi_-=0 \quad \text{in $\BR^N$,} 
\end{aligned}\right.
\end{equation}
where $\BF=\rho\Bf$.
Set $\psi=\psi_+\mathds{1}_{\BR_+^N}+\psi_-\mathds{1}_{\BR_-^N}$ and 
$\phi=\phi_+\mathds{1}_{\BR_+^N}+\phi_-\mathds{1}_{\BR_-^N}$, and then $(\psi,\phi)$ satisfies
\begin{equation*}
\rho\lambda-\Dv(\mu\BD(\psi)-\phi\BI)=\rho\Bf, \quad \dv\psi=0 \quad \text{in $\dot\BR^N$.}
\end{equation*}
In addition, $\jmp{\phi}=0$ as discussed in the appendix below.
Thus $(\Bu^P,\Fp^P)$ are given by  $\Bu^P=\psi+\Bv$ and $\Fp^P=\phi+\Fq$ for
a solution $(\Bv,\Fq)$ to
\begin{equation*}
\left\{\begin{aligned}
\rho\lambda\Bv-\Dv(\mu\BD(\Bv)-\Fq\BI)&=0 && \text{in $\dot\BR^N$,} \\
\dv\Bv&=0 && \text{in $\dot\BR^N$,} \\
\jmp{(\mu\BD(\Bv)-\Fq\BI)\Be_N}&=-\jmp{\mu\BD(\psi)\Be_N} && \text{on $\BR^{N-1}$,} \\
\jmp{\Bv}&=-\jmp{\psi} && \text{on $\BR^{N-1}$.}
\end{aligned}\right.
\end{equation*}

We now see that by \eqref{w-final} with $\Bg=-\jmp{\mu\BD(\psi)\Be_N}$ and $\Bh=-\jmp{\psi}$
\begin{align*}
\wht v_{N}(\xi',0\,-,\lambda)
&=\frac{L_{21}}{F(A,\lambda)}
\big\{\mu_+(i\xi'\cdot\pd_N\wht \psi_+'(\xi',0,\lambda)-A^2\wht \psi_{N+}(\xi',0,\lambda)) \\
&-\mu_-(i\xi'\cdot\pd_N\wht \psi_-'(\xi',0,\lambda)-A^2\wht \psi_{N-}(\xi',0,\lambda))  \\
&+\mu_+(A+B_+)i\xi'\cdot (\wht \psi_+'(\xi',0,\lambda)-\wht \psi_-'(\xi',0,\lambda)) \\
& -\mu_+A(B_+-A)(\wht \psi_{N+}(\xi',0,\lambda)-\wht \psi_{N-}(\xi',0,\lambda))\big\} \\
&+\frac{L_{22}}{F(A,\lambda)}
\big\{A(2\mu_+\pd_N\wht \psi_{N+}-2\mu_-\pd_N\wht \psi_{N-}) \\
& -\mu_+(-A+B_+)i\xi'\cdot (\wht \psi_+'(\xi',0,\lambda)-\wht \psi_-'(\xi',0,\lambda)) \\
&+\mu_+ B_+(A+B_+)(\wht \psi_{N+}(\xi',0,\lambda)-\wht \psi_{N-}(\xi',0,\lambda))\big\}.
\end{align*}
Since $\wht u_N^P(\xi',0,\lambda)=\wht u_N^P(\xi',0\,-,\lambda)$ by $\jmp{\Bu^P}=0$,
there holds 
$$\wht u_N^P(\xi',0,\lambda)=\wht \psi_{N-}(\xi',0,\lambda)+\wht v_N(\xi',0\,-,\lambda).$$
Combing this property with the above formula of $\wht v_{N}(\xi',0\,-,\lambda)$
 and \eqref{res-app} in the appendix below yields the desired formula of $\wht u_N^P(\xi',0,\lambda)$.
This completes the proof of Proposition \ref{thm:sol-para}.
\end{proof}

\subsection{Solution formulas for the hyperbolic-parabolic part}
In this subsection, we introduce solution formulas of $(\eta^k,\Bu^k,\Fp^k)$, $k=1,2$, for \eqref{eq:hype-2}.
System \eqref{eq:hype-2} can be written as
\begin{equation*}
\left\{\begin{aligned}
\rho\lambda \Bu^k -\Dv(\mu\BD(\Bu^k)-\Fp^k\BI)  &=0 && \text{in $\dot\BR^N$,} \\
\dv\Bu^k&=0 && \text{in $\dot\BR^N$, } \\
\jmp{(\mu\BD(\Bu^k)-\Fp^k\BI)\Be_N}&=(\omega-\sigma\Delta')\eta^k\Be_N && \text{on $\BR^{N-1}$, } \\
\jmp{\Bu^k}&=0 && \text{on $\BR^{N-1}$,}
\end{aligned}\right.
\end{equation*}
coupled with
\begin{equation}\label{eq:hype-4}
\lambda\eta^k-u_N^k|_{x_N=0}=z^k \quad \text{on $\BR^{N-1}$.}
\end{equation}

In what follows,
we apply the argumentation in Step 1 for the proof of  Proposition \ref{thm:sol-para} in the previous subsection.
To this end, we set $\Bg=(0,\dots,0,(\omega-\sigma\Delta')\eta^k)^\SST$ and $\Bh=0$ in \eqref{rp:w}.
Then $G(\Bg,\Bh)$ and $H(\Bg,\Bh)$ in \eqref{dfn:G-H} are given by
\begin{equation}\label{relat:G-H}
G(\Bg,\Bh)=0, \quad H(\Bg,\Bh)=-A(\omega+\sigma A^2)\wht \eta^k.
\end{equation}
Since $\Bh=0$, we have by \eqref{eq:w-16}
\begin{equation}\label{relat:beta}
\beta_{J+}=\beta_{J-} \quad (J=1,\dots,N).
\end{equation}
Combining this relation with \eqref{sol:beta} and \eqref{relat:G-H} furnishes
\begin{equation}\label{exp:beta}
\beta_{N+}=\beta_{N-}=-\frac{L_{22}}{F(A,\lambda)}A(\omega+\sigma A^2)\wht \eta^k,
\end{equation}
and thus \eqref{sol:form} gives 
\begin{equation*}
\wht u_N^k(0)=-\frac{L_{22}}{F(A,\lambda)}A(\omega+\sigma A^2)\wht \eta^k.
\end{equation*}
Inserting this formula into \eqref{eq:hype-4} yields
\begin{equation*}
\lambda\wht \eta^k+\frac{L_{22}}{F(A,\lambda)}A(\omega+\sigma A^2)\wht \eta^k=\wht z^k.
\end{equation*}
Solving this equation, we obtain
\begin{equation}\label{Lop-det-2}
\wht \eta^k=\frac{F(A,\lambda)}{L(A,\lambda)}\wht z^k, \quad 
L(A,\lambda)=\lambda F(A,\lambda)+A(\omega+\sigma A^2)(D_++D_-),
\end{equation}
where we have used $L_{22}=D_++D_-$ with $D_\pm=\mu_\pm B_\pm+\mu_\mp A$ given in \eqref{dfn:DEM}.
At this point, we note the following lemma.

\begin{lem}\label{lem:nonzero}
\begin{enumerate}[$(1)$]
\item
Let $\varepsilon\in(0,\pi/2)$ and $\rho_\pm$ be any positive constants.
Then there exists a constant $\delta_0\geq 1$ such that 
for any $\xi'\in\BR^{N-1}\setminus\{0\}$ and $\lambda\in\Sigma_\varepsilon$ with $|\lambda|\geq \delta_0$
\begin{align*}
|L(A,\lambda)|&\geq C_{\varepsilon}(|\lambda|^{1/2}+A)\{|\lambda|(|\lambda|^{1/2}+A)^2+\sigma A^3\}, \\
 |\pd_{\xi'}^{\alpha'} L(A,\lambda)^{-1}| 
&\leq C_{\alpha',\varepsilon}[(|\lambda|^{1/2}+A)\{|\lambda|(|\lambda|^{1/2}+A)^2+\sigma A^3\}]^{-1}A^{-|\alpha'|},
\end{align*}
where $\alpha'\in\BN_0^{N-1}$ and $C_{\varepsilon}$, $C_{\alpha',\varepsilon}$ are positive constants 
independent of $\xi'$ and $\lambda$.
\item
Suppose that $\rho_- >\rho_+>0$.
Let $\xi'\in\BR^{N-1}\setminus\{0\}$ and $\lambda\in\BC$ with $\Re\lambda\geq 0$.
Then $L(A,\lambda)\neq 0$.
\end{enumerate}
\end{lem}	

\begin{proof}
(1) See \cite[Lemma 6.1]{SS11b}.

(2)
The proof is similar to \cite[Lemma 3.2]{SaS1}, so that the detailed proof may be omitted.
\end{proof}

We continue to calculate the solution formulas for \eqref{eq:hype-2}.
By \eqref{exp:beta} and \eqref{Lop-det-2}
\begin{equation}\label{eq:26}
\beta_{N+}=\beta_{N-}=-\frac{L_{22}}{L(A,\lambda)}A(\omega+\sigma A^2)\wht z^k,
\end{equation}
while by \eqref{sol:beta}, \eqref{relat:G-H}, and \eqref{relat:beta}
\begin{equation}\label{eq:27}
i\xi'\cdot\beta_+'=i\xi'\cdot\beta_-'=-\frac{L_{12}}{L(A,\lambda)}A(\omega+\sigma A^2)\wht z^k.
\end{equation}
It thus holds by \eqref{eq:26} and \eqref{eq:27} that
\begin{equation*}
i\xi'\cdot\beta_\pm'\mp B_\pm\beta_{N\pm}=
-\frac{A(\omega+\sigma A^2)}{L(A,\lambda)}
\left(L_{12}\mp B_\pm L_{22}\right)\wht z^k,
\end{equation*}
which, combined with \eqref{eq:w-20} and \eqref{eq:w-22}, furnishes
\begin{align}
&\alpha_{N\pm}
=\mp\frac{1}{A-B_\pm}\cdot\frac{A(\omega+\sigma A^2)}{L(A,\lambda)}\left(L_{12}\mp B_\pm L_{22}\right)\wht z^k, \notag \\
&\gamma_\pm
=\frac{\mu_\pm(A+B_\pm)}{A}\cdot
\frac{A(\omega+\sigma A^2)}{L(A,\lambda)}\left(L_{12}\mp B_\pm L_{22}\right)\wht z^k. \label{final:gam}
\end{align}
By the first equation of \eqref{eq:w-11} and the above formula of $\gamma_\pm$,  we have
\begin{equation}\label{alpha-tan}
\alpha_{j\pm} = \frac{i\xi_j}{(A-B_\pm)A}\cdot\frac{A(\omega+\sigma A^2)}
{L(A,\lambda)}\left(L_{12}\mp B_\pm L_{22}\right)\wht z^k \quad (j=1,\dots,N-1).
\end{equation}
Noting $g_j=0$ for $j=1,\dots,N-1$, we have by \eqref{eq:w-14}, \eqref{relat:beta}, \eqref{eq:26},
and \eqref{alpha-tan}
\begin{align}
&\beta_{j+}=\beta_{j-}= \notag\\ 
&-\frac{i\xi_j(\omega+\sigma A^2)}{E L(A,\lambda)}
\left[(\mu_++\mu_-)L_{12}+\left\{\mu_+(A-B_+)-\mu_-(A-B_-)\right\}L_{22}\right]\wht z^k, \label{beta-tan}
\end{align}
together with $E=\mu_+B_++\mu_-B_-$ given in \eqref{dfn:DEM}.

Let us define for $j=1,\dots,N-1$
\begin{align*}
\CI_{j\pm}(\xi',\lambda)
&=i\xi_j(\omega+\sigma A^2)(L_{12}\mp B_\pm L_{22}), \\
\CI_{N\pm}(\xi',\lambda)
&=\mp A(\omega+\sigma A^2)(L_{12}\mp B_\pm L_{22}), \\
\CJ_{j}(\xi',\lambda)
&=-i\xi_j(\omega+\sigma A^2)\left[(\mu_++\mu_-)L_{12}+\left\{\mu_+(A-B_+)-\mu_-(A-B_-)\right\}L_{22}\right], \\
\CJ_{N}(\xi',\lambda) 
&=-E L_{22}A(\omega+\sigma A^2),
\end{align*}
where $L_{11}$, $L_{12}$, $L_{21}$, and $L_{22}$ are given in \eqref{dfn:cofa}.
Recall $\CM_\pm(a)$ in \eqref{dfn:DEM}.
Then, in view of \eqref{sol:form}, 
we have achieved by \eqref{eq:26}, \eqref{final:gam}, \eqref{alpha-tan}, and \eqref{beta-tan}
\begin{equation}\label{solform-velo}
\wht u_{m\pm}^k(x_N)
=\frac{\CI_{m\pm}(\xi',\lambda)}{L(A,\lambda)}\CM_\pm(\pm x_N)\wht z^k
+\frac{\CJ_m(\xi',\lambda)}{EL(A,\lambda)}e^{\mp B_\pm x_N}\wht z^k
\end{equation}
for $\pm x_N>0$ and $m=1,\dots,N$,
with the pressure $\wht \Fp_\pm^k(x_N)=\gamma_\pm e^{\mp x_N}$ $(\pm x_N>0)$. 
For the above $\wht\eta^k$, $\wht u_{m\pm}^k(x_N)$, and $\wht \Fp_\pm^k(x_N)$,
we define
\begin{align*}
&\eta^k=\CF_{\xi'}^{-1}\big[\wht \eta^k(\xi',\lambda)\big](x'), \quad
u_{m\pm}^k=\CF_{\xi'}^{-1}\big[\wht u_{m\pm}^k(\xi',x_N,\lambda)\big](x'), \\
&\Fp_\pm^k=\CF_{\xi'}^{-1}\big[\wht \Fp_{\pm}^k(\xi',x_N,\lambda)\big](x').
\end{align*}
Then setting 
\begin{equation*}
u_m^k=u_{m+}^k\mathds{1}_{\BR_+^N}+u_{m-}^k\mathds{1}_{\BR_-^N} \quad \text{and} \quad
\Fp^k=\Fp_+^k\mathds{1}_{\BR_+^N}+\Fp_-^k\mathds{1}_{\BR_-^N},
\end{equation*}
we observe that $\eta^k$, $\Bu^k=(u_1^k,\dots,u_N^k)^\SST$, and $\Fp^k$
become a solution to \eqref{eq:hype-2}.
This completes the calculation of solution formulas for \eqref{eq:hype-2}.

\subsection{Representation formulas for \eqref{eq:main}}\label{subsec:3-3}
In this subsection, we give the representation formulas of solutions for \eqref{eq:main}.
To this end, we first consider 
$\CH_Z^1(t)d$, $\CH_Z^2(t)\Bf$, $\CU_Z^1(t)d$, and $\CU_Z^2(t)\Bf$ given in \eqref{eq:sol-main}.
Together with Proposition \ref{thm:sol-para},
inserting $\wht\eta^1$ and $\wht\eta^2$ of \eqref{Lop-det-2} into 
$\wht \CH_{Z}^1(t;\Gamma)d$ and $\wht \CH_{Z}^2(t;\Gamma)\Bf$ in \eqref{hat-H-U}, respectively, yields
\begin{align*}
&\wht \CH_{Z}^1(t;\Gamma)d
=\frac{\varphi_Z(\xi')}{2\pi i}\int_{\Gamma} e^{\lambda t}
\frac{F(A,\lambda)}{L(A,\lambda)}\,d\lambda \, \wht  d(\xi'), \\ 
&\wht \CH_{Z}^2(t;\Gamma)\Bf  \\
&=
\sum_{\Fa,\Fb\in\{+,-\}}\sum_{j=1}^N\int_0^\infty 
\Big(\frac{\varphi_Z(\xi')}{2\pi i}\int_{\Gamma}  e^{\lambda t}
\frac{\Phi_{j}^{\Fa,\Fb}(\xi',\lambda)\CM_\Fb(y_N)}{A(B_\Fb+A)L(A,\lambda)}\,d\lambda\Big) 
\wht f_j(\xi',\Fa y_N)\,dy_N \\
&+\sum_{\Fa,\Fb\in\{+,-\}}\sum_{j=1}^N\int_0^\infty 
\Big(\frac{\varphi_Z(\xi')}{2\pi i}\int_{\Gamma}  e^{\lambda t}
\frac{\Psi_{j}^{\Fa,\Fb}(\xi',\lambda)e^{-B_\Fb y_N}}{AB_\Fb(B_\Fb+A)L(A,\lambda)}\,d\lambda\Big)
 \wht f_j(\xi',\Fa y_N)\,dy_N.
\end{align*}

Let us define for $\pm x_N>0$ and $m=1,\dots,N$
\begin{align*}
\wht \CU_{Z,m\pm}^1(t;\Gamma)d
&=\frac{\varphi_Z(\xi')}{2\pi i}\int_\Gamma e^{\lambda t}\wht u_{m\pm}^1(\xi',x_N,\lambda)\,d\lambda, \\
\wht \CU_{Z,m\pm}^2(t;\Gamma)\Bf
&=\frac{\varphi_Z(\xi')}{2\pi i}\int_\Gamma e^{\lambda t}\wht u_{m\pm}^2(\xi',x_N,\lambda)\,d\lambda.
\end{align*}
Together with Proposition \ref{thm:sol-para},
inserting $\wht u_{m\pm}^1$ and $\wht u_{m\pm}^2$ of \eqref{solform-velo}
into $\wht \CU_{Z,m\pm}^1(t;\Gamma)d$ and $\wht \CU_{Z,m\pm}^2(t;\Gamma)\Bf$, respectively,
yields the following formulas: for $\pm x_N>0$ and $m=1,\dots,N$
\begin{align*}
\wht \CU_{Z,m \pm}^1(t;\Gamma)d
&=\frac{\varphi_Z(\xi')}{2\pi i}\int_{\Gamma} e^{\lambda t}
\frac{ \CI_{m\pm}(\xi',\lambda)}{L(A,\lambda)}\CM_\pm(\pm x_N)\,d\lambda \, \wht d(\xi') \\
&+ \frac{\varphi_Z(\xi')}{2\pi i}\int_{\Gamma} e^{\lambda t}
 \frac{\CJ_m(\xi',\lambda)}{EL(A,\lambda)}e^{\mp B_\pm x_N}\,d\lambda \, \wht d(\xi'),
\end{align*}
and furthermore,
\begin{align*}
&\wht \CU_{Z,m \pm}^2(t;\Gamma)\Bf   \\
&=\sum_{\Fa,\Fb\in\{+,-\}}\sum_{j=1}^N\int_0^\infty\Big(
\frac{\varphi_Z(\xi')}{2\pi i}\int_\Gamma e^{\lambda t}
\frac{\Phi_{j}^{\Fa,\Fb}(\xi',\lambda)\CI_{m\pm}(\xi',\lambda)\CM_{\pm}(\pm x_N)\CM_\Fb(y_N)}{
A(B_\Fb+A)F(A,\lambda)L(A,\lambda)}\,d\lambda\Big) \\
&\times \wht f_j(\xi',\Fa y_N)\,dy_N \\
&+\sum_{\Fa,\Fb\in\{+,-\}}\sum_{j=1}^N\int_0^\infty\Big(
\frac{\varphi_Z(\xi')}{2\pi i}\int_\Gamma e^{\lambda t}
\frac{\Psi_j^{\Fa,\Fb}(\xi',\lambda)\CI_{m\pm}(\xi',\lambda)\CM_{\pm}(\pm x_N)e^{-B_\Fb y_N}}{
AB_\Fb(B_\Fb+A)F(A,\lambda)L(A,\lambda)}\,d\lambda\Big) \\
&\times \wht f_j(\xi',\Fa y_N)\,dy_N \\
&+\sum_{\Fa,\Fb\in\{+,-\}}\sum_{j=1}^N\int_0^\infty
\Big(\frac{\varphi_Z(\xi')}{2\pi i}\int_\Gamma e^{\lambda t}
\frac{\Phi_j^{\Fa,\Fb}(\xi',\lambda)\CJ_{m}(\xi',\lambda)e^{\mp B_\pm x_N}\CM_\Fb(y_N)}
{A(B_\Fb+A)F(A,\lambda)
E L(A,\lambda)}\,d\lambda\Big) \\
&\times \wht f_j(\xi',\Fa y_N)\,dy_N \\
&+\sum_{\Fa,\Fb\in\{+,-\}}\sum_{j=1}^N\int_0^\infty
\Big(\frac{\varphi_Z(\xi')}{2\pi i}\int_\Gamma e^{\lambda t}
\frac{\Psi_j^{\Fa,\Fb}(\xi',\lambda)\CJ_m(\xi',\lambda)e^{\mp B_\pm x_N}e^{-B_\Fb y_N}}
{AB_\Fb(B_\Fb+A)F(A,\lambda)
E L(A,\lambda)}\,d\lambda\Big) \\
&\times  \wht f_j(\xi',\Fa y_N)\,dy_N.
\end{align*}
One defines
\begin{align*}
\wht \CU_{Z,m}^1(t;\Gamma)d
&=(\wht \CU_{Z,m+}^1(t;\Gamma)d)\mathds{1}_{\BR_+^N}
+(\wht \CU_{Z,m-}^1(t;\Gamma)d)\mathds{1}_{\BR_-^N}, \\
\wht \CU_{Z,m}^2(t;\Gamma)\Bf
&=(\wht \CU_{Z,m+}^2(t;\Gamma)\Bf)\mathds{1}_{\BR_+^N}
+(\wht \CU_{Z,m-}^2(t;\Gamma)\Bf)\mathds{1}_{\BR_-^N}, 
\end{align*}
and then there holds the following relation between these formulas and $\wht \CU_Z^1(t;\Gamma)d$, $\wht \CU_Z^2(t;\Gamma)\Bf$
given in \eqref{hat-H-U}:
\begin{align*}
\wht \CU_Z^1(t;\Gamma)d
&=(\wht \CU_{Z,1}^1(t;\Gamma)d,\dots,\wht \CU_{Z,N}^1(t;\Gamma)d)^\SST, \\
\wht \CU_Z^2(t;\Gamma)\Bf
&=(\wht \CU_{Z,1}^2(t;\Gamma)\Bf,\dots,\wht \CU_{Z,N}^2(t;\Gamma)\Bf)^\SST.
\end{align*}

Summing up the above argumentation, we have obtained the representation formulas of 
$\CH_Z^1(t)d$, $\CH_Z^2(t)\Bf$, $\CU_Z^1(t)d$, and $\CU_Z^2(t)\Bf$ given in \eqref{eq:sol-main}
from the formulas of $\wht \CH_Z^1(t;\Gamma)d$, $\wht \CH_Z^2(t;\Gamma)\Bf$,
$\wht \CU_Z^1(t;\Gamma)d$, and $\wht \CU_Z^2(t;\Gamma)\Bf$ as above.
Furthermore, those representation formulas of $\CH_Z^1(t)d$, $\CH_Z^2(t)\Bf$, $\CU_Z^1(t)d$, and $\CU_Z^2(t)\Bf$
give the representation  formulas of solutions for \eqref{eq:main} by the relation \eqref{U:decomp}.

Finally, we introduce another useful formula of $L(A,\lambda)$.

\begin{lem}\label{ano-form-L}
Let $L(A,\lambda)$ be given in \eqref{Lop-det-2} and set
\begin{equation*}
\CL_A(\lambda)=\lambda^2+\frac{4 AD_+D_-}{(\rho_++\rho_-)(D_++D_-)}\lambda+\alpha A+\wtd\sigma A^3, \quad 
\wtd\sigma = \frac{\sigma}{\rho_++\rho_-},
\end{equation*}
where $D_\pm$ and $\alpha$ are defined in \eqref{dfn:DEM} and \eqref{dfn:al-beta}, respectively.
Then
\begin{equation*}\label{dfn:CL}
L(A,\lambda)=(\rho_++\rho_-)(D_++D_-)\CL_A(\lambda).
\end{equation*}
\end{lem}

\begin{proof}
The desired relation follows from an elementary calculation, so that the detailed proof may be omitted.
\end{proof}

\section{Analysis of boundary symbol}
We assume $\rho_->\rho_+>0$ throughout this section
and analyze mainly the boundary symbol $\CL_A(\lambda)$ introduced in the last part of the previous section.
Note that $\alpha$ in \eqref{dfn:al-beta} is positive by the assumption $\rho_->\rho_+>0$.

\subsection{Analysis of low frequency part}
Recalling $\theta_2,\lambda_1$ given in \eqref{dfn:theta} and
$z_0=\min\{\mu_+/\rho_+,\mu_-/\rho_-\}$, we define for $A=|\xi'|$
\begin{equation}\label{z1-z2}
z_1^\pm =-\frac{z_0}{2}A^2\pm i \frac{z_0}{4}A^2, \quad
z_2^\pm= \lambda_1e^{\pm i (\pi-\theta_2)},
\end{equation}
and also
\begin{align}
\wht \Gamma_1^\pm &=\{\lambda\in\BC : \lambda=-\frac{z_0}{2}A^2+\frac{z_0}{4}A^2 e^{\pm is}, \,0\leq s \leq \frac{\pi}{2}\}, \notag \\
\wht \Gamma_2^\pm &=\{\lambda\in\BC : \lambda = z_1^\pm (1-s) +z_2^\pm s, \, 0\leq s \leq 1\}, \notag \\
\wht \Gamma_3 &=\{\lambda\in\BC : \lambda = \lambda_1e^{ is}, \, -(\pi-\theta_2)\leq s \leq \pi-\theta_2\}. \label{gamma1-def}
\end{align}
In addition, we set
\begin{equation*}
\CF_A(\lambda)=(\lambda-\zeta_+)(\lambda-\zeta_-), \quad \CG_A(\lambda)=\CL_A(\lambda)-\CF_A(\lambda),
\end{equation*}
where $\zeta_\pm$ and $\CL_A(A)$ are given in Theorem \ref{thm:main3} and Lemma \ref{ano-form-L}, respectively.
Then
\begin{align}
\CG_A(\lambda)
&=\frac{4 AD_+D_-}{(\rho_++\rho_-)(D_++D_-)}\lambda+\wtd\sigma A^3 \notag \\
&-2\sqrt{2}\alpha^{1/4}\beta A^{5/4}\lambda 
+2\sqrt{2}\alpha^{3/4}\beta A^{7/4}-4\alpha^{1/2}\beta^2 A^{10/4}
\label{dfn:com}
\end{align}
and the following lemma holds.

\begin{lem}\label{lem:roots-1}
There exists a constant $A_1\in(0,1)$
such that $|\CF_A(\lambda)|>|\CG_A(\lambda)|$ for $A\in(0,A_1)$
and $\lambda\in \wht \Gamma_1^+\cup\wht \Gamma_1^-\cup \wht \Gamma_2^+\cup \wht \Gamma_2^-\cup\wht \Gamma_3$.
\end{lem}

\begin{proof}
{\bf Case 1}: $\lambda\in \wht \Gamma_1^+\cup\wht \Gamma_1^-$. 
Let $\lambda=-(z_0/2)A^2+(z_0/4)A^2 e^{is}$ $(-\pi/2\leq s \leq \pi/2)$.
It is clear that
\begin{equation*}
|\CF_A(\lambda)|\geq C A \quad \text{for $A\in(0,A_1)$,}
\end{equation*}
with a sufficiently small $A_1$ 
and a positive constant $C$ independent of $A$ and $\lambda$.

Next, we estimate $|\CG_A(\lambda)|$ from above. Since
\begin{align*}
D_\pm =A\left(\mu_\pm\sqrt{\frac{\rho_\pm}{\mu_\pm}\left(-\frac{z_0}{2}+\frac{z_0}{4}e^{is}\right)+1}+\mu_\mp\right), \\
\Re\left(\sqrt{\frac{\rho_\pm}{\mu_\pm}\left(-\frac{z_0}{2}+\frac{z_0}{4}e^{is}\right)+1}\right)>0,
\end{align*}
there holds
\begin{equation*}
\left|\frac{D_+D_-}{D_++D_-}\right|\leq \frac{|D_+||D_-|}{(\Re D_+)+(\Re D_-)} \leq  CA \quad \text{for $A>0$.}
\end{equation*}
One thus sees that
\begin{equation*}
|\CG_A(\lambda)|\leq C A^{7/4} \quad \text{for $A\in(0,1)$,}
\end{equation*}
which implies $|\CF_A(\lambda)|>|\CG_A(\lambda)|$ for $A\in(0,A_1)$ when $A_1$ is sufficiently small.

{\bf Case 2}: $\lambda\in \wht\Gamma_2^+\cup \wht\Gamma_2^-$.
We consider $\lambda\in \wht\Gamma_2^+$ only.
Let $\lambda=z_1^+(1-s)+z_2^+s$ $(0\leq s \leq 1)$. 
We write $\lambda=-a+bi$ $(a,b\geq 0)$, that is,
\begin{equation*}
 a=\frac{z_0}{2}A^2(1-s)+\lambda_1 (\cos\theta_2)s, \quad 
 b=\frac{z_0}{4}A^2(1-s)+\lambda_1(\sin\theta_2)s.
\end{equation*}
We then observe that
\begin{align}
|\lambda-\zeta_\pm|^2
&=(-a+\sqrt{2}\alpha^{1/4}\beta A^{5/4})^2+(b\mp\alpha^{1/2}A^{1/2}\pm\sqrt{2}\alpha^{1/4}\beta A^{5/4})^2 \notag \\
&=a^2+b^2+\alpha A \mp 2\alpha^{1/2}A^{1/2}b+O(A^{5/4}) \quad \text{as $A\to 0$.} \label{est:path-2}
\end{align}
From this, we immediately see that there exists a constant $A_1\in(0,1)$ such that 
\begin{equation*}
|\lambda-\zeta_-|\geq C(|\lambda|+A^{1/2}) \quad \text{for $A\in(0,A_1)$.}
\end{equation*}
Since $2\alpha^{1/2}A^{1/2}b\leq \varepsilon \alpha A+b^2/\varepsilon$ $(\varepsilon>0)$,
choosing $\varepsilon=2/3$ furnishes
\begin{equation}\label{ineq-sub1}
a^2+b^2+\alpha A - 2\alpha^{1/2}A^{1/2}b\geq 
a^2-\frac{1}{2}b^2+\frac{1}{3}\alpha A.
\end{equation}
In addition,
\begin{equation*}
b=\frac{z_0}{4}A^2(1-s)+\tan\theta_2\cdot \lambda_1(\cos\theta_2)s 
=\frac{1}{2}\left\{\frac{z_0}{2}A^2(1-s)+\frac{1}{4}\lambda_1(\cos\theta_2)s\right\}\leq \frac{1}{2}a,
\end{equation*}
which implies
\begin{equation*}
a^2-\frac{1}{2}b^2 = \frac{1}{4}(a^2+b^2)+\frac{3}{4}(a^2-b^2)\geq \frac{1}{4}(a^2+b^2).
\end{equation*}
It thus holds by \eqref{est:path-2} and \eqref{ineq-sub1} that $|\lambda-\zeta_+|\geq C (|\lambda|+A^{1/2})$
when $A$ is small enough, and therefore
\begin{equation}\label{ineq:F-1}
|\CF_A(\lambda)|\geq C (|\lambda|+A^{1/2})^2 \quad \text{for $A\in(0,A_1)$.}
\end{equation}
In particular,
\begin{equation}\label{ineq:F-2}
|\CF_A(\lambda)|\geq CA^{1/2} (|\lambda|+A^{1/2}) \quad \text{for $A\in(0,A_1)$.}
\end{equation}

Next, we estimate $|\CG_A(\lambda)|$ from above.
By Lemma \ref{fundlem:D} and \eqref{dfn:com}, we have for $A\in(0,1)$
\begin{align}
|\CG_A(\lambda)|
&\leq C\big(A|\lambda|(|\lambda|^{1/2}+A)+A^{5/4}|\lambda|+A^{7/4}\big) \notag \\
&= CA^{3/4} \big(A^{1/4}|\lambda|(|\lambda|^{1/2}+A)+A^{1/2}|\lambda|+A\big) .
\label{est:G-A:1}
\end{align}
Since $|\lambda|\leq \lambda_1$, one sees for $A\in(0,1)$ that 
\begin{equation*}
A^{1/4}|\lambda|(|\lambda|^{1/2}+A) \leq |\lambda|(|\lambda|^{1/2}+A^{1/2})
\leq \lambda_1^{1/2}(|\lambda|+\lambda_1^{1/2}A)\leq C(|\lambda|+A^{1/2})
\end{equation*}
and that $A^{1/2}|\lambda|+A\leq |\lambda|+A^{1/2}$.
Combining these inequalities with \eqref{est:G-A:1} furnishes
\begin{equation*}
|\CG_A(\lambda)|\leq C A^{3/4}(|\lambda|+A^{1/2}) \quad \text{for $A\in (0,1)$},
\end{equation*}
which implies $|\CF_A(\lambda)|>|\CG_A(\lambda)|$ for $A\in(0,A_1)$ when $A_1$ is sufficiently small.

{\bf Case 3}: $\lambda\in \wht\Gamma_3$. Let $\lambda = \lambda_1 e^{is}$ $(-(\pi-\theta_2)\leq s \leq \pi-\theta_2)$.
Then 
\begin{equation*}
|\CF_A(\lambda)|\geq C \quad \text{for $A\in(0,A_1)$,}
\end{equation*}
with a sufficiently small $A_1$ and 
a positive constant $C$ independent of $A$ and $\lambda$. 
Since \eqref{est:G-A:1} is valid for $\lambda\in \wht\Gamma_3$ when $A\in(0,1)$,
there holds
\begin{equation*}
|\CG_A(\lambda)|\leq C A^{3/4} \quad \text{for $A\in(0,1)$.}
\end{equation*}
It therefore holds that $|\CF_A(\lambda)|>|\CG_A(\lambda)|$ for $A\in(0,A_1)$ when $A_1$ is sufficiently small. 
This completes the proof of Lemma \ref{lem:roots-1}.
\end{proof}

By Lemma \ref{lem:roots-1} and Rouch\'e's theorem, we immediately have
\begin{prp}\label{prp:simple}
Let $A_1$ be the positive constant given in Lemma $\ref{lem:roots-1}$ and $A\in(0,A_1)$.
Let $K$ be the region enclosed by
$\wht \Gamma_1^+\cup\wht\Gamma_1^-\cup \wht \Gamma_2^+\cup\wht \Gamma_2^-\cup\wht \Gamma_3$.
Then $\CL_A(\lambda)$ has two zeros in $K$.
\end{prp}

Recalling $\wht\Gamma_{\rm res}^\pm$ given in Theorem \ref{thm:main3},
we prove

\begin{lem}\label{lem:roots-2}
There exists a constant $A_2\in(0,1)$ such that 
$|\CF_A(\lambda)|>|\CG_A(\lambda)|$ for $A\in(0,A_2)$ and $\lambda\in\wht \Gamma_{\rm Res}^+\cup\wht \Gamma_{\rm Res}^-$.
\end{lem}

\begin{proof}
We consider $\lambda\in \wht \Gamma_{\rm Res}^+$ only.
Let $\lambda=\zeta_++A^{6/4}e^{is}$ $(0\leq s \leq 2\pi)$. It is clear that
\begin{equation*}
|\CF_A(\lambda)|\geq CA^{8/4} \quad  \text{for $A\in(0,A_2)$,}
\end{equation*}
with a sufficiently small $A_2$ and 
a positive constant $C$ independent of $A$ and $\lambda$.

Next, we estimate $|\CG_A(\lambda)|$ from above. It holds that
\begin{equation*}
B_\pm=\sqrt{\frac{\rho_\pm}{\mu_\pm}}e^{i(\pi/4)}\alpha^{1/4}A^{1/4}(1+O(A^{3/4})) \quad \text{as $A\to 0$,}
\end{equation*} 
which implies 
\begin{equation*}
D_\pm=\sqrt{\rho_\pm\mu_\pm}e^{i(\pi/4)}\alpha^{1/4}A^{1/4}+O(A) \quad \text{as $A\to 0$.}
\end{equation*}
Therefore,
\begin{equation*}
\frac{D_+D_-}{D_++D_-}=\frac{\sqrt{\rho_+\mu_+}\sqrt{\rho_-\mu_-}}{\sqrt{\rho_+\mu_+}+\sqrt{\rho_-\mu_-}}
e^{i(\pi/4)}\alpha^{1/4}A^{1/4}+O(A) \quad \text{as $A\to 0$.} 
\end{equation*}
From this, recalling the definition of $\beta$ given in \eqref{dfn:al-beta}, we have 
\begin{equation*}
\frac{4AD_+D_-}{(\rho_++\rho_-)(D_++D_-)}\lambda
= (-1+i)\cdot 2\sqrt{2}\alpha^{3/4}\beta A^{7/4}+O(A^{10/4}) \quad \text{as $A\to 0$.}
\end{equation*}
In addition,
\begin{equation*}
-2\sqrt{2}\alpha^{1/4}\beta A^{5/4}\lambda=-i\cdot  2\sqrt{2}\alpha^{3/4}\beta A^{7/4}  +O(A^{10/4}) \quad \text{as $A\to 0$.}
\end{equation*}
Combining these two formulas with \eqref{dfn:com} shows that
\begin{align*}
\CG_A(\lambda)=O(A^{10/4}) \quad \text{as $A\to 0$,}
\end{align*}
which implies $|\CF_A(\lambda)|>|\CG_A(\lambda)|$ for $A\in(0,A_2)$ when $A_2$ is sufficiently small.
This completes the proof of Lemma \ref{lem:roots-2}.
\end{proof}

We now obtain

\begin{prp}\label{prp:roots}
Let $A_1$ and $A_2$ be the positive constants given in Lemma $\ref{lem:roots-1}$
and Lemma $\ref{lem:roots-2}$, respectively.
Then there exists a constant $A_3\in(0,\min\{A_1,A_2\})$ such that 
the following assertions hold.
\begin{enumerate}[$(1)$]
\item
Let $A\in(0,A_3)$. Then for the $K$ in Proposition $\ref{prp:simple}$
\begin{align*}
&\wht \Gamma_{\rm Res}^+ \subset K\cap \{\lambda\in\BC\setminus\{0\} :\frac{\pi}{2}<\arg\lambda<\frac{3}{4}\pi\}, \\
&\wht \Gamma_{\rm Res}^- \subset K\cap\{\lambda\in\BC\setminus\{0\} :-\frac{3}{4}\pi<\arg\lambda<-\frac{\pi}{2}\}.
\end{align*}
\item
Let $A\in(0,A_3)$ and $K_\pm$ be the regions enclosed by $\wht\Gamma_{\rm Res}^\pm$, respectively.
Then $\CL_A(\lambda)$ has a simple zero denoted by $\lambda_+$ in $K_+$ 
and another simple zero denoted by $\lambda_-$ in $K_-$.
\item
Let $\CL_A'(\lambda)$ be the derivative of $\CL_A(\lambda)$ with respect to $\lambda$.
Then the inequalities
\begin{equation*}
|\CL_A'(\lambda_+)|\geq C A^{1/2}, \quad |\CL_A'(\lambda_-)|\geq CA^{1/2}
\end{equation*}
hold for $A\in(0,A_3)$ and a positive constant $C$ independent of $A$.
\end{enumerate}
\end{prp}

\begin{rmk}
The zeros $\lambda_\pm$ of $\CL_A(\lambda)$ satisfy
\begin{equation*}
\lambda_\pm=\zeta_\pm +O(A^{6/4}) = \pm i\alpha^{1/2} A^{1/2}-\sqrt{2}\alpha^{1/4}\beta (1\pm i )A^{5/4}+O(A^{6/4}) \quad \text{as $A\to 0$.}
\end{equation*}
When gravity is not taken into account, i.e. $\alpha=0$,
the asymptotics of the zeros of $\CL_A(\lambda)$ are obtained in \cite{PS2009}.
\end{rmk}

\begin{proof}[Proof of Proposition $\ref{prp:roots}$]
(1) The desired properties can be proved by an elementary calculation,
so that the detailed proof may be omitted.

(2) The result follows from Lemma \ref{lem:roots-2} and Rouch\'e's theorem immediately.

(3) Since $\CL_A'(\lambda_\pm)=\CF_A'(\lambda_\pm)+\CG_A'(\lambda_\pm)$,
one has the desired inequalities immediately by direct calculations together with \eqref{dfn:com}.
This completes the proof of Proposition \ref{prp:roots}.
\end{proof}

\subsection{Analysis of high frequency part}
Let us define 
\begin{equation}\label{cap:lam}
\Lambda(a,b)=\{\lambda\in\BC : \lambda = x+yi, \, -a\leq  x \leq 0, \, -b \leq y \leq b\} \quad (a,b\geq 0).
\end{equation}
We then have
\begin{prp}\label{prp:high}
Let $a,b>0$. Then there exists a sufficiently large positive number $A_{\rm high}=A_{\rm high}(a,b)$
such that the following assertions hold.
\begin{enumerate}[$(1)$]
\item
For any $A\geq A_{\rm high}$ and $\lambda\in\Lambda(a,b)$,
\begin{equation*}
C_1 A \leq \Re B_\pm \leq  |B_\pm| \leq C_2A, 
\end{equation*}
where $C_1$ and $C_2$ are positive constants independent of $A$ and $\lambda$.
\item
For any $A\geq A_{\rm high}$ and $\lambda\in\Lambda(a,b)$,
\begin{equation*}
|F(A,\lambda)|\geq CA^3, \quad  |L(A,\lambda)|\geq  C \sigma A^4,
\end{equation*}
where $C$ is a positive constant independent of $A$, $\lambda$, and $\sigma$.
\end{enumerate}
\end{prp}

\begin{proof}
(1) See \cite[Lemma 5.3]{SaS1}.

(2) First, we estimate $|F(A,\lambda)|$ from below.
Since
\begin{equation}\label{asym-B}
B_\pm = A+O(\frac{1}{A}) \quad \text{as $A\to \infty$,}
\end{equation}
it holds that
\begin{equation*}
F(A,\lambda)=4(\mu_++\mu_-)^2A^3+O(A) \quad \text{as $A\to \infty$.}
\end{equation*}
This implies the desired inequality for $|F(A,\lambda)|$.

Next, we estimate $|L(A,\lambda)|$ from below.
Recall the formula of $\CL_A(\lambda)$ in Lemma \ref{ano-form-L}.
The asymptotics \eqref{asym-B} gives
\begin{equation*}
D_\pm=(\mu_++\mu_-)A+O(\frac{1}{A}) \quad \text{as $A\to \infty$,}
\end{equation*}
which implies
\begin{equation*}
\CL_A(\lambda)=\wtd\sigma A^3+O(A^2) \quad \text{as $A\to \infty$.}
\end{equation*}
Thus, from the formula of $L(A,\lambda)$ in Lemma \ref{ano-form-L},
we see that
\begin{equation*}
L(A,\lambda)=2(\mu_++\mu_-)\sigma A^4+O(A^3) \quad \text{as $A\to \infty$.}
\end{equation*}
This yields the desired inequality of $|L(A,\lambda)|$,
which completes the proof of Proposition \ref{prp:high}.
\end{proof}

Next, we consider $A\in[M_1,M_2]$ for $M_2>M_1>0$.

\begin{prp}\label{prp:mid}
Let $b>0$ and $M_2>M_1>0$. Then there exist $a_0\in(0,1)$
such that the following assertions hold.
\begin{enumerate}[$(1)$]
\item
For any $A\in[M_1, M_2]$ and $\lambda\in\Lambda(a_0,b)$, 
\begin{equation*}
C_1 \leq \Re B_\pm \leq  |B_\pm| \leq C_2, 
\end{equation*}
where $C_1$ and $C_2$ are positive constants independent of $A$ and $\lambda$. 
\item
For any $A\in [M_1,M_2]$ and $\lambda\in\Lambda(a_0,b)$,
\begin{equation*}
|F(A,\lambda)|\geq C, \quad  |L(A,\lambda)|\geq  C,
\end{equation*}
where $C$ is a positive constant independent of $A$ and $\lambda$.
\end{enumerate}
\end{prp}

\begin{proof}
(1) The desired inequalities can be proved by an elementary calculation,
so that the detailed proof may be omitted.

(2) By Lemma \ref{lem:fund-sym}, we see that $F(A,\lambda)\neq 0$ for $A>0$ and  $\Re\lambda\geq 0$.
Then the continuity of $F(A,\lambda)$ and the compactness of $[M_1,M_2]\times \Lambda(0,b)$ 
implies there exists an $m:=\min\{|F(A,\lambda)| : A\in[M_1,M_2], \lambda \in\Lambda(0,b)\}>0$.
Choosing a sufficiently small $a_1\in(0,1)$,
we see that $F(A,\lambda)$ is uniform continuous on $[M_1,M_2]\times \Lambda(a_1,b)$.
Thus there exists an $a_0\in(0,a_1]$ such that 
\begin{equation*}
|F(A,\lambda)|\geq \frac{m}{2} \quad \text{for $(A,\lambda)\in[M_1,M_2]\times\Lambda(a_0,b)$,}
\end{equation*}
which implies the desired inequality of $|F(A,\lambda)|$ holds.
Analogously, the inequality for $|L(A,\lambda)|$ follows from Lemma \ref{lem:nonzero} (2).
This completes the proof of Proposition \ref{prp:mid}.
\end{proof}

\section{Time-decay estimates for low frequency part}
This section proves Theorem \ref{thm:main3}.
Suppose $\rho_->\rho_+>0$ throughout this section.

Let us denote the points of intersection between
$\lambda=se^{\pm i (3\pi/4)}$ $(s\geq 0)$ and $\Gamma_0^\pm$ given in \eqref{dfn:gamma_0}
by $z_3^\pm$. Then we define
\begin{align*}
\wht\Gamma_4^{\pm}
&=\{\lambda\in\BC : \lambda=z_1^\pm(1-s)+z_3^\pm s, 0\leq s \leq 1\}, \\
\wht\Gamma_5^\pm
&=\{\lambda\in\BC :  \lambda=z_3^\pm + se^{\pm i (\pi-\theta_1)}, s\geq 0\},
\end{align*}
where $z_1^\pm$ are given in \eqref{z1-z2}.
Let $A_3$ be the positive constant given in Proposition \ref{prp:roots} and let $A_0\in(0,A_3)$.
By Propositions \ref{prp:simple} and \ref{prp:roots} together with Cauchy's integral theorem,
we see for the operators of \eqref{eq:sol-main} that
\begin{align}
\CS_{A_0}^1(t)d
&=\sum_{\Fa\in\{+,-\}}\sum_{j\in\{{\rm Res},1,4,5\}}\CS_{A_0}^{1}(t;\Gamma_j^\Fa)d, \quad
\CS_{A_0}^{1}(t;\Gamma_j^\Fa)d
:= \CF_{\xi'}^{-1}[\wht \CS_{A_0}^{\,1}(t;\wht \Gamma_j^\Fa)d ](x'), \notag \\
\CS_{A_0}^2(t)\Bf
&=\sum_{\Fa\in\{+,-\}}\sum_{j\in\{{\rm Res},1,4,5\}} \CS_{A_0}^{2}(t;\Gamma_j^\Fa)\Bf, \quad
\CS_{A_0}^{2}(t;\Gamma_j^\Fa)\Bf
:=\CF_{\xi'}^{-1}[\wht \CS_{A_0}^{\,2}(t;\wht \Gamma_j^\Fa)\Bf](x'), \label{dfn:u-h}
\end{align}
where $\CS\in\{\CH,\CU\}$.
Here we have used $\wht \Gamma_1^\pm $ in \eqref{gamma1-def}
and the symbols $\wht \CS_{A_0}^{\,1}(t;\wht \Gamma_j^\Fa)d$, $\wht \CS_{A_0}^{\,2}(t;\wht \Gamma_j^\Fa)\Bf$
introduced in Subsection \ref{subsec:3-3}.

At this point, we introduce several lemmas used in the following argumentation.
From \cite[Lemmas 4.3 and 4.4]{SaS1}, we have the following two lemmas.

\begin{lem}\label{lem:fund1}
Let $c_1,c_2,d \geq 0$ and $\nu_1,\nu_2>0$.
Then there exists a positive constant $C$ such that for any $\tau\geq 0$, $a\geq 0$, and $Z\geq 0$
\begin{equation*}
e^{-c_1 (Z^{\nu_1}) \tau} Z^d e^{-c_2 (Z^{\nu_2}) a}\leq C(\tau^{d/\nu_1}+a^{d/\nu_2})^{-1}.
\end{equation*}
\end{lem}

\begin{lem}\label{lem:fund2}
Let $1\leq p,q\leq \infty$ and $r$ be the dual exponent of $p$.
Suppose that $a>0$, $b_1>0$, and $b_2>0$.
\begin{enumerate}[$(1)$]
\item
For $f\in L_p(0,\infty)$, $x_N>0$, and $\tau>0$, set
\begin{equation*}
I(x_N,\tau)=\int_0^\infty \frac{f(y_N)}{\tau^a+(x_N)^{b_1}+(y_N)^{b_2}}\,dy_N.
\end{equation*}
Then there exists a positive constant $C$, independent of $f$, such that for any $\tau>0$
\begin{equation*}
\|I(\cdot,\tau)\|_{L_q(0,\infty)}\leq C\tau^{-a\left(1-\frac{1}{b_1q}-\frac{1}{b_2r}\right)}\|f\|_{L_p(0,\infty)},
\end{equation*}
provided that $b_1q>1$ and $b_2r(1-1/(b_1 q))>1$.
\item
For $f\in L_p(0,\infty)$ and $\tau>0$, set
\begin{equation*}
J(\tau)=\int_0^\infty \frac{f(y_N)}{\tau^a+(y_N)^{b_2}}\,dy_N.
\end{equation*}
Then there exists a positive constant $C$, independent of $f$, such that for any $\tau>0$
\begin{equation*}
|J(\tau)|\leq C \tau^{-a\left(1-\frac{1}{b_2r}\right)}\|f\|_{L_p(0,\infty)},
\end{equation*}
provided that $b_2r>1$.
\end{enumerate}
\end{lem}

Next, we introduce time-decay estimates 
arising in the study of an evolution equation with the fractional Laplacian.

\begin{lem}\label{lem:Lp-Lq}
Let $1\leq p \leq q \leq \infty$, $\theta>0$, and $\nu>0$. 
Then the following assertions hold.
\begin{enumerate}[$(1)$]
\item
For any $\tau>0$ and $\varphi\in L_p(\BR^{N-1})$
\begin{equation*}
\|\CF_{\xi'}^{-1}[e^{- \nu  \tau |\xi'|^\theta}\wht \varphi(\xi')]\|_{L_q(\BR^{N-1})}
\leq C \tau ^ {-\frac{N-1}{\theta}\left(\frac{1}{p}-\frac{1}{q}\right)}\|\varphi\|_{L_p(\BR^{N-1})},
\end{equation*}
with a positive constant $C$ independent of $\tau$ and $\varphi$.
\item
If it is assumed that $1\leq p \leq 2$ additionally, then for any $\tau>0$ and $\varphi\in L_p(\BR^{N-1})$
\begin{equation*}
\|e^{-\nu\tau |\xi'|^\theta}\wht \varphi\|_{L_2(\BR^{N-1})}
\leq C\tau^{-\frac{N-1}{\theta}\left(\frac{1}{p}-\frac{1}{2}\right)},
\end{equation*}
with a positive constant $C$ independent of $\tau$ and $\varphi$.
\end{enumerate}
\end{lem}

\begin{proof}
(1) See e.g. \cite[Lemma 3.1]{MYZ08} and \cite[Lemma 2.5]{YS16}.

(2) The desired estimate follows from (1) and Parseval's identity immediately, so that the detailed proof may be omitted.
\end{proof}

Let $L_p(\BR^n,X)$ be the $X$-valued Lebesgue spaces on $\BR^n$, $n\in\BN$, for $1\leq p \leq \infty$.
The following lemma is proved in \cite[Theorem 2.3]{SS01}.

\begin{lem}\label{lem:SS01}
Let $X$ be a Banach space and $\|\cdot\|_X$ its norm.
Suppose that $L$ and $n$ be a non-negative integer and positive integer, respectively.
Let $0<\sigma\leq 1$ and $s=L+\sigma-n$.
Let $f(\xi)$ be a $C^\infty$-function on $\BR^n\setminus\{0\}$ with value $X$ and
 satisfy the following two conditions:
\begin{enumerate}[$(1)$]
\item
$\pd_\xi^\gamma f \in L_1(\BR^n,X)$ for any multi-index $\gamma\in\BN_0^n$ with $|\gamma|\leq L$.
\item
For any multi-index $\gamma\in\BN_0^n$, there exists a positive constant $M_\gamma$ such that
\begin{equation*}
\|\pd_\xi^{\gamma} f(\xi)\|_X\leq M_\gamma|\xi|^{s-|\delta|}\quad(\xi\in\BR^n\setminus\{0\}).
\end{equation*}
\end{enumerate}
Then there exists a positive constant $C_{n,s}$ such that
\begin{equation*}
\|\CF_{\xi}^{-1}[f](x)\|_X\leq C_{n,s}\left(\max_{|\gamma|\leq L+2}M_\gamma\right)|x|^{-(n+s)}
\quad(x\in\BR^n\setminus\{0\}).
\end{equation*}
\end{lem}

\subsection{Analysis for $\Gamma_{\rm Res}^\pm$.}
In this subsection, we prove

\begin{thm}\label{thm:gam_res}
Let $1 \leq p < 2 \leq q \leq \infty$ and $\langle t\rangle =t+1$.
Then there exists a constant $A_0\in(0,A_3)$ such that for any $t>0$ and $(d,\Bf)\in Y_p$
\begin{align*}
\|\CH_{A_0}^{1}(t;\Gamma_{\rm Res}^\pm)d\|_{L_q(\BR^{N-1})}
&\leq C \langle t\rangle^{-\frac{4(N-1)}{5}\left(\frac{1}{p}-\frac{1}{q}\right)}\|d\|_{L_p(\BR^{N-1})},  \\
\|\CH_{A_0}^{2}(t;\Gamma_{\rm Res}^\pm)\Bf\|_{L_q(\BR^{N-1})}
&\leq C \langle t\rangle^{-\frac{4(N-1)}{5}\left(\frac{1}{p}-\frac{1}{q}\right)-\frac{4}{5}\left(\frac{1}{p}-\frac{1}{2}\right)}\|\Bf\|_{L_p(\dot\BR^{N})},  \\
\|\CU_{A_0}^1(t;\Gamma_{\rm Res}^\pm)d\|_{L_q(\dot\BR^N)}
&\leq C\langle t\rangle^{-\frac{4(N-1)}{5}\left(\frac{1}{p}-\frac{1}{q}\right)-\frac{4}{5}\left(\frac{1}{2}-\frac{1}{q}\right)}\|d\|_{L_p(\BR^{N-1})}, \\
\|\CU_{A_0}^2(t;\Gamma_{\rm Res}^\pm)\Bf\|_{L_q(\dot\BR^N)}
&\leq C\langle t\rangle^{-\frac{4N}{5}\left(\frac{1}{p}-\frac{1}{q}\right)}\|\Bf\|_{L_p(\dot\BR^{N})},
\end{align*}
where $C$ is a positive constant independent of $t$, $d$, and $\Bf$.
\end{thm}

Recalling $A=|\xi'|$ and $\lambda_\pm$ given in Proposition \ref{prp:roots}, we define 
\begin{equation*}
\CB_\pm=\sqrt{\frac{\rho_\pm}{\mu_\pm}\lambda_\pm+A^2}, \quad 
\CD_\pm=\mu_\pm\CB_\pm+\mu_\mp A, \quad 
\CE=\mu_+\CB_++\mu_-\CB_-.
\end{equation*}
Then we immediately obtain

\begin{lem}\label{lem:g-res-sym}
There exists a constant $A_4\in(0,A_3)$ such that for any $A\in(0,A_4)$
\begin{equation*}
C_1 A^{1/4}\leq \Re \CB_\pm \leq |\CB_\pm| \leq C_2 A^{1/4}, \quad  
C_1 A^{3/4} \leq |F(A,\lambda_\pm)|\leq  C_2 A^{3/4}, 
\end{equation*}
with positive constants $C_1$ and $C_2$ independent of $\xi'$,
and also
\begin{alignat*}{2}
|\Phi_j^{\Fa,\Fb}(\xi',\lambda_\pm)| &\leq C A^{2+(3/4)}, 
\quad &|\Psi_j^{\Fa,\Fb}(\xi',\lambda_\pm)|&\leq C A^{2+(3/4)}, \\
|\CI_{m\pm}(\xi',\lambda_\pm)|&\leq C A^{1+(2/4)},
\quad &|\CJ_m(\xi',\lambda_\pm)| &\leq C A^{1+(2/4)}, 
\end{alignat*}
with a positive constant $C$ independent of $\xi'$,
where $\Fa,\Fb\in\{+,-\}$ and $j,m=1,\dots,N$. 
\end{lem}

Let $A\in(0,A_4)$.
Then, by Lemma \ref{lem:g-res-sym},
we have the following estimates for the symbols of the representation formulas given in Subsection \ref{subsec:3-3}:
for the height function,
\begin{align}
\left|\frac{F(A,\lambda_\pm)}{\CD_{+}+\CD_-}\right| \leq CA^{1/2}, \quad 
\left|\frac{\Phi_j^{\Fa,\Fb}(\xi',\lambda_\pm)}{A(\CB_\Fb+A)(\CD_++\CD_-)}\right| &\leq C A^{5/4}, \notag \\
\left|\frac{\Psi_j^{\Fa,\Fb}(\xi',\lambda_\pm)}{A\CB_\Fb(\CB_\Fb+A)(\CD_++\CD_-)}\right| &\leq C A; \label{est:sym-1}
\end{align}
for the fluid velocity, 
\begin{align}
\left|\frac{\CI_{m\pm}(\xi',\lambda_\pm)}{\CD_++\CD_-}\right|
\leq C A^{5/4}, \quad \left|\frac{\CJ_{m}(\xi',\lambda_\pm)}{\CE(\CD_++\CD_-)}\right|&\leq C A, \notag \\
\left|\frac{\Phi_j^{\Fa,\Fb}(\xi',\lambda_\pm)\CI_{m\pm}(\xi',\lambda_\pm)}{A(\CB_\Fb+A)F(A,\lambda_\pm)(\CD_++\CD_-)}\right|
&\leq  CA^2, \notag \\
\left|\frac{\Psi_j^{\Fa,\Fb}(\xi',\lambda_\pm)\CI_{m\pm}(\xi',\lambda_\pm)}{A\CB_\Fb(\CB_\Fb+A)F(A,\lambda_\pm)(\CD_++\CD_-)}\right|
& \leq CA^{7/4}, \notag\\
\left|\frac{\Phi_j^{\Fa,\Fb}(\xi',\lambda_\pm)\CJ_{m}(\xi',\lambda_\pm)}{A(\CB_\Fb+A)F(A,\lambda_\pm)\CE(\CD_++\CD_-)}\right|
&\leq  CA^{7/4}, \notag \\
\left|\frac{\Psi_j^{\Fa,\Fb}(\xi',\lambda_\pm)\CJ_{m}(\xi',\lambda_\pm)}{A\CB_\Fb(\CB_\Fb+A)F(A,\lambda_\pm)\CE
(\CD_++\CD_-)}\right|
&\leq  CA^{6/4}. \label{est:sym-2}
\end{align}

To prove Theorem \ref{thm:gam_res},
we introduce some technical lemma.
Let us define the following operators:
\begin{align}
&[K_{A_0}(t;\Gamma)d](x)
=\CF_{\xi'}^{-1}\left[
\frac{\varphi_{A_0}(\xi')}{2\pi i}
\int_{\wht\Gamma}e^{\lambda t} k(\xi',\lambda)\,d\lambda \, \wht d(\xi')\right](x'),  \notag\\
&[K_{A_0,\CM}^{\Fa,\Fb}(t;\Gamma)\Bf](x) \notag \\
&=\int_0^\infty\CF_{\xi'}^{-1}\left[
\frac{\varphi_{A_0}(\xi')}{2\pi i}\int_{\wht\Gamma}e^{\lambda t}
k_\CM(\xi',\lambda)\CM_\Fb(y_N)\,d\lambda \, \wht \Bf(\xi',\Fa y_N)\right](x') \,dy_N, \notag \\
&[K_{A_0,B}^{\Fa,\Fb}(t;\Gamma)\Bf](x) \notag \\
&=\int_0^\infty\CF_{\xi'}^{-1}\left[
\frac{\varphi_{A_0}(\xi')}{2\pi i}\int_{\wht\Gamma}e^{\lambda t}
k_B(\xi',\lambda)
e^{-B_\Fb y_N}\,d\lambda \, \wht \Bf(\xi',\Fa y_N)\right](x') \,dy_N, \label{dfn:op-heig}
\end{align}
and also for $\pm x_N>0$
\begin{align}
&[L_{A_0,\CM}^{\pm}(t;\Gamma)d](x)=\CF_{\xi'}^{-1}\left[
\frac{\varphi_{A_0}(\xi')}{2\pi i}\int_{\wht\Gamma}e^{\lambda t} l_\CM(\xi',\lambda)
\CM_\pm(\pm x_N)\,d\lambda \, \wht d(\xi')\right](x'), \notag \\
&[L_{A_0,B}^{\pm}(t;\Gamma)d](x)=\CF_{\xi'}^{-1}\left[
\frac{\varphi_{A_0}(\xi')}{2\pi i}\int_{\wht\Gamma}e^{\lambda t}
l_B(\xi',\lambda)e^{\mp B_\pm y_N}\,d\lambda \, \wht d(\xi')\right](x'), \notag \\
&[L_{A_0,\CM\CM}^{\pm,\Fa,\Fb}(t;\Gamma)\Bf](x) \notag \\
&=\int_0^\infty\CF_{\xi'}^{-1}\left[
\frac{\varphi_{A_0}(\xi')}{2\pi i}\int_{\wht\Gamma}e^{\lambda t}
l_{\CM\CM}(\xi',\lambda)
\CM_\pm(\pm x_N)\CM_\Fb(y_N)\,d\lambda \,  \wht \Bf(\xi',\Fa y_N)\right](x')\,dy_N, \notag \\
&[L_{A_0,\CM B}^{\pm,\Fa,\Fb}(t;\Gamma)\Bf](x) \notag \\
&=\int_0^\infty\CF_{\xi'}^{-1}\left[
\frac{\varphi_{A_0}(\xi')}{2\pi i}\int_{\wht\Gamma}e^{\lambda t}
l_{\CM B}(\xi',\lambda)
\CM_\pm(\pm x_N)e^{-B_\Fb y_N}\,d\lambda \,  \wht \Bf(\xi',\Fa y_N)\right](x')\,dy_N, \notag \\
&[L_{A_0,B\CM}^{\pm,\Fa,\Fb}(t;\Gamma)\Bf](x) \notag \\
&=\int_0^\infty\CF_{\xi'}^{-1}\left[
\frac{\varphi_{A_0}(\xi')}{2\pi i}\int_{\wht\Gamma}e^{\lambda t}
l_{B\CM}(\xi',\lambda)
e^{\mp B_\pm x_N}\CM_\Fb(y_N)\,d\lambda \, \wht \Bf(\xi',\Fa y_N)\right](x') \,dy_N, \notag \\
&[L_{A_0,BB}^{\pm,\Fa,\Fb}(t;\Gamma)\Bf](x) \notag \\
&=
\int_0^\infty\CF_{\xi'}^{-1}\left[
 \frac{\varphi_{A_0}(\xi')}{2\pi i}\int_{\wht\Gamma}e^{\lambda t}
l_{BB}(\xi',\lambda)
e^{\mp B_\pm x_N}e^{-B_\Fb y_N}\,d\lambda \, \wht \Bf(\xi',\Fa y_N)\right](x') \,dy_N.  \label{dfn:op-velo}
\end{align}
Here it is assumed that the symbols
\begin{align*}
&k(\xi',\lambda), \ k_\CM(\xi',\lambda), \ k_B(\xi',\lambda), 
\ l_\CM(\xi',\lambda), \ l_B(\xi',\lambda), \\
&l_{\CM\CM}(\xi',\lambda), \ l_{\CM B}(\xi',\lambda), 
\ l_{B \CM}(\xi',\lambda), \ l_{B B}(\xi',\lambda)
\end{align*}
are infinitely many times differentiable with respect to $\xi'\in\BR^{N-1}\setminus\{0\}$
and holomorphic with respect to $\lambda\in\BC\setminus(-\infty,-z_0|\xi'|^2]$.
Then we have

\begin{lem}\label{lem:res-main}
Let $1\leq p < 2 \leq q \leq \infty$, $\langle t\rangle=t+1$, and $\Fa,\Fb\in\{+,-\}$.
Suppose that 
\begin{equation*}
Z(\xi',\lambda)=\frac{\wtd Z(\xi',\lambda)}{\CL_A(\lambda)}
\quad \text{for $Z\in\{k,k_\CM,k_B,l_\CM,l_B,l_{\CM\CM},l_{\CM B},l_{B\CM},l_{BB}\}$}
\end{equation*}
and that there exists a constant $A_5\in(0,A_3)$ such that for any $A\in(0,A_5)$
\begin{alignat*}{3}
|\wtd k(\xi',\lambda_\pm)|&\leq CA^{1/2},
\quad &|\wtd k_\CM(\xi',\lambda_\pm)|&\leq CA^{5/4},
\quad &|\wtd k_B(\xi',\lambda_\pm)| &\leq CA, \\
|l_\CM(\xi',\lambda_\pm)|&\leq CA^{5/4},
\quad &|\wtd l_B(\xi',\lambda_\pm)|&\leq CA,
\quad &|\wtd l_{\CM\CM}(\xi',\lambda_\pm)| &\leq CA^2, \\
|l_{\CM B}(\xi',\lambda_\pm)|&\leq CA^{7/4},
\quad &|\wtd l_{B\CM}(\xi',\lambda_\pm)|&\leq CA^{7/4},
\quad &|\wtd l_{BB}(\xi',\lambda_\pm)| &\leq CA^{6/4},
\end{alignat*}
with some positive constant $C$ independent of $\xi'$.
Then there exists a constant $A_0\in(0,A_5)$ such that the following assertions hold.
\begin{enumerate}[$(1)$]
\item
For any $t>0$ and $(d,\Bf)\in Y_p$
\begin{align*}
\|K_{A_0}(t;\Gamma_{\rm Res}^\pm)d\|_{L_q(\BR^{N-1})}
&\leq C\langle t\rangle^{-\frac{4(N-1)}{5}\left(\frac{1}{p}-\frac{1}{q}\right)}
 \|d\|_{L_p(\BR^{N-1})}, \\
\|K_{A_0,\CM}^{\Fa,\Fb}(t;\Gamma_{\rm Res}^\pm)\Bf\|_{L_q(\BR^{N-1})}
&\leq C\langle t\rangle^{-\frac{4(N-1)}{5}\left(\frac{1}{p}-\frac{1}{q}\right)-\frac{4}{5}\left(\frac{1}{p}-\frac{1}{2}\right)}
 \|\Bf\|_{L_p(\BR_\Fa^{N})}, \\
\|K_{A_0,B}^{\Fa,\Fb}(t;\Gamma_{\rm Res}^\pm)\Bf\|_{L_q(\BR^{N-1})}
&\leq C\langle t\rangle^{-\frac{4(N-1)}{5}\left(\frac{1}{p}-\frac{1}{q}\right)-\frac{4}{5}\left(\frac{1}{p}-\frac{1}{2}\right)}
 \|\Bf\|_{L_p(\BR_\Fa^{N})},
\end{align*}
with some positive constant $C$ independent of $t$, $d$, and $\Bf$.
\item
Let $\Gamma=\Gamma_{\rm Res}^+$ or $\Gamma=\Gamma_{\rm Res}^-$.
Then for any $t>0$ and $(d,\Bf)\in Y_p$
\begin{align*}
\|L_{A_0,\CM}^{\pm}(t;\Gamma)d\|_{L_q(\BR_\pm^N)}
&\leq C\langle t\rangle^{-\frac{4(N-1)}{5}\left(\frac{1}{p}-\frac{1}{q}\right)-\frac{4}{5}\left(\frac{1}{2}-\frac{1}{q}\right)}\|d\|_{L_p(\BR^{N-1})}, \\
\|L_{A_0,B}^{\pm}(t;\Gamma)d\|_{L_q(\BR_\pm^N)}
&\leq C\langle t\rangle^{-\frac{4(N-1)}{5}\left(\frac{1}{p}-\frac{1}{q}\right)-\frac{4}{5}\left(\frac{1}{2}-\frac{1}{q}\right)}\|d\|_{L_p(\BR^{N-1})}, \\
\|L_{A_0,\CM\CM}^{\pm,\Fa,\Fb}(t;\Gamma)\Bf\|_{L_q(\BR_\pm^N)}
&\leq C\langle t\rangle^{-\frac{4N}{5}\left(\frac{1}{p}-\frac{1}{q}\right)}\|\Bf\|_{L_p(\BR_\Fa^{N})}, \\
\|L_{A_0,\CM B}^{\pm,\Fa,\Fb}(t;\Gamma)\Bf\|_{L_q(\BR_\pm^N)}
&\leq C\langle t\rangle^{-\frac{4N}{5}\left(\frac{1}{p}-\frac{1}{q}\right)}\|\Bf\|_{L_p(\BR_\Fa^{N})}, \\
\|L_{A_0,B\CM}^{\pm,\Fa,\Fb}(t;\Gamma)\Bf\|_{L_q(\BR_\pm^N)}
&\leq C\langle t\rangle^{-\frac{4N}{5}\left(\frac{1}{p}-\frac{1}{q}\right)}\|\Bf\|_{L_p(\BR_\Fa^{N})}, \\
\|L_{A_0,BB}^{\pm,\Fa,\Fb}(t;\Gamma)\Bf\|_{L_q(\BR_\pm^N)}
&\leq C\langle t\rangle^{-\frac{4N}{5}\left(\frac{1}{p}-\frac{1}{q}\right)}\|\Bf\|_{L_p(\BR_\Fa^{N})}, 
\end{align*}
with some positive constant $C$ independent of $t$, $d$, and $\Bf$.
\end{enumerate}
\end{lem}

\begin{proof}
We here consider
\begin{equation*}
K_{A_0}(t;\Gamma_{\rm Res}^+), \quad 
K_{A_0,\CM}^{+,+}(t;\Gamma_{\rm Res}^+), \quad 
L_{A_0,\CM}^+(t;\Gamma_{\rm Res}^+), \quad 
L_{A_0,\CM\CM}^{+,+,+}(t;\Gamma_{\rm Res}^+),
\end{equation*}
only. The other cases can be proved analogously (cf. also \cite[Subsection 4.1]{SaS1}).

{\bf Case 1}: $K_{A_0}(t;\Gamma_{\rm Res}^+)$. By the residue theorem, we have
\begin{equation*}
[K_{A_0}(t; \Gamma_{\rm Res}^+)d](x')
=\CF_{\xi'}^{-1}\Big[
\varphi_{A_0}(\xi')
e^{\lambda_+ t}\frac{\wtd k(\xi',\lambda_+)}{L_A'(\lambda_+)}\,d\lambda \, \wht d(\xi')\Big](x').
\end{equation*}
Recalling $\Re\zeta_\pm=-\sqrt{2}\alpha^{1/4}\beta A^{5/4}$,
we write this formula as
\begin{equation*}
[K_{A_0}(t; \Gamma_{\rm Res}^+)d](x')
=
\CF_{\xi'}^{-1}\Big[
\varphi_{A_0}(\xi')
e^{\frac{\Re\zeta_+}{3}\langle t\rangle}
e^{-\frac{\Re\zeta_+}{3}\langle t\rangle}
e^{\lambda_+ t}\frac{\wtd k(\xi',\lambda_+)}{L_A'(\lambda_+)}\, \wht d(\xi')\Big](x').
\end{equation*}
Combining this formula with Lemma \ref{lem:Lp-Lq} yields
\begin{equation*}
\begin{aligned}
&\|K_{A_0}(t; \Gamma_{\rm Res}^+)d\|_{L_q(\BR^{N-1})} \\
&\leq C\langle t \rangle^{-\frac{4(N-1)}{5}\left(\frac{1}{2}-\frac{1}{q}\right)}
\Big\|\CF_{\xi'}^{-1}\Big[\varphi_{A_0}(\xi')
e^{-\frac{\Re\zeta_+}{3}\langle t\rangle}
e^{\lambda_+ t}\frac{\wtd k(\xi',\lambda_+)}{L_A'(\lambda_+)}\wht d\,\Big]\Big\|_{L_2(\BR^{N-1})}=:I_1(t).
\end{aligned}
\end{equation*}
We choose a sufficiently small $A_0$ so that
\begin{equation*}
|e^{-\frac{\Re\zeta_+}{3}\langle t\rangle}e^{\lambda_+ t}|\leq  C e^{\frac{\Re\zeta_+}{3}\langle t\rangle}
\quad \text{on $\spp\varphi_{A_0}$},
\end{equation*}
and thus we have by Parseval's identity, Proposition \ref{prp:roots}, and the assumption for $\wtd k(\xi',\lambda_+)$ 
\begin{equation}\label{eq:model-1}
I_1(t)
\leq 
C\langle t \rangle^{-\frac{4(N-1)}{5}\left(\frac{1}{2}-\frac{1}{q}\right)}
\|\varphi_{A_0}(\xi')e^{\frac{\Re\zeta_+}{3}\langle t\rangle}\wht d\|_{L_2(\BR^{N-1})}.
\end{equation}
Since $0\leq \varphi_{A_0}\leq 1$, this implies 
$$I_1(t) \leq C \langle t \rangle^{-\frac{4(N-1)}{5}\left(\frac{1}{2}-\frac{1}{q}\right)}
\|e^{(\Re\zeta_+/3)\langle t\rangle}\wht d\|_{L_2(\BR^{N-1})}.$$
Applying Lemma \ref{lem:Lp-Lq} to the right-hand side of the last inequality
furnishes the desired estimate for $K_{A_0}(t; \Gamma_{\rm Res}^+)$.
This completes the proof of Case 1.

{\bf Case 2}: $K_{A_0,\CM}^{+,+}(t;\Gamma_{\rm Res}^+)$. 
In the same way as we have obtained \eqref{eq:model-1}, we obtain
\begin{align*}
&\|K_{A_0,\CM}^{+,+}(t;\Gamma_{\rm Res}^+)\Bf\|_{L_q(\BR^{N-1})} \\
&\leq C\langle t \rangle^{-\frac{4(N-1)}{5}\left(\frac{1}{2}-\frac{1}{q}\right)}
 \int_0^\infty \|\varphi_{A_0}(\xi')
e^{\frac{\Re\zeta_+}{3}\langle t\rangle}
A^{3/4} \CM_+ (y_N)\wht \Bf(\xi',y_N)\|_{L_2(\BR^{N-1})}dy_N \\
&=:I_2(t).
\end{align*}
We choose a sufficiently small $A_0$ so that
\begin{equation}\label{ineq:M}
|\CM_+(y_N)|\leq C A^{-1/4} e^{-cAy_N} \quad \text{on $\spp\varphi_{A_0}$}
\end{equation}
for positive constant $C$ and $c$. Then
\begin{equation*}
I_2(t)\leq 
C\langle t \rangle^{-\frac{4(N-1)}{5}\left(\frac{1}{2}-\frac{1}{q}\right)}
\int_0^\infty \| 
e^{\frac{\Re\zeta_+}{3}\langle t\rangle}
A^{1/2} e^{-cA y_N}\wht \Bf(\xi',y_N) \|_{L_2(\BR^{N-1})}dy_N,
\end{equation*}
which, combined with Lemmas \ref{lem:fund1} and \ref{lem:Lp-Lq}, implies
\begin{equation*}
I_2(t)
\leq 
C\langle t \rangle^{-\frac{4(N-1)}{5}\left(\frac{1}{p}-\frac{1}{q}\right)}
\int_0^\infty \frac{\|\Bf(\xi',y_N)\|_{L_p(\BR^{N-1})}}
{\langle t\rangle^{2/5}+y_N^{1/2}}\,dy_N.
\end{equation*}
Since $1\leq p <2$,
applying Lemma \ref{lem:fund2} to the right-hand side of the last inequality shows that
the desired estimate for $K_{A_0,\CM}^{+,+}(t;\Gamma_{\rm Res}^+)$ holds.
This completes the proof of Case 2.

{\bf Case 3}: $L_{A_0,\CM}^+(t;\Gamma_{\rm Res}^+)$.
In the same way as we have obtained \eqref{eq:model-1}, we obtain
\begin{align*}
&\|L_{A_0,\CM}^{+}(t;\Gamma_{\rm Res}^+)d\|_{L_q(\BR^{N-1})}  \\
&\leq C\langle t \rangle^{-\frac{4(N-1)}{5}\left(\frac{1}{2}-\frac{1}{q}\right)}
\| \varphi_{A_0}(\xi')
e^{\frac{\Re\zeta_+}{3}\langle t\rangle}
A^{3/4} \CM_+ (x_N)\wht d(\xi')\|_{L_2(\BR^{N-1})} 
 =: I_3(x_N,t).
\end{align*}
By \eqref{ineq:M}, we see that
\begin{equation}\label{eq:model2}
I_3(x_N,t)\leq 
C\langle t \rangle^{-\frac{4(N-1)}{5}\left(\frac{1}{2}-\frac{1}{q}\right)}
\|
e^{\frac{\Re\zeta_+}{3}\langle t\rangle}
A^{1/2} e^{-c A x_N}\wht d(\xi')\|_{L_2(\BR^{N-1})}.
\end{equation}
When $q=2$, it follows from \eqref{eq:model2} that
\begin{align*}
\int_0^\infty I_3(x_N,t)^2\,dx_N &\leq C\int_0^\infty 
\|
e^{\frac{\Re\zeta_+}{3}\langle t\rangle}
A^{1/2} e^{-c A x_N}\wht d(\xi')\|_{L_2(\BR^{N-1})}^2\,dx_N \\
&\leq C
\|e^{\frac{\Re\zeta_+}{3}\langle t\rangle}
\wht d(\xi')\|_{L_2(\BR^{N-1})}^2.
\end{align*}
Combining this with Lemma \ref{lem:Lp-Lq}
yields the desired estimate of $L_{A_0,\CM}^+(t;\Gamma_{\rm Res}^+)$ for $q=2$.
When $q>2$, it follows from \eqref{eq:model2} and Lemmas \ref{lem:fund1} and \ref{lem:Lp-Lq} that
\begin{align*}
I(x_N,t)\leq C
\langle t \rangle^{-\frac{4(N-1)}{5}\left(\frac{1}{p}-\frac{1}{q}\right)}
\frac{\|d\|_{L_p(\BR^{N-1})}}{\langle t\rangle^{2/5}+x_N^{1/2}}.
\end{align*}
In this inequality, taking $L_q$ norm of both sides with respect to $x_N\in(0,\infty)$ 
furnishes the desired estimate of $L_{A_0,\CM}^+(t;\Gamma_{\rm Res}^+)$ for $q>2$.
This completes the proof of Case 3.

{\bf Case 4}: $L_{A_0,\CM\CM}^{+,+,+}(t;\Gamma_{\rm Res}^+)$. 
In the same way as we have obtained \eqref{eq:model-1}, we obtain
\begin{align*}
&\|L_{A_0,\CM\CM}^{+,+,+}(t;\Gamma_{\rm Res}^+)d\|_{L_q(\BR^{N-1})} \\
&\leq C\langle t \rangle^{-\frac{4}{5}\left(\frac{1}{2}-\frac{1}{q}\right)}
\int_0^\infty \| \varphi_{A_0}(\xi')
e^{\frac{\Re\zeta_+}{3}\langle t\rangle}
A^\frac{6}{4} \CM_+ (x_N)\CM_+(y_N)\wht \Bf(\xi',y_N)\|_{L_2(\BR^{N-1})}dy_N \\
&=:I_4(x_N,t).
\end{align*}
Combining this with \eqref{ineq:M} and Lemmas \ref{lem:fund1} and \ref{lem:Lp-Lq} yields
\begin{align*}
&I_4(x_N,t) \\
&\leq C\langle t \rangle^{-\frac{4}{5}\left(\frac{1}{2}-\frac{1}{q}\right)}
\int_0^\infty \|
e^{\frac{\Re\zeta_+}{3}\langle t\rangle}
Ae^{-cA(x_N+y_N)}\wht \Bf(\xi', y_N)\|_{L_2(\BR^{N-1})}\,dy_N \\
&\leq C\langle t \rangle^{-\frac{4(N-1)}{5}\left(\frac{1}{p}-\frac{1}{q}\right)} 
\int_0^\infty\frac{\|\Bf(\xi',y_N)\|_{L_p(\BR^{N-1})}}{\langle t \rangle^{4/5}+x_N+y_N}\,dy_N.
\end{align*}
Lemma \ref{lem:fund2} thus yields the desired estimate for $L_{A_0,\CM\CM}^{+,+,+}(t;\Gamma_{\rm Res}^+)$.
This completes the proof of Lemma \ref{lem:res-main}.
\end{proof}

Combining Lemma \ref{lem:res-main} with \eqref{est:sym-1} and \eqref{est:sym-2}
yields Theorem \ref{thm:gam_res} immediately.
This completes the proof of Theorem \ref{thm:gam_res}.

\subsection{Analysis for $\Gamma_1^\pm$.}
In this subsection, we prove 
\begin{thm}\label{gam_1-decay}
Let $1\leq p \leq 2 \leq q \leq \infty$ and $\langle t \rangle=t+1$.
Then there exists a constant $A_0\in(0,A_3)$ such that for any $t>0$ and $(d,\Bf)\in Y_p$
\begin{align*}
\|\CH_{A_0}^{1}(t;\Gamma_1^\pm)d\|_{L_q(\BR^{N-1})}
&\leq C \langle t\rangle^{-\frac{N-1}{2}\left(\frac{1}{p}-\frac{1}{q}\right)-\frac{3}{2}}\|d\|_{L_p(\BR^{N-1})},  \\
\|\CH_{A_0}^{2}(t;\Gamma_1^\pm)\Bf\|_{L_q(\BR^{N-1})}
&\leq C \langle t\rangle^{-\frac{N-1}{2}\left(\frac{1}{p}-\frac{1}{q}\right)
-\frac{1}{2}\left(\frac{1}{p}-\frac{1}{2}\right)-\frac{3}{4}}\|\Bf\|_{L_p(\dot\BR^N)},  \\
\|\CU_{A_0}^1(t;\Gamma_1^\pm)d\|_{L_q(\dot\BR^N)}
&\leq C\langle t\rangle^{-\frac{N-1}{2}\left(\frac{1}{p}-\frac{1}{q}\right)
-\frac{1}{2}\left(\frac{1}{2}-\frac{1}{q}\right)-\frac{3}{4}}\|d\|_{L_p(\BR^{N-1})}, \\
\|\CU_{A_0}^2(t;\Gamma_1^\pm)\Bf\|_{L_q(\dot\BR^N)}
&\leq C\langle t\rangle^{-\frac{N}{2}\left(\frac{1}{p}-\frac{1}{q}\right)}\|\Bf\|_{L_p(\dot\BR^{N})},
\end{align*}
where $C$ is a positive constant independent of $t$, $d$, and $\Bf$.
\end{thm}

To prove Theorem \ref{gam_1-decay},
we start with the following lemma.

\begin{lem}\label{lem:sym-gam1}
There exists a constant $A_6\in(0,A_3)$ such that for any $A\in(0,A_6)$
and $\lambda\in\wht \Gamma_1^+\cup \wht \Gamma_1^-$
\begin{equation*}
C_1 A \leq \Re B_\pm\leq |B_\pm| \leq C_2 A, \quad 
C_1A^3\leq |F(A,\lambda)|\leq C_2 A^3, 
\end{equation*}
with positive constants $C_1$ and $C_2$ independent of $\xi'$ and $\lambda$, and also
\begin{alignat*}{2}
|\Phi_j^{\Fa,\Fb}(\xi',\lambda)|&\leq CA^5, \quad 
&|\Psi_j^{\Fa,\Fb}(\xi',\lambda)|&\leq CA^5, \\
|\CI_{m\pm}(\xi',\lambda)|&\leq CA^3, \quad 
&|\CJ_m(\xi',\lambda)|&\leq CA^3,
\end{alignat*}
with a positive constant $C$ independent of $\xi'$ and $\lambda$,
where $\Fa,\Fb\in\{+,-\}$ and $j,m=1,\dots,N$.
\end{lem}

\begin{proof}
See \cite[Lemma 4.9]{SaS1} for $B_\pm$.
Then the desired estimates for $\Phi_j^{\Fa,\Fb}$, $\Psi_j^{\Fa,\Fb}$, $\CI_{m\pm}$, and $\CJ_m$
follow from the estimates of $B_\pm$ immediately.

We now estimate $|F(A,\lambda)|$. Let $\lambda\in\wht \Gamma_1^+\cup\wht \Gamma_1^-$.
It is clear that $|F(A,\lambda)|\leq CA^3$ by $|B_\pm|\leq C A$.
In what follows, we prove $|F(A,\lambda)|\geq CA^3$. 
Since $\lambda=-(z_0/2)A^2+(z_0/4)A^2 e^{is}$ for $s\in[-\pi/2, \pi/2]$, there holds
\begin{equation}\label{formula-of-F}
F(A,\lambda)=A^3F(1,\zeta),
\quad \zeta=-\frac{z_0}{2}+\frac{z_0}{4}e^{is}.
\end{equation}
It suffices to show that
\begin{equation}\label{est:F-zeta}
F(1,\zeta)\neq 0 \quad \text{for $s\in[-\pi/2,\pi/2]$.}
\end{equation}

Let $s\neq 0$. Then $\zeta\in\Sigma_\varepsilon$ for some $\varepsilon=\varepsilon(s)$.
Therefore $F(1,\zeta)\neq 0$, since $|F(1,\zeta)|\geq C_\varepsilon$ 
for some positive constant $C_\varepsilon$ by Lemma \ref{lem:fund-sym}.

Next, we consider $s=0$. In this case, $\zeta=-z_0/4$ and set $b_\pm=\sqrt{(\rho_\pm/\mu_\pm)\zeta+1}$.
Then
\begin{equation}\label{ineq:bpm}
\frac{\sqrt{3}}{2} \leq b_\pm <1,
\end{equation}
and $F(1,\zeta)$ can be written as
\begin{align*}
F(1,\zeta)&=-(\mu_+-\mu_-)^2+(3\mu_+-\mu_-)\mu_+b_++(3\mu_--\mu_+)\mu_-b_- \\
&+(\mu_+b_++\mu_-b_-)^2+\mu_+\mu_-(b_++b_-)^2 
+(\mu_+b_++\mu_-b_-)(\mu_+b_+^2+\mu_-b_-^2).
\end{align*}
Since it follows from \eqref{ineq:bpm} that
\begin{align*}
&-(\mu_+-\mu_-)^2+(3\mu_+-\mu_-)\mu_+b_++(3\mu_--\mu_+)\mu_-b_- \\
&=-(\mu_+^2+\mu_-^2)+2\mu_+\mu_-+3\mu_+^2b_+ + 3\mu_-^2b_--\mu_+\mu_-(b_++b_-)  \\
&\geq 
-(\mu_+^2+\mu_-^2)+2\mu_+\mu_-+\frac{3\sqrt{3}}{2}(\mu_+^2 + \mu_-^2)-2\mu_+\mu_->0,
\end{align*}
we have $F(1,\zeta)>0$ for $s=0$.
Thus \eqref{est:F-zeta} holds, which implies $|F(1,\zeta)|\geq C$ for any $s\in[-\pi/2,\pi/2]$
and a positive constant $C$ independent of $s$.
Therefore $|F(A,\lambda)|\geq CA^3$ by \eqref{formula-of-F},
which completes the proof of Lemma \ref{lem:sym-gam1}.
\end{proof}

Note that $|\CL_A(\lambda)|\geq CA$ for $\lambda\in\wht\Gamma_1^+\cup\wht\Gamma_1^-$ when $A$ is small enough
as seen in Case 1 of the proof of Lemma \ref{lem:roots-1},
and thus it follows from Lemma \ref{lem:sym-gam1} that $|L(A,\lambda)|\geq CA^2$.
By this inequality and Lemma \ref{lem:sym-gam1},
we have the following estimates for the symbols of the representation formulas given in Subsection \ref{subsec:3-3}:
for the height function,
\begin{equation}\label{sym-gam1-1}
\left|\frac{F(A,\lambda)}{L(A,\lambda)}\right| \leq CA , \quad
\left|\frac{\Phi_j^{\Fa,\Fb}(\xi',\lambda)}{A(B_\Fb+A)L(A,\lambda)}\right| \leq C A, \quad
\left|\frac{\Psi_j^{\Fa,\Fb}(\xi',\lambda)}{AB_\Fb(B_\Fb+A)L(A,\lambda)}\right|\leq C;
\end{equation}
for the velocity
\begin{alignat}{2}\label{sym-gam1-2}
\left|\frac{\CI_{m\pm}(\xi',\lambda)}{L(A,\lambda)}\right|
&\leq C A, \quad 
&\left|\frac{\CJ_{m}(\xi',\lambda)}{E L(A,\lambda)}\right|
&\leq C, \notag \\
\left|\frac{\Phi_j^{\Fa,\Fb}(\xi',\lambda)\CI_{m\pm}(\xi',\lambda)}
{A(B_\Fb+A)F(A,\lambda) L(A,\lambda)}\right|
& \leq CA, \quad
&\left|\frac{\Psi_j^{\Fa,\Fb}(\xi',\lambda)\CI_{m\pm}(\xi',\lambda)}{AB_\Fb(B_\Fb+A)F(A,\lambda)L(A,\lambda)}\right|
& \leq C, \notag \\
\left|\frac{\Phi_j^{\Fa,\Fb}(\xi',\lambda)\CJ_{m}(\xi',\lambda)}{A(B_\Fb+A)F(A,\lambda)E L(A,\lambda)}\right|
&\leq  C, \quad 
&\left|\frac{\Psi_j^{\Fa,\Fb}(\xi',\lambda)\CJ_{m}(\xi',\lambda)}{AB_\Fb(B_\Fb+A)F(A,\lambda)E L(A,\lambda)}\right|
&\leq  \frac{C}{A}.
\end{alignat}

Now, recalling the operators difined in \eqref{dfn:op-heig} and \eqref{dfn:op-velo},
we introduce the following lemma (cf. \cite[Lemma 4.10]{SaS1} for details).

\begin{lem}\label{lem:gam-1}
Let $1\leq p\leq 2 \leq q \leq \infty$, $\langle t\rangle =t+1$, and $\Fa,\Fb\in\{+,-\}$.
Suppose that there exists a constant $A_7\in(0,A_3)$ such that for any $A\in(0,A_7)$
and $\lambda\in\wht \Gamma_1^+\cup\wht \Gamma_1^-$
\begin{alignat*}{3}
|k(\xi',\lambda)|&\leq CA, 
\quad &|k_{\CM}(\xi',\lambda)|&\leq CA, 
\quad &|k_B(\xi',\lambda)|&\leq C, \\
|l_\CM(\xi',\lambda)|&\leq CA, 
\quad &|l_B(\xi',\lambda)|&\leq C, 
\quad &|l_{\CM\CM}(\xi',\lambda)|&\leq CA, \\
|l_{\CM B}(\xi',\lambda)| &\leq C, 
\quad &|l_{B\CM}(\xi',\lambda)| &\leq C, 
\quad &|l_{BB}(\xi',\lambda)| &\leq CA^{-1},
\end{alignat*}
with some positive constant $C$ independent of $\xi'$ and $\lambda$.
Then there exists a constant $A_0\in(0,A_7)$ such that the following assertions hold.
\begin{enumerate}[$(1)$]
\item
For any $t>0$ and $(d,\Bf)\in Y_p$
\begin{align*}
\|K_{A_0}(t; \Gamma_1^\pm)d\|_{L_q(\BR^{N-1})}
&\leq C \langle t\rangle^{-\frac{N-1}{2}\left(\frac{1}{p}-\frac{1}{q}\right)-\frac{3}{2}}\|d\|_{L_p(\BR^{N-1})}, \\
\|K_{A_0,\CM}^{\Fa,\Fb}(t;\Gamma_1^\pm)\Bf\|_{L_q(\BR^{N-1})}
&\leq C \langle t\rangle^{-\frac{N-1}{2}\left(\frac{1}{p}-\frac{1}{q}\right)-\frac{1}{2}\left(\frac{1}{p}-\frac{1}{2}\right)-\frac{3}{4}}
\|\Bf\|_{L_p(\BR_\Fa^N)}, \\
\|K_{A_0,B}^{\Fa,\Fb}(t;\Gamma_1^\pm)\Bf\|_{L_q(\BR^{N-1})}
&\leq C \langle t\rangle^{-\frac{N-1}{2}\left(\frac{1}{p}-\frac{1}{q}\right)-\frac{1}{2}\left(\frac{1}{p}-\frac{1}{2}\right)-\frac{3}{4}}
\|\Bf\|_{L_p(\BR_\Fa^N)},
\end{align*}
with some positive constant $C$ independent of $t$, $d$, and $\Bf$.
\item
Let $\Gamma=\Gamma_1^+$ or $\Gamma=\Gamma_1^-$.
Then for any $t>0$ and $(d,\Bf)\in Y_p$
\begin{align*}
\|L_{A_0,\CM}^{\pm}(t;\Gamma)d\|_{L_q(\BR_\pm^N)}
&\leq C\langle t\rangle^{-\frac{N-1}{2}\left(\frac{1}{p}-\frac{1}{q}\right)
-\frac{1}{2}\left(\frac{1}{2}-\frac{1}{q}\right)-\frac{3}{4}}\|d\|_{L_p(\BR^{N-1})}, \\
\|L_{A_0,B}^{\pm}(t;\Gamma)d\|_{L_q(\BR_\pm^N)}
&\leq C\langle t\rangle^{-\frac{N-1}{2}\left(\frac{1}{p}-\frac{1}{q}\right)
-\frac{1}{2}\left(\frac{1}{2}-\frac{1}{q}\right)-\frac{3}{4}}\|d\|_{L_p(\BR^{N-1})}, \\
\|L_{A_0,\CM\CM}^{\pm,\Fa,\Fb}(t;\Gamma)\Bf\|_{L_q(\BR_\pm^N)}
&\leq C\langle t\rangle^{-\frac{N}{2}\left(\frac{1}{p}-\frac{1}{q}\right)}\|\Bf\|_{L_p(\BR_\Fa^N)}, \\
\|L_{A_0,\CM B}^{\pm,\Fa,\Fb}(t;\Gamma)\Bf\|_{L_q(\dot\BR^N)}
&\leq C\langle t\rangle^{-\frac{N}{2}\left(\frac{1}{p}-\frac{1}{q}\right)}\|\Bf\|_{L_p(\BR_\Fa^N)}, \\
\|L_{A_0,B\CM}^{\pm,\Fa,\Fb}(t;\Gamma)\Bf\|_{L_q(\BR_\pm^N)}&
\leq C\langle t\rangle^{-\frac{N}{2}\left(\frac{1}{p}-\frac{1}{q}\right)}\|\Bf\|_{L_p(\BR_\Fa^N)}, \\
\|L_{A_0,BB}^{\pm,\Fa,\Fb}(t;\Gamma)\Bf\|_{L_q(\BR_\pm^N)}&
\leq C\langle t\rangle^{-\frac{N}{2}\left(\frac{1}{p}-\frac{1}{q}\right)}\|\Bf\|_{L_p(\BR_\Fa^N)},
\end{align*}
with some positive constant $C$ independent of $t$, $d$, and $\Bf$.
\end{enumerate}
\end{lem}

Combining Lemma \ref{lem:gam-1} with \eqref{sym-gam1-1} and \eqref{sym-gam1-2}
proves Theorem \ref{gam_1-decay} immediately.
This completes the proof of Theorem \ref{gam_1-decay}.

\subsection{Analysis for $\Gamma_4^\pm$.}

In this subsection, we prove
\begin{thm}\label{thm:gam-4}
Let $1\leq p < 2 \leq q \leq \infty$ and $\langle t\rangle =t+1$
Then there exists a constant $A_0\in (0,A_3)$ such that for any $t>0$ and $(d,\Bf)\in Y_p$
\begin{align*}
\|\CH_{A_0}^{1}(t;\Gamma_4^\pm)d\|_{L_q(\BR^{N-1})}
&\leq C\langle t\rangle^{-\frac{N-1}{2}\left(\frac{1}{p}-\frac{1}{q}\right)-\frac{3}{4}\gamma_1}\|d\|_{L_p(\BR^{N-1})}, \\
\|\CH_{A_0}^{2}(t;\Gamma_4^\pm)\Bf\|_{L_q(\BR^{N-1})}
&\leq C\langle t\rangle^{-\frac{N-1}{2}\left(\frac{1}{p}-\frac{1}{q}\right)
-\frac{1}{2}\left(\frac{1}{p}-\frac{1}{2}\right)-\frac{3}{4}\gamma_2}\|\Bf\|_{L_p(\dot\BR^{N})}, \\
\|\CU_{A_0}^{1}(t;\Gamma_4^\pm)d\|_{L_q(\dot\BR^{N})}
&\leq C\langle t\rangle^{
-\frac{N-1}{2}\left(\frac{1}{p}-\frac{1}{q}\right)-\frac{1}{2}\left(\frac{1}{2}-\frac{1}{q}\right)-\frac{3}{4}\gamma_3}\|d\|_{L_p(\BR^{N-1})}, \\
\|\CU_{A_0}^{2}(t;\Gamma_4^\pm)\Bf\|_{L_q(\dot\BR^{N})}
&\leq C\langle t\rangle^{-\frac{N}{2}\left(\frac{1}{p}-\frac{1}{q}\right)}\|\Bf\|_{L_p(\dot\BR^{N})}, 
\end{align*}
where $C$ is a positive constant independent of $t$, $d$, and $\Bf$. Here
\begin{align*}
&0<\gamma_1<\min\left\{1,2(N-1)\left(\frac{1}{p}-\frac{1}{2}\right)\right\}, \quad
0<\gamma_2<\min\left\{1,2N\left(\frac{1}{p}-\frac{1}{2}\right)\right\}, \\
&0<\gamma_3<\min\left\{1,2\left((N-1)\left(\frac{1}{p}-\frac{1}{2}\right)+\frac{1}{2}-\frac{1}{q}\right)\right\}.
\end{align*}
\end{thm}

Note that $\wht \Gamma_4^+\cup\wht \Gamma_4^-\subset\Sigma_{\theta_2}$
for $\theta_2$ given in \eqref{dfn:theta}.
By Lemma \ref{lem:fund-sym},
we have for any $\xi'\in\BR^{N-1}\setminus\{0\}$ and $\lambda\in\wht \Gamma_+^4\cup\wht\Gamma_-^4$
\begin{alignat}{2}
|\Phi_j^{\Fa,\Fb}(\xi',\lambda)|&\leq CA^2(|\lambda|^{1/2}+A)^3,
\quad &|\Psi_j^{\Fa,\Fb}(\xi',\lambda)|&\leq C A^2(|\lambda|^{1/2}+A)^3, \notag \\
|\CI_{m\pm}(\xi',\lambda)|&\leq CA(|\lambda|^{1/2}+A)^2, 
\quad &|\CJ_m(\xi',\lambda)|&\leq CA(|\lambda|^{1/2}+A)^2, \label{phi-psi-i-j}
\end{alignat}
where $\Fa,\Fb\in\{+,-\}$ and $j,m=1,\dots,N$.
In addition,
similarly to \eqref{ineq:F-1}, there holds
\begin{equation*}
|\CL_A(\lambda)|\geq C(|\lambda|+A^{1/2})^2
\end{equation*}
for a sufficiently small $A$ and $\lambda\in \wht \Gamma_+^4\cup\wht\Gamma_-^4$,
which, combined with $|\lambda|+A^{1/2}\geq (1/2)(|\lambda|^{1/2}+A^{1/4})^2$
and Lemma \ref{fundlem:D},
furnishes
\begin{equation*}
|L(A,\lambda)|\geq C(|\lambda|^{1/2}+A)(|\lambda|^{1/2}+A^{1/4})^4.
\end{equation*}
By this inequality together with \eqref{phi-psi-i-j} and Lemma \ref{lem:fund-sym},
we have
the following estimates for the symbols of the representation formulas given in Subsection \ref{subsec:3-3}:
for the height function
\begin{align}
&\left|\frac{F(A,\lambda)}{L(A,\lambda)}\right| 
\leq \frac{C}{(|\lambda|^{1/2}+A^{1/4})^2}, \quad
\left|\frac{\Phi_j^{\Fa,\Fb}(\xi',\lambda)}{A(B_\Fb+A)L(A,\lambda)}\right| 
\leq \frac{CA(|\lambda|^{1/2}+A)}{(|\lambda|^{1/2}+A^{1/4})^{4}}, \notag\\
&\left|\frac{\Phi_j^{\Fa,\Fb}(\xi',\lambda)}{AB_\Fb(B_\Fb+A)L(A,\lambda)}\right|
\leq \frac{CA}{(|\lambda|^{1/2}+A^{1/4})^{4}};\label{0223-1}
\end{align}
for the velocity
\begin{equation}
\left|\frac{\CI_{m\pm}(\xi',\lambda)}{L(A,\lambda)}\right|
\leq \frac{CA(|\lambda|^{1/2}+A) }{(|\lambda|^{1/2}+A^{1/4})^{4}}, 
\quad \left|\frac{\CJ_{m}(\xi',\lambda)}{E L(A,\lambda)}\right|
\leq \frac{CA}{(|\lambda|^{1/2}+A^{1/4})^{4}}, \label{0223-2}
\end{equation}
and also
\begin{align}
\left|\frac{\Phi_j^{\Fa,\Fb}(\xi',\lambda)\CI_{m\pm}(\xi',\lambda)}
{A(B_\Fb+A)F(A,\lambda)L(A,\lambda)}\right| &\leq CA, \notag \\ 
\left|\frac{\Psi_j^{\Fa,\Fb}(\xi',\lambda)\CI_{m\pm}(\xi',\lambda)}{AB_\Fb(B_\Fb+A)F(A,\lambda)L(A,\lambda)}\right|
&\leq \frac{CA}{|\lambda|^{1/2}+A}, \notag \\
\left|\frac{\Phi_j^{\Fa,\Fb}(\xi',\lambda)\CJ_{m}(\xi',\lambda)}{A(B_\Fb+A)F(A,\lambda)E L(A,\lambda)}\right|
&\leq  \frac{CA}{|\lambda|^{1/2}+A}, \notag \\
\left|\frac{\Psi_j^{\Fa,\Fb}(\xi',\lambda)\CJ_m(\xi',\lambda)}{AB_\Fb(B_\Fb+A)F(A,\lambda)E L(A,\lambda)}\right|
&\leq \frac{CA}{(|\lambda|^{1/2}+A)^{2}}.\label{0223-3}
\end{align}

We now prove

\begin{lem}\label{lem:gam-4}
Let $1\leq p < 2 \leq q\leq  \infty$, $\langle t\rangle=t+1$, and $\Fa,\Fb\in\{+,-\}$.
Suppose that there exists a constant $A_8\in(0,A_3)$ such that for any $A\in(0,A_8)$
and $\lambda\in\wht \Gamma_4^+\cup\wht \Gamma_4^-$
\begin{alignat*}{2}
&|k(\xi',\lambda)|\leq \frac{C}{(|\lambda|^{1/2}+A^{1/4})^{2}}, 
&&\quad |k_{\CM}(\xi',\lambda)|\leq \frac{CA(|\lambda|^{1/2}+A)}{(|\lambda|^{1/2}+A^{1/4})^{4}},  \\
&|k_B(\xi',\lambda)|\leq \frac{CA}{(|\lambda|^{1/2}+A^{1/4})^{4}}, 
&& \quad |l_\CM(\xi',\lambda)| \leq \frac{CA(|\lambda|^{1/2}+A)}{(|\lambda|^{1/2}+A^{1/4})^{4}}, \\
&|l_B(\xi',\lambda)|  \leq \frac{CA}{(|\lambda|^{1/2}+A^{1/4})^{4}}, 
&& \quad |l_{\CM\CM}(\xi',\lambda)| \leq CA, \\
&|l_{\CM B}(\xi',\lambda)| \leq \frac{CA}{|\lambda|^{1/2}+A}, 
&& \quad |l_{B\CM}(\xi',\lambda)| \leq \frac{CA}{|\lambda|^{1/2}+A},  \\
&|l_{BB}(\xi',\lambda)| \leq \frac{CA}{(|\lambda|^{1/2}+A)^2},
\end{alignat*}
with some positive constant of $\xi'$ and $\lambda$.
Then there exists a constant $A_0\in(0,A_8)$ such that 
the following assertions hold.
\begin{enumerate}[$(1)$]
\item
For any $t>0$ and $(d,\Bf)\in Y_p$
\begin{align*}
\|K_{A_0}(t; \Gamma_4^\pm)d\|_{L_q(\BR^{N-1})}
&\leq C \langle t\rangle^{-\frac{N-1}{2}\left(\frac{1}{p}-\frac{1}{q}\right)-\frac{3}{4}\gamma_1}\|d\|_{L_p(\BR^{N-1})}, \\
\|K_{A_0,\CM}^{\Fa,\Fb}(t;\Gamma_4^\pm)\Bf\|_{L_q(\BR^{N-1})}
&\leq C \langle t\rangle^{-\frac{N-1}{2}\left(\frac{1}{p}-\frac{1}{q}\right)
-\frac{1}{2}\left(\frac{1}{p}-\frac{1}{2}\right)-\frac{3}{4}\gamma_2}
\|\Bf\|_{L_p(\BR_\Fa^N)}, \\
\|K_{A_0,B}^{\Fa,\Fb}(t;\Gamma_4^\pm)\Bf\|_{L_q(\BR^{N-1})}
&\leq C \langle t\rangle^{-\frac{N-1}{2}\left(\frac{1}{p}-\frac{1}{q}\right)-\frac{1}{2}\left(\frac{1}{p}-\frac{1}{2}\right)-\frac{3}{4}}
\|\Bf\|_{L_p(\BR_\Fa^N)} ,
\end{align*}
with some positive constant $C$ independent of $t$, $d$, and $\Bf$. 
\item
Let $\Gamma=\Gamma_4^+$ or $\Gamma=\Gamma_4^-$.
Then for any $t>0$ and $(d,\Bf)\in Y_p$
\begin{align*}
\|L_{A_0,\CM}^{\pm}(t;\Gamma)d\|_{L_q(\BR_\pm^N)}
&\leq C\langle t\rangle^{-\frac{N-1}{2}\left(\frac{1}{p}-\frac{1}{q}\right)
-\frac{1}{2}\left(\frac{1}{2}-\frac{1}{q}\right)-\frac{3}{4}\gamma_3}\|d\|_{L_p(\BR^{N-1})}, \\
\|L_{A_0,B}^{\pm}(t;\Gamma)d\|_{L_q(\BR_\pm^N)}
&\leq C\langle t\rangle^{-\frac{N-1}{2}\left(\frac{1}{p}-\frac{1}{q}\right)
-\frac{1}{2}\left(\frac{1}{2}-\frac{1}{q}\right)-\frac{3}{4}}\|d\|_{L_p(\BR^{N-1})}, \\
\|L_{A_0,\CM\CM}^{\pm,\Fa,\Fb}(t;\Gamma)\Bf\|_{L_q(\BR_\pm^N)}
&\leq C\langle t\rangle^{-\frac{N}{2}\left(\frac{1}{p}-\frac{1}{q}\right)}\|\Bf\|_{L_p(\BR_\Fa^N)}, \\
\|L_{A_0,\CM B}^{\pm,\Fa,\Fb}(t;\Gamma)\Bf\|_{L_q(\dot\BR^N)}
&\leq C\langle t\rangle^{-\frac{N}{2}\left(\frac{1}{p}-\frac{1}{q}\right)}\|\Bf\|_{L_p(\BR_\Fa^N)}, \\
\|L_{A_0,B\CM}^{\pm,\Fa,\Fb}(t;\Gamma)\Bf\|_{L_q(\BR_\pm^N)}&
\leq C\langle t\rangle^{-\frac{N}{2}\left(\frac{1}{p}-\frac{1}{q}\right)}\|\Bf\|_{L_p(\BR_\Fa^N)}, \\
\|L_{A_0,BB}^{\pm,\Fa,\Fb}(t;\Gamma)\Bf\|_{L_q(\BR_\pm^N)}&
\leq C\langle t\rangle^{-\frac{N}{2}\left(\frac{1}{p}-\frac{1}{q}\right)}\|\Bf\|_{L_p(\BR_\Fa^N)},
\end{align*}
with some positive constant $C$ independent of $t$, $d$, and $\Bf$.
\end{enumerate}
\end{lem}

\begin{proof}
We here consider $K_{A_0}(t;\Gamma_4^+)$, $K_{A_0,\CM}^{+,+}(t;\Gamma_4^+)$,
and $L_{A_0,\CM}^+(t;\Gamma_4^+)$ only.
The desired estimates for 
\begin{equation*}
L_{A_0,B}^\pm(t;\Gamma), \quad
L_{A_0,\CM\CM}^{\pm,\Fa,\Fb}(t;\Gamma), \quad
L_{A_0,\CM B}^{\pm,\Fa,\Fb}(t;\Gamma), \quad
L_{A_0,B\CM}^{\pm,\Fa,\Fb}(t;\Gamma),\quad
L_{A_0,BB}^{\pm,\Fa,\Fb}(t;\Gamma)
\end{equation*}
are proved in \cite[Lemma 4.13]{SaS1},
and $K_{A_0,B}^{\Fa,\Fb}(t;\Gamma_4^\pm)$ can be proved similarly to the case of $L_{A_0,B}^\pm(t;\Gamma)$.

{\bf Case 1}: $K_{A_0}(t;\Gamma_4^+)$.
Since $\lambda=z_1^+(1-s)+z_3^+s$ for $0\leq s \leq 1$, there holds
\begin{align*}
&[K_{A_0}(t,\Gamma_4^+)d](x') \\
&=\CF_{\xi'}^{-1}
\left[e^{-z_0 A^2 \langle t\rangle/8}\cdot e^{z_0 A^2 \langle t\rangle/8 }\frac{\varphi_{A_0}(\xi')}{2\pi i}
\int_0^1 e^{\lambda t}
 k(\xi',\lambda)(z_3^+-z_1^+)\,ds \, \wht d(\xi')\right](x').
\end{align*}
It thus holds that by Lemma \ref{lem:Lp-Lq}, Parseval's identity, 
and the assumption for $k(\xi',\lambda)$ 
\begin{align*}
&\|K_{A_0}(t,\Gamma_4^+)d\|_{L_q(\BR^{N-1})} \\
&\leq C\langle t \rangle^{-\frac{N-1}{2}\left(\frac{1}{2}-\frac{1}{q}\right)}
\int_0^1\Big\|\varphi_{A_0}(\xi')\frac{e^{z_0A^2 \langle t\rangle/8 }e^{(\Re\lambda) t}}{(|\lambda|^{1/2}+A^{1/4})^2}
\wht d(\xi')\Big\|_{L_2(\BR^{N-1})}ds
=:I_1(t).
\end{align*}
We choose a sufficiently small $A_0\in(0,1)$ so that 
\begin{equation*}
e^{z_0A^2\langle t\rangle/8}e^{(\Re\lambda) t}\leq 
C e^{-z_0A^2\langle t\rangle/8}e^{-cs\langle t\rangle} \quad \text{on $\spp\varphi_{A_0}$}
\end{equation*}
for positive constants $C$ and $c$.
Then 
\begin{equation}\label{ineq-1-gam4}
I_1(t)\leq
C\langle t \rangle^{-\frac{N-1}{2}\left(\frac{1}{2}-\frac{1}{q}\right)}
\int_0^1e^{-cs\langle t\rangle} 
\Big\|\varphi_{A_0}(\xi')\frac{e^{-z_0A^2\langle t\rangle/8}}{(|\lambda|^{1/2}+A^{1/4})^2}
\wht d(\xi')\Big\|_{L_2(\BR^{N-1})}ds.
\end{equation}

Let $0<\delta<2$. Since 
\begin{equation*}
|\lambda|^{1/2}\geq C(A\sqrt{1-s}+\sqrt{s}) \quad \text{for $s\in [0,1]$,}
\end{equation*}
we see that 
\begin{equation}\label{est-lambdahalf}
(|\lambda|^{1/2}+A^{1/4})^2\geq (|\lambda|^{1/2})^{2-\delta}A^{\delta/4}\geq C(\sqrt{s})^{2-\delta}A^{\delta/4}.
\end{equation}
Combining this inequality with \eqref{ineq-1-gam4} furnishes
\begin{align*}
I_1(t)
&\leq C\langle t \rangle^{-\frac{N-1}{2}\left(\frac{1}{2}-\frac{1}{q}\right)}
\int_0^1 \frac{e^{-cs\langle t\rangle}}{\sqrt{s}^{\,2-\delta}}\,ds
\cdot \big\|e^{-z_0A^2\langle t\rangle/8}A^{-\delta/4}\wht d(\xi')\big\|_{L_2(\BR^{N-1})} \\
&\leq C\langle t \rangle^{-\frac{N-1}{2}\left(\frac{1}{2}-\frac{1}{q}\right)-\frac{\delta}{2}}
\big\|e^{-z_0A^2\langle t\rangle/8}A^{-\delta/4}\wht d(\xi')\big\|_{L_2(\BR^{N-1})},
\end{align*}
which, combined with Parseval's identity and Young's inequality, yields
\begin{align*}
I_1(t)
&\leq 
C\langle t \rangle^{-\frac{N-1}{2}\left(\frac{1}{2}-\frac{1}{q}\right)-\frac{\delta}{2}}
\|\CF_{\xi'}^{-1}[e^{-z_0A^2\langle t\rangle/8}A^{-\delta/4}\wht d(\xi')]\big\|_{L_2(\BR^{N-1})} \\
&\leq C\langle t \rangle^{-\frac{N-1}{2}\left(\frac{1}{2}-\frac{1}{q}\right)-\frac{\delta}{2}}
\|\CF_{\xi'}^{-1}[e^{-z_0A^2\langle t\rangle/8}A^{-\delta/4}]\|_{L_r(\BR^{N-1})}\|d\|_{L_p(\BR^{N-1})},
\end{align*}
where $1+(1/2)=(1/p)+(1/r)$.

From now on, we estimate $J_1(t):=\|\CF_{\xi'}^{-1}[e^{-z_0A^2\langle t\rangle/8}A^{-\delta/4}]\|_{L_r(\BR^{N-1})}$
by Lemma \ref{lem:SS01}.
By the Leibniz rule and Lemmas \ref{lem:fund-sym} and \ref{fundlem:D},
 we have for any multi-index $\alpha'\in\BN_0^{N-1}$ and $\xi'\in\BR^{N-1}\setminus\{0\}$
\begin{equation*}
|\pd_{\xi'}^{\alpha'}(e^{-z_0A^2\langle t\rangle/8}A^{-\delta/4})|\leq C_{\alpha'}A^{-(\delta/4)-|\alpha'|}
e^{-z_0A^2\langle t\rangle/16},
\end{equation*}
where $C_{\alpha'}$ is a positive constant independent of $\xi'$ and $t$.
Lemma \ref{lem:SS01} with $\sigma=1-(\delta/4)$, $L=N-2$, and $n=N-1$ 
then furnishes
\begin{equation*}
|\CF_{\xi'}^{-1}[e^{-z_0A^2\langle t\rangle/8}A^{-\delta/4}](x')|\leq C|x'|^{-(N-1-(\delta/4))} \quad (x'\in\BR^{N-1}\setminus\{0\}).
\end{equation*}
By direct calculations, we also have
\begin{equation*}
|\CF_{\xi'}^{-1}[e^{-z_0A^2\langle t\rangle/8}A^{-\delta/4}](x')|
\leq C\int_{\BR^{N-1}}e^{-z_0A^2\langle t\rangle/8}A^{-\delta/4}\,d\xi'
\leq C\langle t\rangle^{-(N-1-(\delta/4))/2},
\end{equation*}
and thus we obtain by these two inequalities
\begin{equation*}
|\CF_{\xi'}^{-1}[e^{-z_0A^2\langle t\rangle/8}A^{-\delta/4}](x')| 
\leq \frac{C}{\langle t\rangle^{(N-1-(\delta/4))/2}+|x'|^{(N-1-(\delta/4))}}.
\end{equation*}

Let us choose the above $\delta$ so that 
\begin{equation*}
0<\delta<\min\left\{2,4(N-1)\left(\frac{1}{p}-\frac{1}{2}\right)\right\}.
\end{equation*}
Then we have
\begin{equation*}
J_1(t)\leq C\langle t\rangle^{-\frac{1}{2}(N-1-\frac{\delta}{4})+\frac{N-1}{2r}}
 =C\langle t\rangle^{-\frac{N-1}{2}\left(\frac{1}{p}-\frac{1}{2}\right)+\frac{\delta}{8}}.
\end{equation*}
Hence
\begin{equation*}
I_1(t)\leq C\langle t\rangle^{-\frac{N-1}{2}\left(\frac{1}{p}-\frac{1}{q}\right)
-\frac{3}{8}\delta}\|d\|_{L_p(\BR^{N-1})},
\end{equation*}
which implies the desired estimate for $K_{A_0}(t;\Gamma_4^+)$ holds.

{\bf Case 2}: $K_{A_0,\CM}^{+,+}(t;\Gamma_4^+)$.
In the same way as we have obtained \eqref{ineq-1-gam4}, we obtain
\begin{align*}
&\|K_{A_0,\CM}^{+,+}(t;\Gamma_4^+)\Bf\|_{L_q(\BR^{N-1})} 
\leq C\langle t\rangle^{-\frac{N-1}{2}\left(\frac{1}{2}-\frac{1}{q}\right)} \\
&\times \int_0^\infty\int_0^1 e^{-cs\langle t\rangle} \Big\|
\varphi_{A_0}(\xi')e^{-z_0A^2\langle t\rangle/8}
k_\CM(\xi',\lambda)\CM_+(y_N)\wht \Bf(\xi',y_N)\Big\|_{L_2(\BR^{N-1})}dsdy_N  \\
&=:I_2(t).
\end{align*}
Since it holds by Lemma \ref{lem:fund-sym} that
  $C_1(|\lambda|^{1/2}+A)\leq|B_++A|\leq C_2|\lambda|^{1/2}$
for $A\in(0,1)$ and $\lambda\in\wht \Gamma_4^+$,
we see that
\begin{equation*}
\left|\frac{1}{B_+-A}\right|=\left|\frac{B_++A}{B_+^2-A^2}\right|\leq \frac{C}{|\lambda|^{1/2}}\leq \frac{\wtd C}{|\lambda|^{1/2}+A}.
\end{equation*}
Therefore for $A\in(0,1)$ and $\lambda\in\wht \Gamma_4^+$
\begin{equation*}
|\CM_+(y_N)|\leq \frac{Ce^{-cA y_N}}{|\lambda|^{1/2}+A} ,
\end{equation*}
with positive constants $C$ and $c$, which, combined with 
 the assumption for $k_\CM(\xi',\lambda)$ and \eqref{est-lambdahalf}, furnishes
\begin{equation*}
|k_\CM(\xi',\lambda)\CM_+(y_N)|\leq C\frac{A^{1/2}e^{-cA y_N}}{(|\lambda|^{1/2}+A^{1/4})^2} 
\leq C\frac{A^{1/2}e^{-cA y_N}}{(\sqrt{s})^{2-\delta}A^{\delta/4}} \quad (0<\delta<2).
\end{equation*}

One now sees that
\begin{align}\label{ineq-I2-gam4}
I_2(t)&\leq 
C\langle t\rangle^{-\frac{N-1}{2}\left(\frac{1}{2}-\frac{1}{q}\right)} 
\int_0^1 \frac{e^{-cs\langle t\rangle}}{(\sqrt{s})^{2-\delta}}\,ds \notag \\
&\times\int_0^\infty \|e^{-z_0A^2\langle t\rangle/8}A^{1/2-\delta/4}e^{-cAy_N}\wht\Bf(\xi',y_N)\|_{L_2(\BR^{N-1})}\,dy_N \notag \\
&\leq 
C\langle t\rangle^{-\frac{N-1}{2}\left(\frac{1}{2}-\frac{1}{q}\right)-\frac{\delta}{2}}
\int_0^\infty \|e^{-z_0A^2\langle t\rangle/8}A^{1/2-\delta/4}e^{-cAy_N}\wht\Bf(\xi',y_N)\|_{L_2(\BR^{N-1})}\,dy_N,
\end{align}
and thus for $1+(1/2)=(1/p)+(1/r)$
\begin{align*}
I_2(t)
&\leq C\langle t\rangle^{-\frac{N-1}{2}\left(\frac{1}{2}-\frac{1}{q}\right)-\frac{\delta}{2}} \\
&\times\int_0^\infty\|\CF_{\xi'}^{-1}[e^{-z_0A^2\langle t\rangle/8}A^{1/2-\delta/4}e^{-cAy_N}]\|_{L_r(\BR^{N-1})}
\|\Bf(\cdot,y_N)\|_{L_p(\BR^{N-1})}\,dy_N \\
&\leq C\langle t\rangle^{-\frac{N-1}{2}\left(\frac{1}{2}-\frac{1}{q}\right)-\frac{\delta}{2}} J_2(t)\|\Bf\|_{L_p(\BR_+^N)},
\end{align*}
where for $p'=p/(p-1)$
\begin{align*}
&J_2(t)= \\
&\left\{\begin{aligned}
&\left(\int_0^\infty\|\CF_{\xi'}^{-1}[e^{-z_0A^2\langle t\rangle/8}A^{1/2-\delta/4}e^{-cAy_N}]\|_{L_r(\BR^{N-1})}^{p'}dy_N\right)^{1/p'}
&& (1<p<2), \\
&\sup_{y_N>0}\|\CF_{\xi'}^{-1}[e^{-z_0A^2\langle t\rangle/8}A^{1/2-\delta/4}e^{-cAy_N}]\|_{L_r(\BR^{N-1})}
&&(p=1).
\end{aligned}\right.
\end{align*}
By the Leibniz rule and Lemmas \ref{lem:fund-sym} and \ref{fundlem:D},
 we have for any multi-index $\alpha'\in\BN_0^{N-1}$ and $\xi'\in\BR^{N-1}\setminus\{0\}$
\begin{equation*}
|\pd_{\xi'}^{\alpha'}(e^{-z_0A^2\langle t\rangle/8}A^{1/2-\delta/4}e^{-cAy_N})|\leq C_{\alpha'}A^{1/2-\delta/4-|\alpha'|}
e^{-z_0A^2\langle t\rangle/16},
\end{equation*}
where $C_{\alpha'}$ is a positive constant independent of $\xi'$ and $t$.
Lemma \ref{lem:SS01} with $\sigma=1/2-\delta/4$, $L=N-1$, and $n=N-1$ then furnishes 
\begin{equation*}
|\CF_{\xi'}^{-1}[e^{-z_0A^2\langle t\rangle/8}A^{1/2-\delta/4}e^{-cAy_N}](x')|
\leq C|x'|^{-(N-1+\sigma)} \quad (x'\in\BR^{N-1}\setminus\{0\}).
\end{equation*}
By direct calculations, we also have
\begin{align*}
&|\CF_{\xi'}^{-1}[e^{-z_0A^2\langle t\rangle/8}A^{1/2-\delta/4}e^{-cAy_N}](x')| \\
&\leq C\int_{\BR^{N-1}}A^{1/2-\delta/4}e^{-cAy_N} \,d\xi'\leq C y_N^{-(N-1+\sigma)} \quad (y_N>0), \\
&|\CF_{\xi'}^{-1}[e^{-z_0A^2\langle t\rangle/8}A^{1/2-\delta/4}e^{-cAy_N}](x')| \\
&\leq C\int_{\BR^{N-1}}e^{-z_0A^2\langle t\rangle/8}A^{1/2-\delta/4}\,d\xi'
\leq C\langle t\rangle^{-\frac{1}{2}(N-1+\sigma)} \quad (t>0).
\end{align*}
Combining these three inequalities yields
\begin{equation*}
|\CF_{\xi'}^{-1}[e^{-z_0A^2\langle t\rangle/8}A^{1/2-\delta/4}e^{-cAy_N}](x')|
\leq \frac{C}{|x'|^{N-1+\sigma}+y_N^{N-1+\sigma}+\langle t\rangle^{(N-1+\sigma)/2}},
\end{equation*}
which implies for any $\delta\in(0,2)$
\begin{equation*}
\|\CF_{\xi'}^{-1}[e^{-z_0A^2\langle t\rangle/8}A^{1/2-\delta/4}e^{-cAy_N}]\|_{L_r(\BR^{N-1})}
\leq C(y_N+\langle t\rangle^{1/2})^{-(N-1)\left(\frac{1}{p}-\frac{1}{2}\right)-\left(\frac{1}{2}-\frac{\delta}{4}\right)}.
\end{equation*}

Let us choose the above $\delta$ so that
\begin{equation*}
0<\delta<\min\left\{2,4N\left(\frac{1}{p}-\frac{1}{2}\right)\right\}.
\end{equation*}
Then we have
\begin{equation*}
J_2(t)\leq C\langle t\rangle^{-\frac{N-1}{2}\left(\frac{1}{p}-\frac{1}{2}\right)-\frac{1}{2}\left(\frac{1}{p}-\frac{1}{2}\right)+\frac{\delta}{8}}.
\end{equation*}
Hence
\begin{equation*}
I_2(t)\leq C\langle t\rangle^{-\frac{N-1}{2}\left(\frac{1}{p}-\frac{1}{q}\right)-\frac{1}{2}\left(\frac{1}{p}-\frac{1}{2}\right)-\frac{3}{8}\delta}
\|\Bf\|_{L_p(\BR_+^N)},
\end{equation*}
which implies the desired estimate for $K_{A_0,\CM}^{+,+}(t;\Gamma_4^+)$ holds.

{\bf Case 3}: $L_{A_0,\CM}^+(t;\Gamma_4^+)$.
In the same way as we have obtained \eqref{ineq-I2-gam4}, we obtain for $\delta\in(0,2)$
\begin{align*}
&\|[L_{A_0,\CM}^+(t;\Gamma_4^+)d](\cdot,x_N)\|_{L_q(\BR^{N-1})} \\
&\leq C
\langle t\rangle^{-\frac{N-1}{2}\left(\frac{1}{2}-\frac{1}{q}\right)-\frac{\delta}{2}}
\|e^{-z_0A^2\langle t\rangle/8}A^{1/2-\delta/4}e^{-cAx_N}\wht d(\xi')\|_{L_2(\BR^{N-1})}=:I_3(x_N,t)
\end{align*}
By Parseval's identity and Young's inequality,
\begin{align*}
&I_3(x_N,t) \\
&\leq C\langle t\rangle^{-\frac{N-1}{2}\left(\frac{1}{2}-\frac{1}{q}\right)-\frac{\delta}{2}}
\|\CF_{\xi'}^{-1}[e^{-z_0A^2\langle t\rangle/8}A^{1/2-\delta/4}e^{-cAx_N}]\|_{L_r(\BR^{N-1})}\|d\|_{L_p(\BR^{N-1})},
\end{align*}
where $1+(1/2)=(1/p)+(1/r)$. Similarly to $J_2(t)$ in Case 2, we observe for any $\delta\in(0,2)$ that
\begin{equation*}
\|\CF_{\xi'}^{-1}[e^{-z_0A^2\langle t\rangle/8}A^{1/2-\delta/4}e^{-cAx_N}]\|_{L_r(\BR^{N-1})}
\leq C(x_N+\langle t\rangle^{1/2})^{-(N-1)\left(\frac{1}{p}-\frac{1}{2}\right)-\left(\frac{1}{2}-\frac{\delta}{4}\right)}.
\end{equation*}
It thus holds that 
\begin{equation*}
\left(\int_0^\infty I_3(x_N,t)^q\,dx_N\right)^{1/q}
\leq C\langle t\rangle^{-\frac{N-1}{2}\left(\frac{1}{p}-\frac{1}{q}\right)-\frac{1}{2}\left(\frac{1}{2}-\frac{1}{q}\right)-\frac{3}{8}\delta}
\|d\|_{L_p(\BR^{N-1})}
\end{equation*}
under the assumption 
\begin{equation*}
0<\delta<\min\left\{2,4\left((N-1)\left(\frac{1}{p}-\frac{1}{2}\right)+\frac{1}{2}-\frac{1}{q}\right)\right\}.
\end{equation*}
This implies the desired estimate for $L_{A_0,\CM}^+(t;\Gamma_4^+)$,
which completes the proof of Lemma \ref{lem:gam-4}.
\end{proof}

Combining Lemma \ref{lem:gam-4} with \eqref{0223-1}, \eqref{0223-2}, and \eqref{0223-3}
yields Theorem \ref{thm:gam-4} immediately.
This completes the proof of Theorem \ref{thm:gam-4}.

\subsection{Analysis for $\Gamma_5^\pm$.}
Similarly to \cite[Subsection 4.4]{SaS1}, we can prove by Lemmas \ref{lem:fund-sym} and \ref{lem:nonzero}
the following theorem.

\begin{thm}\label{thm:gam-5}
Let $1\leq p \leq 2 \leq q \leq \infty$.
Then there exist constants $A_0\in(0,A_3)$ and $c_0>0$ such that for any $t\geq 1$
and $(d,\Bf)\in Y_p$
\begin{align*}
\|\CH_{A_0}^{1}(t;\Gamma_5^\pm)d\|_{L_q(\BR^{N-1})}
&\leq Ce^{-c_0t}\|d\|_{L_p(\BR^{N-1})}, \\
\|\CH_{A_0}^{2}(t;\Gamma_5^\pm)\Bf\|_{L_q(\BR^{N-1})}
&\leq Ce^{-c_0t}\|\Bf\|_{L_p(\dot\BR^{N})}, \\
\|\CU_{A_0}^{1}(t;\Gamma_5^\pm)d\|_{L_q(\dot\BR^{N})}
&\leq Ce^{-c_0t}\|d\|_{L_p(\BR^{N-1})}, \\
\|\CU_{A_0}^{2}(t;\Gamma_5^\pm)\Bf\|_{L_q(\dot\BR^{N})}
&\leq Ce^{-c_0t}\|\Bf\|_{L_p(\dot\BR^{N})}, 
\end{align*}
where $C$ is a positive constant independent of $t$, $d$, and $\Bf$. 
\end{thm}

\subsection{Proof of Theorem \ref{thm:main3}.}

Recalling \eqref{dfn:u-h}, we observe for $\CS\in\{\CH,\CU\}$ that
 the operators in Theorem \ref{thm:main3} are given by 
\begin{alignat*}{2}
\CS_{A_0}^{1\pm}(t)d&=\CS_{A_0}^1(t;\Gamma_{\rm res}^\pm)d, 
\quad &\CS_{A_0}^{2\pm}(t)\Bf &=\CS_{A_0}^2(t;\Gamma_{\rm res}^\pm)\Bf, \\
\wtd\CS_{A_0}^{1}(t)d&=\sum_{\Fa\in\{+,-\}}\sum_{j\in\{1,4,5\}}\CS_{A_0}^1(t;\Gamma_j^\Fa)d, 
\quad &\wtd\CS_{A_0}^{2}(t)\Bf&=\sum_{\Fa\in\{+,-\}}\sum_{j\in\{1,4,5\}}\CS_{A_0}^2(t;\Gamma_j^\Fa)\Bf.
\end{alignat*}
Theorems \ref{thm:gam_res}, \ref{gam_1-decay}, \ref{thm:gam-4}, and \ref{thm:gam-5} then yields Theorem \ref{thm:main3} immediately.
This completes the proof of Theorem \ref{thm:main3}.

\section{Time-decay estimates for high frequency part}
This section proves Theorem \ref{thm:main4}.
Suppose $\rho_->\rho_+>0$ throughout this section.

Let us denote the points of intersection between $\lambda=-1+is$ $(s\in\BR)$ and $\Gamma_0^\pm$ given in \eqref{dfn:gamma_0}
by $z_4^\pm$,
and let $A_0$ be the positive constant given in Theorem \ref{thm:main3}.
We define $A_\infty=A_{\rm high}(1,\Im z_4^+)$ for the positive constant $A_{\rm high}$ given in Proposition \ref{prp:high}.
In addition, we set $M_1=A_0/2$ and $M_2=3A_\infty$ in Proposition \ref{prp:mid}.
Then we have $a_0\in(0,1)$ from Proposition \ref{prp:mid} and denote
the points of intersection between $\lambda=-a_0+is$ $(s\in\BR)$ and $\Gamma_0^\pm$ by $z_5^\pm$.
Note that $\Im z_5^-=-\Im z_5^+$.

Now we define integral paths $\wht \Gamma_6$ and $\wht\Gamma_7$ as follows:
\begin{align*}
\wht\Gamma_6
&=\{\lambda\in\BC : \lambda=-a_0+si, -\Im z_5^+ \leq s \leq \Im z_5^+\}, \\
\wht \Gamma_7
&=\{\lambda\in\BC : \lambda=-a_0+ i \Im z_5^+  +se^{i(\pi-\theta_1)}, s\geq 0\} \\
&\cup \{\lambda\in\BC : \lambda=-a_0-i\Im z_5^+ +se^{-i(\pi-\theta_1)}, s\geq 0\},
\end{align*}
Then Propositions \ref{prp:high} and \ref{prp:mid} yield the following lemma.

\begin{lem}\label{lem:highest}
\begin{enumerate}[$(1)$]
\item
$\Re B_\pm>0$, $F(A,\lambda)\neq 0$, and $L(A,\lambda)\neq 0$ when $A\geq A_0/2$ and $\lambda\in\Lambda(a_0,\Im z_5^+)$,
where $\Lambda(a_0,\Im z_5^+)$ is given in \eqref{cap:lam}.
\item
Let $s\in\BR$ and $\alpha'\in\BN_0^{N-1}$.
Then for any  $\lambda\in\wht \Gamma_6$
it holds that on $\spp\varphi_{A_\infty}\cup\spp\varphi_{[A_0,A_\infty]}$
\begin{align*}
&|\pd_\xi' B_\pm^s|\leq  CA^{s-|\alpha'|}, \quad |\pd_{\xi'}^{\alpha'}E^s|\leq C A^{s-|\alpha'|}, \quad
|\pd_{\xi'}^{\alpha'}(A+B_\pm)^s|\leq CA^{s-|\alpha'|}, \\
&|\pd_{\xi'}^{\alpha'}F(A,\lambda)^s|\leq CA^{3s-|\alpha'|}, \quad
|\pd_{\xi'}^{\alpha'}L(A,\lambda)^{-1}|\leq CA^{-4-|\alpha'|}.
\end{align*}
\end{enumerate}
\end{lem}

Let $\CS\in\{\CH,\CU\}$ and $Z\in\{A_\infty,[A_0,A_\infty]\}$.
Recalling \eqref{eq:sol-main}, we have
by Lemma \ref{lem:highest} (1) and Cauchy's integral theorem
\begin{alignat*}{2}
\CS_Z^1(t)d&=\sum_{j=6}^7\CS_Z^1(t;\Gamma_j)d, \quad 
&\CS_Z^1(t;\Gamma_j)d&:=\CF_{\xi'}^{-1}[\wht \CS_Z^1(t;\wht\Gamma_j)d](x'), \\
\CS_Z^2(t)\Bf&=\sum_{j=6}^7\CS_Z^2(t;\Gamma_j)\Bf, \quad 
&\CS_Z^2(t;\Gamma_j)\Bf&:=\CF_{\xi'}^{-1}[\wht \CS_Z^2(t;\wht\Gamma_j)\Bf](x').
\end{alignat*}
Similarly to \cite[Section 5]{SaS1}, we can prove from Lemmas \ref{lem:highest}, \ref{lem:fund-sym}, and \ref{lem:nonzero}
the following theorem by choosing $A_\infty$ larger if necessary.

\begin{thm}\label{thm:high-path}
Let $q\in(1,\infty)$ and $Z\in\{A_\infty,[A_0,A_\infty]\}$.
Let $j=6,7$ and $k\in\BN_0$.
Then for any $t\geq 1$ and $(d,\Bf)\in X_q$
\begin{align*}
\|(\pd_t^k\CH_Z^1(t;\Gamma_j)d,\pd_t^k\CH_Z^2(t;\Gamma_j)\Bf)\|_{W_q^{3-1/q}(\BR^{N-1})}
&\leq C e^{-ct}\|(d,\Bf)\|_{X_q}, \\
\|(\pd_t^k\CU_Z^1(t;\Gamma_j)d,\pd_t^k\CU_Z^2(t;\Gamma_j)\Bf)\|_{H_q^2(\dot\BR^{N})}
&\leq C e^{-ct}\|(d,\Bf)\|_{X_q},
\end{align*}
where $C$ and $c$ are positive constants independent of $t$, $d$, and $\Bf$.
\end{thm}

Theorem \ref{thm:high-path} yields Theorem \ref{thm:main4} immediately.
This completes the proof of Theorem \ref{thm:main4}.

\def\thesection{A}
\renewcommand{\theequation}{A.\arabic{equation}}
\section{}
In this appendix, we consider the whole space problems in \eqref{eq:1-A}
and compute the representation formulas of  $\wht \psi_{\pm}|_{x_N=0}$ and $\pd_N\wht \psi_{\pm}|_{x_N=0}$.
We define the Fourier transform of $f=f(x)$ and 
the inverse Fourier transform of $g=g(\xi)$ by
\begin{equation*}
\CF[f](\xi)=\int_{\BR^N} e^{-ix\cdot\xi}f(x)\,dx, \quad 
\CF_{\xi}^{-1}[g](x)=\frac{1}{(2\pi)^N}\int_{\BR^N}e^{ix\cdot \xi}g(\xi)\,d\xi.
\end{equation*}

Let us denote the $j$th component of $\BF$ by $F_j$.
In \cite[Section 2]{SS11b}, we have 
\begin{align*}
\psi_\pm&=\CF_{\xi}^{-1}\left[\frac{\CF[\BF](\xi)}{\rho_\pm\lambda+\mu_\pm|\xi|^2}\right](x)
+\sum_{j=1}^{N}\CF_{\xi}^{-1}\left[\frac{(i\xi)(i\xi_j)\CF[F_j](\xi)}{|\xi|^2(\rho_\pm\lambda+\mu_\pm|\xi|^2)}\right](x), \\
\phi_+&=\phi_-=-\sum_{j=1}^{N}\CF_\xi^{-1}\left[\frac{i\xi_j\CF[F_j](\xi)}{|\xi|^2}\right](x).
\end{align*}
The formulas of $\phi_\pm$ imply $\jmp{\phi}=\phi_+(x',0\,+)-\phi_-(x',0\,-)=0$.
For $k=1,\dots,N-1$, $\psi_{k\pm}$ can be written as 
\begin{align*}
\psi_{k\pm}
&=\CF_{\xi}^{-1}\left[\frac{\CF[F_k](\xi)}{\rho_\pm\lambda+\mu_\pm|\xi|^2}\right](x)
-\sum_{j=1}^{N-1}\CF_{\xi}^{-1}\left[\frac{\xi_k \xi_j\CF[F_j](\xi)}{|\xi|^2(\rho_\pm\lambda+\mu_\pm|\xi|^2)}\right](x) \\
&+\CF_{\xi}^{-1}\left[\frac{(i\xi_k)(i\xi_N)\CF[F_N](\xi)}{|\xi|^2(\rho_\pm\lambda+\mu_\pm|\xi|^2)}\right](x),
\end{align*}
and also
\begin{equation*}
\psi_{N\pm}
=\CF_{\xi}^{-1}\left[\frac{A^2\CF[F_N](\xi)}{|\xi|^2(\rho_\pm\lambda+\mu_\pm|\xi|^2)}\right](x)
+\sum_{j=1}^{N-1}\CF_{\xi}^{-1}\left[\frac{(i\xi_N)(i\xi_j)\CF[F_j](\xi)}{|\xi|^2(\rho_\pm\lambda+\mu_\pm|\xi|^2)}\right](x),
\end{equation*}
where $A$ is given in \eqref{dfn:AB}. It holds for $l=1,\dots,N$ that
\begin{align*}
\CF[F_l](\xi)&=\int_0^\infty e^{-i y_N\xi_N} \wht F_l(\xi',y_N)\intd y_N
+\int_0^\infty e^{iy_N\xi_N}\wht F_l(\xi',-y_N) \intd y_N \\
&=\sum_{\Fa\in\{+,-\}}\int_0^\infty e^{-i\Fa y_N\xi_N}\wht F_l(\xi',\Fa y_N) \intd y_N
\end{align*}
which, inserted into the above formulas of $\psi_{k\pm}$ and $\psi_{N\pm}$, furnishes
\begin{align*}
&\wht \psi_{k\pm}(\xi',x_N,\lambda) \\
&=\sum_{\Fa\in\{+,-\}}
\int_0^\infty\wht F_k(\xi',\Fa y_N)\left(\frac{1}{2\pi}\int_{-\infty}^\infty
\frac{e^{i(x_N-\Fa y_N)\xi_N}}{\rho_\pm\lambda+\mu_\pm|\xi|^2}\intd\xi_N\right)\intd y_N \\
&-\sum_{j=1}^{N-1}\sum_{\Fa\in\{+,-\}}
\int_0^\infty \xi_k\xi_j \wht F_j(\xi',\Fa y_N)
\left(\frac{1}{2\pi}\int_{-\infty}^\infty\frac{e^{i(x_N-\Fa y_N)\xi_N}}{|\xi|^2(\rho_\pm\lambda+\mu_\pm|\xi|^2)}\intd\xi_N\right)\intd y_N \\
&+\sum_{\Fa\in\{+,-\}}\int_0^\infty i \xi_k \wht F_N(\xi',\Fa y_N)
\left(\frac{1}{2\pi}\int_{-\infty}^\infty\frac{i\xi_N e^{i(x_N-\Fa y_N)\xi_N}}{|\xi|^2(\rho_\pm\lambda+\mu_\pm|\xi|^2)}\intd\xi_N\right)\intd y_N, \\
&\wht \psi_{N\pm}(\xi',x_N,\lambda)  \\
&=\sum_{\Fa\in\{+,-\}}\int_0^\infty A^2 \wht F_N(\xi',\Fa y_N)
\left(\frac{1}{2\pi}\int_{-\infty}^\infty\frac{e^{i(x_N-\Fa y_N)\xi_N}}{|\xi|^2(\rho_\pm\lambda+\mu_\pm|\xi|^2)}\intd\xi_N\right)\intd y_N \\
&+\sum_{j=1}^{N-1}\sum_{\Fa\in\{+,-\}}
\int_0^\infty i\xi_j \wht F_j(\xi',\Fa y_N)
\left(\frac{1}{2\pi}\int_{-\infty}^\infty\frac{i\xi_N e^{i(x_N-\Fa y_N)\xi_N}}{|\xi|^2(\rho_\pm\lambda+\mu_\pm|\xi|^2)}\intd\xi_N\right)\intd y_N.
\end{align*}
Applying $\pd_N$ to these formulas yields
\begin{align*}
&\pd_N\wht \psi_{k\pm}(\xi',x_N,\lambda)  \\
&=\sum_{\Fa\in\{+,-\}}\int_0^\infty\wht F_k(\xi',\Fa y_N)\left(\frac{1}{2\pi}\int_{-\infty}^\infty
\frac{i\xi_Ne^{i(x_N-\Fa y_N)\xi_N}}{\rho_\pm\lambda+\mu_\pm|\xi|^2}\intd\xi_N\right)\intd y_N \\
&-\sum_{j=1}^{N-1}\sum_{\Fa\in\{+,-\}}
\int_0^\infty \xi_k\xi_j \wht F_j(\xi',\Fa y_N)
\left(\frac{1}{2\pi}\int_{-\infty}^\infty\frac{i\xi_N e^{i(x_N-\Fa y_N)\xi_N}}{|\xi|^2(\rho_\pm\lambda+\mu_\pm|\xi|^2)}\intd\xi_N\right)\intd y_N \\
&+\sum_{\Fa\in\{+,-\}}\int_0^\infty i \xi_k \wht F_N(\xi',\Fa y_N)
\left(\frac{1}{2\pi}\int_{-\infty}^\infty\frac{(i\xi_N)^2 e^{i(x_N-\Fa y_N)\xi_N}}{|\xi|^2(\rho_\pm\lambda+\mu_\pm|\xi|^2)}\intd\xi_N\right)\intd y_N, \\
&\pd_N\wht \psi_{N\pm}(\xi',x_N,\lambda) \\
&=\sum_{\Fa\in\{+,-\}}\int_0^\infty A^2 \wht F_N(\xi',\Fa y_N)
\left(\frac{1}{2\pi}\int_{-\infty}^\infty\frac{i\xi_N e^{i(x_N-\Fa y_N)\xi_N}}{|\xi|^2(\rho_\pm\lambda+\mu_\pm|\xi|^2)}\intd\xi_N\right)\intd y_N \\
&+\sum_{j=1}^{N-1}\sum_{\Fa\in\{+,-\}}
\int_0^\infty i \xi_j \wht F_j(\xi',\Fa y_N)
\left(\frac{1}{2\pi}\int_{-\infty}^\infty\frac{(i\xi_N)^2 e^{i(x_N-\Fa y_N)\xi_N}}{|\xi|^2(\rho_\pm\lambda+\mu_\pm|\xi|^2)}\intd\xi_N\right)\intd y_N.
\end{align*}

Now we have

\begin{lem}\label{lem:residue}
Let $\xi'=(\xi_1,\dots,\xi_{N-1})\in\BR^{N-1}\setminus\{0\}$ and $a\in\BR\setminus\{0\}$. 
Then 
\begin{align*}
\frac{1}{2\pi}\int_{-\infty}^\infty \frac{e^{i a \xi_N}}{\rho_\pm\lambda+\mu_\pm|\xi|^2}\intd \xi_N
&=\frac{e^{-B_\pm|a|}}{2\mu_\pm B_\pm}, \\
\frac{1}{2\pi}\int_{-\infty}^\infty \frac{i\xi_N e^{i a \xi_N}}{\rho_\pm\lambda+\mu_\pm|\xi|^2}\intd \xi_N
&=-{\rm sign}(a)\frac{e^{-B_\pm |a|}}{2\mu_\pm}, \\
\frac{1}{2\pi}\int_{-\infty}^\infty \frac{e^{i a \xi_N}}{|\xi|^2(\rho_\pm\lambda+\mu_\pm|\xi|^2)}\intd \xi_N
&=-\frac{1}{2\mu_\pm(A^2-B_\pm^2)}\left(\frac{e^{-A|a|}}{A}-\frac{e^{-B_\pm|a|}}{B_\pm}\right), \\
\frac{1}{2\pi}\int_{-\infty}^\infty \frac{i\xi_N e^{i a \xi_N}}{|\xi|^2(\rho_\pm\lambda+\mu_\pm|\xi|^2)}\intd \xi_N
&=-\frac{{\rm sign}(a)}{2\mu_\pm(A^2-B_\pm^2)}
\left(-e^{-A|a|}+e^{-B_\pm|a|}\right), \\
\frac{1}{2\pi}\int_{-\infty}^\infty \frac{(i\xi_N)^2 e^{i a \xi_N}}{|\xi|^2(\rho_\pm\lambda+\mu_\pm|\xi|^2)}\intd \xi_N
&=-\frac{1}{2\mu_\pm(A^2-B_\pm^2)}\left(Ae^{-A |a|}-B_\pm e^{-B_\pm|a|}\right),
\end{align*}
where $A$ and $B_\pm$ are defined in \eqref{dfn:AB}.
Here ${\rm sign}(a)=1$ when $a>0$ and ${\rm sign}(a)=-1$ when $a<0$. 
\end{lem}

\begin{proof}
The first and third formulas follow from the residue theorem.
Differentiating the first formula with respect to $a$, we have the second formula.
Analogously the fourth and fifth formulas follow from the third formula.
This completes the proof of Lemma \ref{lem:residue}.
\end{proof}

Recall $\CM_\pm(a)$ given in \eqref{dfn:DEM}, and then
\begin{equation*}
e^{-Aa}=(A-B_\pm)\CM_\pm(a)+e^{-B_\pm a} \quad (a\geq 0).
\end{equation*}
By this relation and Lemma \ref{lem:residue}, we have

\begin{lem}\label{lem:residue-2}
Let $\xi'=(\xi_1,\dots,\xi_{N-1})\in\BR^{N-1}\setminus\{0\}$ and $a\in\BR\setminus\{0\}$. 
Then 
\begin{align*}
\frac{1}{2\pi}\int_{-\infty}^\infty \frac{e^{i a \xi_N}}{|\xi|^2(\rho_\pm\lambda+\mu_\pm|\xi|^2)}\intd \xi_N
&=-\frac{\CM_\pm(|a|)}{2\mu_\pm A(A+B_\pm)}
+\frac{e^{-B_\pm|a|}}{2\mu_\pm AB_\pm(A+B_\pm)}, \\
\frac{1}{2\pi}\int_{-\infty}^\infty \frac{i\xi_N e^{i a \xi_N}}{|\xi|^2(\rho_\pm\lambda+\mu_\pm|\xi|^2)}\intd \xi_N
&= \frac{{\rm sign}(a)\cdot\CM_\pm(|a|)}{2\mu_\pm(A+B_\pm)}, \\
\frac{1}{2\pi}\int_{-\infty}^\infty \frac{(i\xi_N)^2 e^{i a \xi_N}}{|\xi|^2(\rho_\pm\lambda+\mu_\pm|\xi|^2)}\intd \xi_N
&=-\frac{A\CM_\pm(|a|)}{2\mu_\pm(A+B_\pm)} 
-\frac{e^{-B_\pm|a|}}{2\mu_\pm(A+B_\pm)}.
\end{align*}
\end{lem}

Together with the above formulas of $\wht \psi_{k\pm}$, $\wht \psi_{N\pm}$,
$\pd_N\wht \psi_{k\pm}$, and $\pd_N\wht \psi_{N\pm}$,
we obtain by Lemmas \ref{lem:residue} and \ref{lem:residue-2}
\begin{align}
\wht \psi_{k\pm}(\xi',0,\lambda) 
&=\sum_{\Fa\in\{+,-\}}\bigg\{\int_0^\infty \frac{e^{-B_\pm y_N}}{2\mu_\pm B_\pm}\wht F_k(\xi',\Fa y_N) \intd y_N  \notag \\
& +\sum_{j=1}^{N-1}\int_0^\infty \frac{\xi_k\xi_j\CM_\pm(y_N)}{2\mu_\pm A(A+B_\pm)} \wht F_j(\xi',\Fa y_N) \intd y_N \notag \\ 
&-\sum_{j=1}^{N-1}\int_0^\infty \frac{\xi_k\xi_j e^{-B_\pm y_N}}{2\mu_\pm AB_\pm(A+B_\pm )}\wht F_j(\xi',\Fa y_N) \intd y_N  \notag \\
& -\Fa\int_0^\infty \frac{i\xi_k\CM_\pm(y_N)}{2\mu_\pm(A+B_\pm)}\wht F_N(\xi',\Fa y_N) \intd y_N\bigg\}, \notag \\
\wht \psi_{N\pm}(\xi',0,\lambda)
&=\sum_{\Fa\in\{+,-\}}\bigg\{
-\int_0^\infty \frac{A\CM_\pm(y_N)}{2\mu_\pm(A+B_\pm)}\wht F_N(\xi',\Fa y_N) \intd y_N \notag \\
&+\int_0^\infty \frac{A e^{-B_\pm y_N}}{2\mu_\pm B_\pm(A+B_\pm)}\wht F_N(\xi',\Fa y_N) \intd y_N \notag \\
& -\Fa \sum_{j=1}^{N-1}\int_0^\infty \frac{i\xi_j\CM_\pm(y_N)}{2\mu_\pm(A+B_\pm)}\wht F_j(\xi',\Fa y_N) \intd y_N\bigg\}, \notag \\
\pd_N\wht\psi_{k\pm}(\xi',0,\lambda) 
&=\sum_{\Fa\in\{+,-\}}\bigg\{
\Fa\int_0^\infty \frac{e^{-B_\pm y_N}}{2\mu_\pm}\wht F_k(\xi',\Fa y_N) \intd y_N \notag \\
&+\Fa\sum_{j=1}^{N-1}\int_0^\infty 
\frac{\xi_j\xi_k \CM_\pm(y_N)}{2\mu_\pm(A+B_\pm)}\wht F_j(\xi',\Fa y_N) \intd y_N \notag \\
&-\int_0^\infty\frac{i\xi_k A \CM_\pm(y_N)}{2\mu_\pm(A+B_\pm)}\wht F_N(\xi',\Fa y_N) \intd y_N \notag \\
&-\int_0^\infty \frac{i\xi_k e^{-B_\pm y_N}}{2\mu_\pm(A+B_\pm)}\wht F_N(\xi',\Fa y_N) \intd y_N  \bigg\}, \notag \\
\pd_N\wht \psi_{N\pm}(\xi',0,\lambda) 
&=\sum_{\Fa\in\{+,-\}}\bigg\{
-\Fa\int_0^\infty  \frac{A^2\CM_\pm(y_N)}{2\mu_\pm(A+B_\pm)} \wht F_N(\xi',\Fa y_N) \intd y_N \notag \\
&-\sum_{j=1}^{N-1}\int_0^\infty \frac{i\xi_j A \CM_\pm(\Fa y_N)}{2\mu_\pm(A+B_\pm)}\wht F_j(\xi',\Fa y_N) \intd y_N \notag \\
&-\sum_{j=1}^{N-1}\int_0^\infty \frac{i\xi_j e^{-B_\pm y_N}}{2\mu_\pm(A+B_\pm)}\wht F_j(\xi',\Fa y_N) \intd y_N\bigg\}.\label{res-app}
\end{align}
This completes the proof of the appendix.


\end{document}